\documentclass[jsl,bibother,reqno]{asl}

\LogoOnfalse
\CopyRightfalse

\makeatletter
\newcounter{chapter}
\renewcommand{\thechapter}{\arabic{chapter}}
\newcommand{\chaptername}{Chapter}
\@addtoreset{section}{chapter}
\@addtoreset{equation}{chapter}
\@addtoreset{figure}{chapter}
\@addtoreset{table}{chapter}
\@addtoreset{footnote}{chapter}
\renewcommand{\thesection}{\thechapter.\arabic{section}}
\renewcommand{\theequation}{\thechapter.\arabic{equation}}
\renewcommand{\thefigure}{\thechapter.\arabic{figure}}
\renewcommand{\thetable}{\thechapter.\arabic{table}}

\let\jslchap@pending\@empty
\newcommand{\ChapterLabel}[1]{\gdef\jslchap@pending{#1}}
\newcommand{\jslchap@applylabel}{%
  \ifx\jslchap@pending\@empty\else
    \expandafter\label\expandafter{\jslchap@pending}%
    \global\let\jslchap@pending\@empty
  \fi}

\newcommand{\chapter}{\@ifstar{\jslchap@star}{\@dblarg\jslchap@numbered}}
\def\jslchap@numbered[#1]#2{%
  \@startsection{chapter}{0}{\z@}%
    {\linespacing\@plus\linespacing}{.5\linespacing}%
    {\normalfont\bfseries\raggedright}[#1]{#2}%
  \jslchap@applylabel}
\newcommand{\jslchap@star}[1]{%
  \@startsection{chapter}{0}{\z@}%
    {\linespacing\@plus\linespacing}{.5\linespacing}%
    {\normalfont\bfseries\raggedright}*{#1}%
  \ifx\jslchap@pending\@empty\else
    \PackageWarning{jsl-chapters}{A pending ChapterLabel was discarded
      at an unnumbered chapter}%
    \global\let\jslchap@pending\@empty
  \fi}

\renewcommand{\appendix}{%
  \par\setcounter{chapter}{0}\setcounter{section}{0}%
  \renewcommand{\thechapter}{\Alph{chapter}}%
  \renewcommand{\chaptername}{}}
\makeatother

\usepackage[monochrome]{xcolor}
\usepackage{graphicx}
\usepackage{mathrsfs}
\usepackage{enumitem}
\usepackage{tikz-cd}
\usepackage[all]{xy}
\usepackage{microtype}
\usepackage{makecell}
\usepackage{tabularx}
\usepackage{array}
\usepackage{url}
\theoremstyle{plain}
\newtheorem*{theorem*}{Theorem}
\newtheorem*{lemma*}{Lemma}
\newtheorem{theorem}{Theorem}[chapter]
\newtheorem{lemma}[theorem]{Lemma}
\newtheorem{claim}{Claim}[theorem]
\newtheorem*{claim*}{Claim}

\newtheorem{corollary}[theorem]{Corollary}

\theoremstyle{definition}
\newtheorem*{definition*}{Definition}
\newtheorem{definition}{Definition}
\theoremstyle{plain}

\title[Maximal Prikry sequences]{Maximal Prikry sequences}
\author{Ernest Schimmerling}
\revauthor{Schimmerling, Ernest}
\address{Department of Mathematical Sciences\\
Carnegie Mellon University\\
5000 Forbes Avenue\\
Pittsburgh, PA 15213, USA}
\email{eschimme@andrew.cmu.edu}

\author{Jiaming Zhang}
\revauthor{Zhang, Jiaming}
\address{Technische Universit\"at Wien \\
Wiedner Hauptstra\ss e 8--10/104\\
1040 Vienna, Austria}
\email{jiaming.zhang@tuwien.ac.at}
\renewcommand{\subjclassname}{\textup{2020} Mathematics Subject Classification}

\begin{document}
\begin{abstract}
We study maximal Prikry sequences over the Jensen--Steel core model $K$.
Assume that there is no inner model with a Woodin cardinal, and let
$\nu>\omega_2$ be regular in $K$ with $\mathrm{cf}(\nu)<|\nu|$ in $V$.
If $o^K(\nu)=1$, we prove that the unique normal measure on $\nu$ in $K$
has a maximal Prikry sequence $C$.
If, in addition, the measurable cardinals of $K$ below $\nu$ are bounded,
then every $A\subseteq\nu$ of cardinality less than $|\nu|$ is contained
in a set $B\in K[C]$ with $|B|^{K[C]}<\nu$.
For measures of higher Mitchell order, we obtain maximal Prikry sequences
from generating sequences under a closure assumption on $|\nu|$.
Under additional bounds on cardinality and Mitchell order, we obtain the same
conclusion without this closure assumption.
Forcing examples illustrate the limitations of maximality and covering.
The proofs build on the covering machinery of Mitchell and the first author.
The appendix gives corrections to the parts used here.
\end{abstract}
\maketitle

\ChapterLabel{Ch1}
\chapter{Introduction}

Covering theorems relate singularity in the ambient universe to the large cardinal structure of an inner model. When an ordinal is regular in the core model but singular in $V$, one can ask for more than the existence of a measure in the core model: can the change of cofinality be witnessed by a canonical Prikry sequence? We study this question for the Jensen--Steel core model $K$, under the assumption that there is no inner model with a Woodin cardinal. Our main results give maximal Prikry sequences and a local covering property for the resulting extensions of $K$.

Let $M$ be a transitive class model of ZFC, and suppose that $U$ is a normal measure on $\nu$ in $M$. A \emph{Prikry sequence for $U$ over $M$} is a cofinal set $C\subseteq\nu$ of order type $\omega$ such that, for every $A\in\mathcal P(\nu)\cap M$,
$$A\in U\quad\Longleftrightarrow\quad C\setminus A\text{ is finite}.$$
This is the Mathias characterization of genericity for the Prikry forcing $\mathbb P_U^M$. Such a sequence is \emph{maximal in $V$} if every Prikry sequence $D$ for $U$ over $M$ that belongs to $V$ satisfies $D\subseteq^* C$, namely, $D\setminus C$ is finite. Any two maximal Prikry sequences for the same measure have finite symmetric difference, and hence maximality implies that the sequence $C$ is unique modulo finite initial segments.

The covering theorem for $L[U]$ of Dodd and Jensen \cite{DJ,D} provides the starting point for this problem. Assume that $0^\dagger$ does not exist but that there is an inner model with a measurable cardinal, and choose $L[U]$ so that its measurable cardinal $\kappa$ is as small as possible. Then either $L[U]$ itself satisfies covering, and every singular cardinal in $V$ remains singular in $L[U]$; or there is a maximal Prikry sequence $C\subseteq\kappa$ for $U$ over $L[U]$ in $V$ such that $L[U][C]$ satisfies covering. In the second case, $C$ is unique modulo finite initial segments. Thus adjoining a single canonical Prikry sequence suffices to recover covering when it fails for $L[U]$. Our aim is to obtain analogous maximality and covering conclusions for the Jensen--Steel core model under a weaker anti-large-cardinal hypothesis.

Mitchell and the first author \cite{MS2} established the following consequence of covering under the weaker anti-large cardinal hypothesis that there is no inner model with a Woodin cardinal. Here and throughout, cardinalities and cofinalities without a superscript are computed in $V$.

\begin{theorem}\label{T1-1}
    Assume that there is no inner model with a Woodin cardinal, and let $K$ be the Jensen--Steel core model. Suppose that $\nu>\omega_2$, $\mathrm{cf}(\nu)<|\nu|$, and $\nu$ is a regular cardinal in $K$. Then $\nu$ is measurable in $K$. Moreover, if $\mathrm{cf}(\nu)>\omega$, then
    $$o^K(\nu)\geq\mathrm{cf}(\nu).$$
\end{theorem}

We use the covering machinery of \cite{MS2} to obtain a maximal Prikry sequence when $K$ has exactly one normal measure on $\nu$.

\begin{theorem}\label{T1-2}
    Assume that there is no inner model with a Woodin cardinal. Suppose that $\nu>\omega_2$, $\mathrm{cf}(\nu)<|\nu|$, $\nu$ is a regular cardinal in $K$, and $o^K(\nu)=1$. Let $U$ be the unique normal measure on $\nu$ in $K$. Then there is a maximal Prikry sequence $C$ for $U$ over $K$ in $V$.
\end{theorem}

Theorem~\ref{T1-1} already implies that $\mathrm{cf}(\nu)=\omega$ under these hypotheses. The content of Theorem~\ref{T1-2} is the existence of a sequence satisfying the Mathias condition and containing every other such sequence modulo a finite set. No cardinal arithmetic hypothesis on $|\nu|$ is required.

If the measurable cardinals of $K$ below $\nu$ are bounded, adjoining this sequence also gives the following local covering property.

\begin{theorem}\label{T1-3}
    Assume the hypotheses of Theorem~\ref{T1-2}, and suppose in addition that
    $$\sup\{\kappa<\nu\mid K\vDash\kappa\text{ is measurable}\}<\nu.$$
    Let $C$ be a maximal Prikry sequence for $U$ over $K$ in $V$. For every $A\subseteq\nu$ with $|A|<|\nu|$, there is a set $B\in K[C]$ such that
    $$A\subseteq B\qquad\text{and}\qquad |B|^{K[C]}<\nu.$$
\end{theorem}

The model $K[C]$ in this statement is independent of the choice of the maximal sequence. The distinction between the external bound $|A|<|\nu|$ and the internal bound $|B|^{K[C]}<\nu$ is relevant because $\nu$ need not be a cardinal of $V$.

Starting with a measurable cardinal below a Woodin cardinal, the first author uses the stationary tower to produce a Prikry sequence for a measure that has no maximal Prikry sequence. A tree-Prikry example of Tom Benhamou and the first author shows that maximality alone need not give the covering property in Theorem~\ref{T1-3} when measurable cardinals are unbounded below $\nu$. Chapter~\ref{Ch4} presents these constructions and their proofs using the stationary tower and a direct analysis of the tree-Prikry measure.

For measures of higher Mitchell order, the relevant starting point is Mitchell's notion of a generating sequence \cite{M}.
Throughout this paper, we use a notion of \emph{Mitchell order} that differs slightly from the usual one.
We say that a total extender $E\in\mathcal M$ has Mitchell order $\lambda$ if the order type of the set of total extenders $F$ on the $\mathcal M$-sequence with $\mathrm{crit}(F) = \mathrm{crit}(E)$ and $\mathrm{lh}(F)<\mathrm{lh}(E)$, ordered by length, is $\lambda$.
A normal measure $U$ on $\nu$ in $\mathcal M$ has Mitchell order $\lambda$ if the trivial completion of its associated total extender has Mitchell order $\lambda$ in $\mathcal M$.
By the extender maximality theorem of Steel and Schlutzenberg \cite[Theorem~1.2 and Remark~1.3]{Schl}, every normal measure belonging to $K$ occurs on its extender sequence, after taking the trivial completion of the associated extender.
Consequently, the two definitions of Mitchell order defined on $K$ coincide, and the normal measures on any fixed cardinal $\nu$ in $K$ are well-ordered by the Mitchell order, with order type $o^K(\nu)$. 
For $\lambda<o^K(\nu)$, we write $U(K,\nu,\lambda)$ for the unique normal measure on $\nu$ in $K$ of Mitchell order $\lambda$.

\begin{definition*}
    A cofinal set $C\subseteq\nu$ is a \emph{generating sequence} for $U(K,\nu,\lambda)$ if there is a function $g:C\to\nu$ such that $g(\alpha)\leq o^K(\alpha)$ for every $\alpha\in C$ and, for every $A\in\mathcal P(\nu)\cap K$, membership $A\in U(K,\nu,\lambda)$ is equivalent to the following condition holding for all sufficiently large $\alpha\in C$:
    \begin{itemize}
        \item either $g(\alpha)=o^K(\alpha)$ and $\alpha\in A$;
        \item or $g(\alpha)<o^K(\alpha)$ and $A\cap\alpha\in U(K,\alpha,g(\alpha))$.
    \end{itemize}
\end{definition*}

For example, a Prikry sequence for $U(K,\nu,\lambda)$ is a generating sequence witnessed by $g(\alpha)=o^K(\alpha)$. Generating sequences allow measure-one sets to be detected through measures below $\nu$, without requiring the sequence itself to be Prikry generic. They occur in Mitchell's analysis of definable witnesses to singularity \cite{M}.

Let $\beta$ be the least ordinal such that either $\beta=o^K(\nu)$ or $U(K,\nu,\beta)$ has no generating sequence in $V$. The second author proved that, when $\beta$ is a successor and the closure assumption below holds, the preceding measure has a maximal Prikry sequence. A cardinal $\chi$ is called $\theta$-\emph{closed} if $\alpha^\theta<\chi$ for every cardinal $\alpha<\chi$.

\begin{theorem}\label{T1-4}
    Assume that there is no inner model with a Woodin cardinal. Suppose that $\nu>\omega_2$, $\mathrm{cf}(\nu)<|\nu|$, $\nu$ is a regular cardinal in $K$, and $|\nu|$ is $\mathrm{cf}(\nu)$-closed. Define $\beta$ as above. If $\beta=\lambda+1$, then $\mathrm{cf}(\nu)=\omega$ and there is a maximal Prikry sequence for $U(K,\nu,\lambda)$ over $K$ in $V$.
\end{theorem}

The proofs analyze the critical sequences obtained by comparing a soundness witness for $K$ with its collapses to elementary substructures. With closure, a suitable critical sequence belongs to the substructure, so elementarity can be used to establish maximality. In the unique-measure case, comparison along a weakly internal approachable chain replaces the closure assumption. For higher Mitchell order, generating sequences identify the relevant part of the critical sequence, and the successor condition on $\beta$ rules out unboundedly many limit points.

The paper is organized as follows. Chapter~\ref{Ch2} reviews the covering machinery and fixes the notation throughout the paper. Chapter~\ref{Ch3} proves Theorems \ref{T1-2} and \ref{T1-3} under the additional hypothesis of $|\nu|^V$ being countably closed. Chapter~\ref{Ch4} presents the forcing examples and their proofs, including a separation lemma that establishes maximality in the tree-Prikry extension. Chapter~\ref{Ch5} removes the closure hypothesis from the unique-measure results. Chapter~\ref{Ch6} proves Theorem~\ref{T1-4}. Chapter~\ref{Ch7} proves a further extension, Theorem~\ref{T7-1}, which removes the closure assumption on $|\nu|$ under stronger hypotheses and is not used in the proofs of Theorems~\ref{T1-2}--\ref{T1-4}. The Appendix gives corrections and additional arguments for the covering machinery from \cite{MS2}.
\ChapterLabel{Ch2}
\chapter{The Covering Mechanism}
The exposition in this chapter is due to the first author. We review the covering machinery
developed by Mitchell and the first author in \cite{MS2} and fix the notation used throughout the paper.
Readers who are familiar with this material may skip to Chapter~\ref{Ch3}.

\section{Preliminaries}
We assume that there is a measurable cardinal $\Omega>\nu$ and let  $U_\Omega$ be a normal measure over $\Omega$.
Let $K$ be the core model defined over $(H_{\Omega^+},\in,U_{\Omega})$ equipped with the Mitchell-Steel indexing.
We refer the reader to \cite{MSt} and \cite{St} for standard terminology.
In fact, Theorem \ref{T1-1} holds true under both Mitchell-Steel indexing and Jensen indexing, but in this article, we will focus solely on the Mitchell-Steel indexing.
Fix some regular cardinal $\Omega_0<\Omega$ such that $\Omega_0>|\nu|^{++}$.
Let $W$ be the model obtained by taking ultrapowers by order zero measures at each measurable cardinal of $K$ greater than $\Omega_0$. 
Following \cite{St}, $W$ is an $A_0$-soundness witness of $K\upharpoonright\Omega_0$.

The machinery for proving covering theorems relies on the analysis of suitable elementary substructures. 
Let 
$$X\prec(H_{\Omega^+},\in,U_{\Omega})$$
be an arbitrary elementary substructure such that $\aleph_1\leq|X|<|\nu|$, $\{\nu,\Omega_0,W\}\subseteq X$ and $\sup(\nu\cap X) = \nu$.
As we proceed, we will identify two additional requirements on $X$.
Let $\pi_X:N_X\to X$ be the inverse of the Mostowski collapse of $X$, $\delta_X = \mathrm{crit}(\pi_X)$, $W_X = \pi_X^{-1}(W)$, and $\mu_X = \pi^{-1}_X(\nu)$. 
Let $(\mathcal T_X,\mathcal U_X)$ be the coiteration of $(W,W_X)$. 
As $W$ is universal, there are no drops of any kind along the main branch of $\mathcal U_X$, and 
$$\mathcal M^{\mathcal U_X}_{\infty}\lhd \mathcal M^{\mathcal T_X}_{\infty}.$$
When the substructure is fixed or implied by the background assumption, we suppress the subscript $X$ and write $\overline W$ for $W_X$.
Let us fix $X$ temporarily, and assume $X$ satisfies the following additional requirement:

\subsubsection*{Requirement 1}
Let $\nu<\Omega_1<\Omega$ with $\Omega_1\in X$.
For every $A_0$-soundness witness $W_1\in X$ for
$K\upharpoonright\Omega_1$, put
$$\overline\Omega_1=\pi_X^{-1}(\Omega_1),\qquad
  \overline W_1=\pi_X^{-1}(W_1).$$
We require
$$(\delta_X^+)^{\overline W_1}<(\delta_X^+)^{W_1}.$$
We also require that the comparison of $W_1$ with $\overline W_1$
uses no extender of length $<\overline\Omega_1$ on the
$\overline W_1$ side.

Taking $\Omega_1=\Omega_0$ gives this condition for $W$ and $W_X$.
Put $\overline\Omega_0=\pi_X^{-1}(\Omega_0)$.
The first correction in the Appendix explains how to use this local
form of Requirement~1 in the covering argument.

Throughout the following definitions and lemmas, let $\alpha$ range
over the $\overline W$-cardinals below $\overline\Omega_0$.
Let $\eta(\alpha)$ be the least $\eta$, such that $\eta<\mathrm{lh}(\mathcal T)$ and $\mathcal M^{\mathcal T}_\eta$ agrees with $\overline W$ below $\alpha$.
Suppose $\mathcal M^{\mathcal T}_{\eta(\alpha)}$ is a set mouse (a mouse $\mathcal M$ with $\mathcal M\cap\mathrm{Ord}<\Omega$). 
This implies that there is a drop along the branch $(0,\eta(\alpha)]_{\mathcal T}$. 
In this case, let $n = \mathrm{deg}^{\mathcal T}(\eta(\alpha))$. Then $\mathcal M^{\mathcal T}_{\eta(\alpha)}$ is $(n+1)$-\emph{sound} above $\alpha$, namely, it is $n$-sound and 
$$\mathcal M^{\mathcal T}_{\eta(\alpha)} = \mathrm{Hull}^{\mathcal M^{\mathcal T}_{\eta(\alpha)}}_{n+1}(\alpha\cup p_{n+1}(\mathcal M^{\mathcal T}_{\eta(\alpha)})).$$
Alternatively, if $\mathcal M^{\mathcal T}_{\eta(\alpha)}$ is a weasel (a mouse $\mathcal M$ with $\mathcal M\cap\mathrm{Ord}=\Omega$), it has the hull and definability property at every ordinal in $[\alpha,\Omega_0)$. 
As a consequence, for every thick class $\Gamma$,
$$\Omega_0\subseteq \mathrm{Hull}^{\mathcal M^{\mathcal T}_{\eta(\alpha)}}(\alpha\cup\Gamma).$$
If $\alpha=(\kappa^+)^{\overline W}$, the corresponding soundness or
definability property already holds above $\kappa$ \cite[Section~2.1]{MS2}.
\subsection*{The $\mathcal P$-structures.} 
Let $\mathcal P_\alpha\unlhd \mathcal M^{\mathcal T}_{\eta(\alpha)}$ be the shortest initial segment $\mathcal P$, if there exists one, such that
$$\rho_{m+1}(\mathcal P)< \alpha\leq\rho_{m}(\mathcal P)$$
for some $m<\omega$, and put $m(\alpha)=m$.
For a type III premouse these projecta are computed in its squash
$\mathcal P^{\mathrm{sq}}$, with
$\rho_0(\mathcal P)=\mathcal P^{\mathrm{sq}}\cap\mathrm{Ord}$.
If such $\mathcal P$ does not exist, we let $\mathcal P_\alpha = \mathcal M^{\mathcal T}_{\eta(\alpha)}$ and $m(\alpha) = \mathrm{deg}^{\mathcal T}(\eta(\alpha))$.
Notice that if $\mathcal P_\alpha\neq\mathcal M^{\mathcal T}_{\eta(\alpha)}$, then $\alpha$ is not a cardinal in $\mathcal M^{\mathcal T}_{\eta(\alpha)}$.

\subsection*{The $\mathcal Q$-structures.} We define a sequence of mice $\langle \mathcal Q_{\alpha}\mid \alpha\in \mathrm{Card}^{\overline W}\cap\overline\Omega_0\rangle$ by recursion. 
If $\alpha$ does not fall into the protomouse case, we define $\mathcal Q_\alpha = \mathcal P_\alpha$ and $n(\alpha) = m(\alpha)$.
The protomouse case means that
\begin{itemize}
    \item $\mathcal P_\alpha$ is a type II active mouse;
    \item $m(\alpha) = 0$;
    \item $\kappa = \mathrm{crit}(\dot F^{\mathcal P_{\alpha}})$;
    \item $\lambda = (\kappa^+)^{\mathcal P_\alpha}<\alpha$;
    \item $\sup(\pi[\lambda])<\pi(\lambda)$.
\end{itemize}
In this case, let $n(\alpha) = n(\lambda)$ and $\mathcal Q_{\alpha} = \mathrm{Ult}_{n(\alpha)}(\mathcal Q_\lambda,\dot F^{\mathcal P_{\alpha}})$.
The wellfoundedness of $\mathcal Q_\alpha$ follows from the enhanced iterability of $K^c$ used in \cite[Section~2.2]{MS2}.

\subsection*{The $\mathcal S$-structures.} We let
$$\Pi_\alpha:\mathcal Q_\alpha\to \mathrm{Ult}_{n(\alpha)}(\mathcal Q_\alpha,\pi,\sup(\pi[\alpha])) = \mathcal S_\alpha,$$
which is the ultrapower embedding from $\mathcal Q_\alpha$ to the $n(\alpha)$-ultrapower given by the extender derived from $\pi$ with length $\sup(\pi[\alpha])$. 
Our last requirement for $X$ is related to the $\mathcal S$-structures.

\subsubsection*{Requirement 2.} For every $\overline W$-cardinal $\alpha<\overline\Omega_0$, 
\begin{itemize}
    \item $\mathcal S_\alpha$ is wellfounded and iterable.
    \item The phalanx $((W,\mathcal S_\alpha),\sup(\pi[\alpha]))$ is iterable.
\end{itemize}

In \cite{MS2} and this article, we only look at structures $X$ which satisfy Requirements 1 and 2. 
For the next two theorems we use \cite[Theorems~2--4]{MS2} with Requirement~1 in the form given above.

\begin{theorem}\label{T2-1}
    Let $X\prec (H_{\Omega^+},\in,U_\Omega)$ with ${}^\omega X\subseteq X$.
    Suppose that $|X|<|\nu|$, $\{\nu,\Omega_0,W\}\subseteq X$, and $\sup(\nu\cap X)=\nu$.
    Then $X$ satisfies Requirements 1 and 2.
\end{theorem}

\begin{theorem}\label{T2-2}
    Let $\varepsilon$ be a regular cardinal such that
    $$\max(\omega_2,\mathrm{cf}(\nu)^+)\leq\varepsilon<\nu.$$
    Suppose $\vec Y = \langle Y_i\mid i<\varepsilon\rangle$ is a continuous chain of elementary substructures that is internally approachable, which means for each $i<\varepsilon$, $Y_i\prec(H_{\Omega^+},\in,U_\Omega)$ and $\langle Y_k\mid k\leq i\rangle\in Y_{i+1}$.
    Moreover, for each $i<\varepsilon$, $|Y_i|<|\nu|$, $\{\nu,\Omega_0,W\}\subseteq Y_i$ and $\sup(\nu\cap Y_i) = \nu$. 
    Suppose $Y_i\cap \varepsilon\in \varepsilon$ and $|Y_i| = |Y_i\cap \varepsilon|$. 
    Then there exists a club subset $\mathcal C\subseteq \varepsilon$ such that, for each $j\in \mathcal C$ with $\mathrm{cf}(j)>\omega$, $Y_j$ satisfies Requirements 1 and 2.
\end{theorem}

We can thus fix a stationary subset $\mathcal F$ of $\mathcal P(H_{\Omega^+})$ such that each $X\in\mathcal F$ satisfies Requirements 1 and 2. 
Let $\mathcal F$ be the collection of all $X$ such that there exist $\varepsilon$, $\vec  Y$ and $\mathcal C$ defined as in Theorem \ref{T2-2}, and $X = Y_j$ where $j\in \mathcal C$, $j = \varepsilon\cap Y_j = \mathrm{type}(j\cap \mathcal C)$, and $\mathrm{cf}(j) = \max(\omega_1,\mathrm{cf}(\nu))$.
The last requirement of $j$ is stronger than that of Theorem \ref{T2-2}.
This is to ensure that $X\cap \mathrm{Ord}$ is $<\mathrm{cf}(\nu)$-closed.
Also, notice that if $|\nu|$ is countably closed, that is, 
$$\forall \kappa<|\nu|(\kappa^\omega<|\nu|),$$
then there exists some $X\in\mathcal F$ such that $^\omega X\subseteq X$.
In the rest of this article, we will only look at substructures which are members of $\mathcal F$.

\section{Analysis of $\mathcal Q$ and $\mathcal S$-structures}
In the following section, we present some crucial analysis of the $\mathcal Q$ and $\mathcal S$ structures.
These analyses played a central role in \cite{MS2}, and are also heavily used in this article.

\subsection*{The $\mathcal Q$-structures}
Fix some $\overline W$-cardinal $\alpha$ such that $\mathcal Q_\alpha\neq\mathcal P_\alpha$. 
By the wellfoundedness of ordinals, we can find the following finite sequence of embeddings
$$\mathcal P_{\lambda_\ell} = \mathcal Q_{\lambda_{\ell}}\xrightarrow{\dot F^{\mathcal P_{\lambda_{\ell-1}}}}\mathcal Q_{\lambda_{\ell-1}}\xrightarrow{\dot F^{\mathcal P_{\lambda_{\ell-2}}}}\cdots \xrightarrow{\dot F^{\mathcal P_{\lambda_{0}}}}\mathcal Q_{\lambda_0} = \mathcal Q_{\alpha},$$
where $1\leq\ell<\omega$, $\lambda_0=\alpha$ and
$\lambda_{i+1}=(\kappa_{i+1}^+)^{\overline W}$ for $i<\ell$.
The sequence is unique and is referred to as the decomposition of $\mathcal Q_\alpha$.
For each $1\leq i\leq \ell$, let $\phi_i:\mathcal Q_{\lambda_i}\to \mathcal Q_{\alpha}$ be the composition of ultrapower embeddings. 
Let $s(\dot F^{\mathcal P_{\lambda_i}})$ be the Dodd parameter of $\dot F^{\mathcal P_{\lambda_i}}$.
See \cite[Section 2]{MS2} for the definition and basic facts about Dodd solidity and Dodd squashes.
Let $\kappa_{i+1} = \mathrm{crit}(\dot F^{\mathcal P_{\lambda_i}})$ and
$$t_\alpha = (s(\dot F^{\mathcal P_{\alpha}})\setminus \alpha)\cup \bigcup_{1\leq i<\ell}\phi_i(s(\dot F^{\mathcal P_{\lambda_i}})\setminus\kappa_i).$$
When $\mathcal Q_\alpha$ is a weasel, $t_\alpha = c(\mathcal Q_\alpha)\setminus \alpha$.
The class parameters of $\mathcal Q_\alpha$ and $\mathcal S_\alpha$
are computed relative to $\Omega_0$, as in \cite[Sections~2.2--2.3]{MS2}.
We have the following propositions:
\begin{itemize}
    \item If $\mathcal Q_\alpha$ is a set mouse, then $\mathcal Q_\alpha$ is $n(\alpha)+1$-sound relative to the parameter
    $$q_{\alpha}=(p_{n(\alpha)+1}(\mathcal Q_\alpha)\setminus\alpha)\cup t_\alpha.$$
    The part coming from the terminal model of the decomposition is
    $$p_{n(\alpha)+1}(\mathcal Q_\alpha)\setminus\alpha
      =\phi_\ell\bigl(p_{n(\alpha)+1}(\mathcal P_{\lambda_\ell})
                         \setminus\kappa_\ell\bigr)$$
    by \cite[Section~2.2]{MS2}. Relative soundness gives
    $$\mathcal Q_\alpha = \mathrm{Hull}^{\mathcal Q_\alpha}_{n(\alpha)+1}(\alpha\cup q_\alpha).$$
    Moreover, $q_{\alpha}$ is the least parameter $r$ such that 
    $$\mathcal P(\alpha)\cap \mathcal Q_\alpha\subseteq \mathrm{cHull}^{\mathcal Q_\alpha}_{n(\alpha)+1}(\alpha\cup r).$$
    Here, $\mathrm{cHull}$ indicates the Mostowski collapse of the hull.
    \item If $\mathcal Q_\alpha$ is a weasel, then $\mathcal Q_\alpha$ has the $t_\alpha$-definability property at every ordinal in the interval $[\alpha,\Omega_0)$, which means for every thick class $\Gamma$,
    $$\Omega_0\subseteq \mathrm{Hull}^{\mathcal Q_\alpha}(\alpha\cup t_\alpha\cup\Gamma).$$
    Moreover, $t_{\alpha}$ is the least parameter $r$ such that $\mathcal Q_\alpha$ has the $r$-hull property at $\alpha$. 
    In other words, for every thick class $\Gamma$, 
    $$\mathcal P(\alpha)\cap \mathcal Q_\alpha\subseteq \mathrm{cHull}^{\mathcal Q_\alpha}(\alpha\cup r\cup\Gamma).$$
\end{itemize}
To accommodate the case when $\mathcal Q_\alpha = \mathcal P_\alpha$, we let $q_\alpha = p_{n(\alpha)+1}(\mathcal P_\alpha)\setminus \alpha$ and $t_\alpha = \emptyset$.
The following comparison between $\mathcal P_\alpha$ and $\mathcal Q_\alpha$ is \cite[Lemma~2.1]{MS2}.
\begin{lemma}\label{L2-3}
    Let $\alpha<\overline\Omega_0$ be a $\overline W$-cardinal. 
    Suppose $\mathcal P_\alpha\neq \mathcal Q_\alpha$. 
    Take a large enough ordinal $\xi<\alpha$ such that
    \begin{itemize}
        \item $(\kappa^+)^{\mathcal P_\alpha}\leq \xi$, where $\kappa = \mathrm{crit}(\dot F^{\mathcal P_\alpha})$;
        \item $p_1(\mathcal P_\alpha)\cap \alpha\subseteq \xi$;
        \item $s(\dot F^{\mathcal P_\alpha})\cap \alpha\subseteq \xi$;
        \item $p_1(\mathcal P_\alpha)$ is $\Sigma_1$-definable from parameters in $\xi\cup (s(\dot F^{\mathcal P_\alpha})\setminus \alpha)$;
        \item $s(\dot F^{\mathcal P_\alpha})$ is $\Sigma_1$-definable from parameters in $\xi\cup (p_1(\mathcal P_\alpha)\setminus \alpha)$.
    \end{itemize}
    Then one of the following holds:
    \begin{itemize}
        \item $\mathcal Q_\alpha$ is a set mouse, and
        $$\alpha\cap \mathrm{Hull}^{\mathcal P_\alpha}_{1}(\xi\cup p_{1}(\mathcal P_\alpha)) = \alpha\cap \mathrm{Hull}^{\mathcal Q_\alpha}_{n(\alpha)+1}(\xi\cup q_\alpha).$$
        \item $\mathcal Q_\alpha$ is a weasel. For every thick class $\Gamma$ of ordinals fixed by the embeddings 
        $$W\xrightarrow{i^{\mathcal T}_{0,\eta(\lambda_\ell)}}\mathcal P_{\lambda_\ell}\xrightarrow{\phi_{\ell}}\mathcal Q_{\alpha},$$
        we have
        $$\alpha\cap \mathrm{Hull}^{\mathcal P_\alpha}_1(\xi\cup p_{1}(\mathcal P_\alpha)) = \alpha\cap \mathrm{Hull}^{\mathcal Q_\alpha}(\xi\cup t_\alpha\cup \Gamma).$$
    \end{itemize}
\end{lemma}

\subsection*{The $\mathcal S$-structures} 
Since $\mathcal S_\alpha$ is the ultrapower of $\mathcal Q_\alpha$, we have the following observation:
\begin{itemize}
    \item If $\mathcal S_\alpha$ is a set mouse, then $\mathcal S_\alpha$ is $n(\alpha)+1$-sound above $\sup(\Pi_\alpha[\alpha])$ relative to the parameter $\Pi_\alpha(q_\alpha)$. That is,
    $$\mathcal S_\alpha = \mathrm{Hull}^{\mathcal S_\alpha}_{n(\alpha)+1}(\sup(\Pi_\alpha[\alpha])\cup \Pi_\alpha(q_\alpha)).$$
    \item If $\mathcal S_\alpha$ is a weasel, then it has the $\Pi_\alpha(t_\alpha)$-definability property throughout $[\sup(\Pi_\alpha[\alpha]),\Omega_0)$. Thus for every thick class $\Gamma$,
    $$\Omega_0\subseteq \mathrm{Hull}^{\mathcal S_\alpha}(\sup(\Pi_\alpha[\alpha])\cup \Pi_\alpha(t_\alpha)\cup \Gamma).$$
    More precisely, $c(\mathcal S_\alpha)\setminus\sup(\pi[\alpha])
    =\Pi_\alpha(t_\alpha)$ \cite[Section~2.3]{MS2}. The remaining
    coordinates of $c(\mathcal S_\alpha)$ already belong to the base
    of the displayed hull.
\end{itemize}
The next three lemmas compare $\mathcal S_\alpha$ with $W$, following \cite[Lemmas~2.2--2.4]{MS2}. 
In the set-mouse case we include the additional alternative described below.

\begin{lemma}\label{L2-4}
    Put $\Lambda_\alpha=\sup(\pi[\alpha])$. If $\mathcal S_\alpha$ is
    a set mouse, then either $\mathcal S_\alpha\lhd W$ or
    $E^W_{\Lambda_\alpha}\neq\emptyset$ and
    $\mathcal S_\alpha=\mathrm{Ult}(W^*,E^W_{\Lambda_\alpha})$,
    where $W^*$ is the longest initial segment of $W$ to which
    $E^W_{\Lambda_\alpha}$ applies.
\end{lemma}

The proof of \cite[Lemma 2.2]{MS2} is largely the same as that of \cite[Lemma 3.10]{MSS}. 
However, the proof does not rule out the possibility that $\mathcal S_\alpha$ can be one ultrapower away from an initial segment of $W$.
Therefore we add the second case in Lemma \ref{L2-4} to cover this possibility.
It is worth noticing that, if $\sup(\Pi_\alpha[\alpha])$ is a limit cardinal in $W$, then $\mathcal S_\alpha\lhd W$. 
This happens when $\alpha = \mu$, and thus $\sup(\Pi_\mu[\mu]) = \nu$.

\begin{lemma}\label{L2-5}
    Suppose $\mathcal S_\alpha$ is a weasel, and $\mathcal P_\alpha = \mathcal Q_\alpha$. 
    Assume that $\alpha>\delta = \mathrm{crit}(\pi)$.
    We have the following chain of embeddings:
    $$W\xrightarrow{i^{\mathcal T}_{0,\eta(\alpha)}}\mathcal M^{\mathcal T}_{\eta(\alpha)} = \mathcal P_\alpha = \mathcal Q_\alpha\xrightarrow{\Pi_\alpha}\mathcal S_\alpha.$$
    Let the composition be $j:W\to \mathcal S_\alpha$. 
    Writing $\kappa=\mathrm{crit}(j)$, we have $\kappa<\delta$ and $j(\kappa)>\sup(\pi[\alpha])$.
    Let $G$ be the extender derived from $j$ with length $\sup(\pi[\alpha])$. 
    Then $G\in W$ is a short extender, $\mathcal S_\alpha=\mathrm{Ult}(W,G)$,
    and $j$ is the ultrapower map.
\end{lemma}

\begin{lemma}\label{L2-6}
    Suppose $\mathcal S_\alpha$ is a weasel, and $\mathcal P_\alpha \neq \mathcal Q_\alpha$. 
    Assume that $\alpha>\delta = \mathrm{crit}(\pi)$.
    We have the following chain of embeddings:
    $$W\xrightarrow{i^{\mathcal T}_{0,\eta(\lambda_{\ell})}}\mathcal M^{\mathcal T}_{\eta(\lambda_{\ell})} = \mathcal P_{\lambda_{\ell}} = \mathcal Q_{\lambda_{\ell}}\xrightarrow{\phi_\ell}\mathcal Q_{\alpha}\xrightarrow{\Pi_\alpha}\mathcal S_\alpha.$$
    Let the composition be $j:W\to \mathcal S_\alpha$. 
    Writing $\kappa=\mathrm{crit}(j)$, we have $\kappa<\delta$, $j(\kappa)>\sup(\pi[\alpha])$ and $\Pi_\alpha(t_\alpha)\subseteq j(\kappa)$.
    Let $G$ be the extender derived from $j$ with support $\sup(\pi[\alpha])\cup\Pi_\alpha(t_\alpha)$. 
    Then $G\in W$ is a short extender, $\mathcal S_\alpha=\mathrm{Ult}(W,G)$,
    and $j$ is the ultrapower map.
\end{lemma}

The definability property used in \cite[Lemmas~2.3 and~2.4]{MS2} applies
only below $\Omega_0$. For the full ultrapower identification, one also uses
$\mathcal Q_\alpha=\mathrm{Hull}^{\mathcal Q_\alpha}(h[W]\cup\Omega_0)$,
where $h\colon W\to\mathcal Q_\alpha$ is the part of $j$ preceding $\Pi_\alpha$.
For this equality it suffices to check that the support ordinals used to form
$h$ have images below $\Omega_0$ in $\mathcal Q_\alpha$.

\subsection*{The $\mathcal S_{XY}$-structure} 
We now analyze the correspondence between two different structures in $\mathcal F$ and will no longer suppress the subscript $X$.
Where the subscript $X$ clashes with other indices, it will be moved to the superscript.
Let $X, Y\in\mathcal F$ such that $X\subseteq Y$.
We write $\pi_{XY} = \pi_Y^{-1}\circ\pi_{X}$.
For every $W_X$-cardinal $\alpha<\pi_X^{-1}(\Omega_0)$, let
$$\Pi^{XY}_\alpha:\mathcal Q^X_\alpha\longrightarrow
  \mathcal S^{XY}_\alpha=
  \mathrm{Ult}_{n_X(\alpha)}
     (\mathcal Q^X_\alpha,\pi_{XY},\sup(\pi_{XY}[\alpha])).$$
In the weasel case this is the ordinary ultrapower.
The factor map $\Phi^{XY}_\alpha:\mathcal S^{XY}_\alpha\to\mathcal S^X_\alpha$
takes a representative with support $a$ and source function $f$
to the representative with support $\pi_Y(a)$ and the same source
function. This is well defined by
$\pi_Y\circ\pi_{XY}=\pi_X$, and
$\Phi^{XY}_\alpha\circ\Pi^{XY}_\alpha=\Pi^X_\alpha$.
In the weasel case the factor therefore takes $\Pi^{XY}_\alpha(t^X_\alpha)$
to $\Pi^X_\alpha(t^X_\alpha)$.

We now impose further assumptions on $X$ and $Y$. 
To do this, we shrink $\mathcal F$ to a stationary subset of $\mathcal P(H_{\Omega^+})$.
Let $\mu_X = \pi^{-1}_X(\nu)$, $\theta_X = \eta(\mu_X)$. 
Let $\mathcal G\subseteq \mathcal F$ be a stationary subset such that the following questions have the same answers for every $X\in \mathcal G$:

\noindent \textbf{(List A)}
\begin{itemize}[label=\textbullet, leftmargin=2em, nosep, topsep=3pt]
    \item Is $\mathcal{M}_{\theta_X}^{\mathcal T_X}$ a set mouse or a weasel?
    \item What is the value of $m_X(\mu_X)$?
    \item Is $\mathcal{Q}_{\mu_X}^X$ defined by the protomouse case?
    \item Is $\mathcal{Q}_{\mu_X}^X$ a set mouse or a weasel?
    \item What is the value of $n_X(\mu_X)$?
\end{itemize}
We write $m = m_X(\mu_X)$ and $n = n_X(\mu_X)$ for the common values.
Let $\mathcal M_X$ denote $\mathcal{M}_{\theta_X}^{\mathcal T_X}$, $\mathcal Q_X$ for $\mathcal{Q}_{\mu_X}^X$ and $\mathcal S_X$ for $\mathcal{S}_{\mu_X}^X$.
For each pair $X\subseteq Y$ in $\mathcal G$, write
$\mathcal S_{XY} = \mathcal S^{XY}_{\mu_X}$,
$\Pi_{XY}=\Pi^{XY}_{\mu_X}$ and $\Phi_{XY}=\Phi^{XY}_{\mu_X}$.
Here $\sup(\pi_{XY}[\mu_X])=\mu_Y$, since $X\cap\nu$ is cofinal in
$\nu$. Thus $\Phi_{XY}\upharpoonright\mu_Y=\pi_Y\upharpoonright\mu_Y$,
as in \cite[p.~20]{MS2}.

The following iterability result is \cite[Lemma~4.1]{MS2}.
\begin{lemma}\label{L2-7}
    Let $\overrightarrow {\mathcal P_Y}^\frown\mathcal S_{XY}$ denote the phalanx
    $$( \langle \mathcal{P}_{(\aleph_{\alpha+1})^{W_Y}}^Y \mid \alpha < \beta \rangle^\frown \langle \mathcal{S}_{XY} \rangle, \langle (\aleph_{\alpha+1})^{W_Y} \mid \alpha < \beta \rangle^\frown \langle \mu_Y \rangle ),$$
    where $(\aleph_\beta)^{W_Y} = \mu_Y$.
    Then $\overrightarrow {\mathcal P_Y}^\frown\mathcal S_{XY}$ is iterable.
\end{lemma}

Given the iterability of $\overrightarrow {\mathcal P_Y}^\frown\mathcal S_{XY}$, we compare the phalanx and 
$$( \langle \mathcal{P}_{(\aleph_{\alpha+1})^{W_Y}}^Y \mid \alpha < \beta \rangle^\frown \langle \mathcal M_Y\rangle,
    \langle (\aleph_{\alpha+1})^{W_Y} \mid \alpha < \beta \rangle^\frown \langle \mu_Y\rangle).$$
In these phalanxes the earlier components are indexed by the successor
$W_Y$-cardinals below the limit cardinal $\mu_Y$, as in
\cite[pp.~20--21]{MS2}. The last component on this side is
$\mathcal P^Y_{\mu_Y}=\mathcal M_Y$.
Let the coiteration trees be $\mathcal T_{XY}$ and $\mathcal U_{XY}$, respectively.
The following lemmas are the direct consequences of this comparison argument under different cases for $\mathcal Q_X$.
These lemmas also serve as corrections to \cite[Lemma 4.2]{MS2}.
We will record the proofs in the Appendix of this article.

\begin{lemma}\label{L2-8}
Suppose $\mathcal{S}_{XY}$ is a set mouse.
Then $\mathcal T_{XY}$ is trivial and, either
$$\mathcal{S}_{XY} \unlhd \mathcal{M}_Y,$$
or there is $\lambda < \mu_Y$ and an extender $E$ on the $\mathcal{M}_Y$ sequence such that $\mathrm{crit}(E) = \kappa$, $\lambda = (\kappa^+)^{W_Y}$ and
$$\mathcal{S}_{XY} = \mathrm{Ult}_n( \mathcal{P}^Y_\lambda , E).$$
\end{lemma}

\begin{lemma}\label{L2-9}
Suppose $\mathcal S_{XY}$ is a weasel. Put $\tau=\pi_Y^{-1}(\Omega_0)$.
Both $\mathcal T_{XY}$ and $\mathcal U_{XY}$ have successor length
and the same last model.
Neither final branch drops, and $\operatorname{root}^{\mathcal T_{XY}}=\mathcal S_{XY}$.
If $\operatorname{root}^{\mathcal U_{XY}}=\mathcal M_Y$, then $\mathcal M_Y$ is a weasel and
$$
  \mathcal M_Y\upharpoonright\tau=\mathcal S_{XY}\upharpoonright\tau.
$$
If $\operatorname{root}^{\mathcal U_{XY}}=\mathcal P^Y_\lambda$ for some
$\lambda<\mu_Y$, then $\mathcal P^Y_\lambda$ is a weasel and there is an extender
$G$ such that
$$
  \mathrm{Ult}(\mathcal P^Y_\lambda,G)\upharpoonright\tau
    =\mathcal S_{XY}\upharpoonright\tau.
$$
\end{lemma}

\section{Proof of the Covering Theorem} 
The following facts about the main branch follow from \cite[Lemmas~3.3--3.6]{MS2}.
\begin{lemma}\label{L2-10}
    Fix any $X\in \mathcal F$.
    Then $\theta_X$ is a limit ordinal, and $\mathcal P_X = \mathcal M_X$.
    Fix $\zeta_X<_{\mathcal T_X}\theta_X$ such that $\mu_X\in\operatorname{ran}(i^{\mathcal T_X}_{\zeta_X,\theta_X})$ and there is no drop on $[\zeta_X,\theta_X)_{\mathcal T_X}$.
    For each $\eta\in[\zeta_X,\theta_X)_{\mathcal T_X}$ put
    $\mu_\eta^X=\mathrm{crit}(i^{\mathcal T_X}_{\eta,\theta_X})$.
    Let $I_X$ consist of those $\eta$ for which $i^{\mathcal T_X}_{\eta,\theta_X}(\mu_\eta^X)=\mu_X$.
    Then $I_X$ is a club subset of $[\zeta_X,\theta_X)_{\mathcal T_X}$.
\end{lemma}

Let $\overline C_X = \{\mu^X_\eta\mid \eta\in I_X\}$ and $C_X = \pi_X[\overline C_X]$. 
$C_X$ will be called the critical sequence of $X$.
As in \cite[Section~3]{MS2}, $\mu_X$ is a measurable cardinal of $\mathcal M_X$.
The analysis in \cite[Section~3.1]{MS2} gives the following criterion for membership
in the order-zero measure $U(\mathcal M_X,\mu_X,0)$.
For any $\eta\in I_X$, let $\eta^*+1 = \min(\eta,\theta_X)_{\mathcal T_X}$.
We say that an extender $E$ is equivalent to a normal measure $U$, denoted by $E\simeq U$, if $E$ has no generator greater than its critical point, and the normal measure derived from $i_E$ is $U$.
For any $A\in \mathcal P(\mu_X)\cap \mathcal M_X$, there exists a large enough $\zeta'<\theta_X$ such that, $A\in U(\mathcal M_X,\mu_X,0)$ if and only if
$$ \forall \eta\in I_X\setminus \zeta' \left\{\begin{aligned}
    &\mu_\eta^X\in A,&\text{ if }E^{\mathcal T_X}_{\eta^*}\simeq U(\mathcal M^{\mathcal T_X}_{\eta},\mu^X_{\eta},0);\\
    &A\cap \mu^X_\eta\in U(\mathcal M^{\mathcal T_X}_{\eta},\mu^X_\eta,0), &\text{otherwise}.
\end{aligned}\right.$$
The local order-zero measure exists in $\mathcal M^{\mathcal T_X}_\eta$ by elementarity, since $i^{\mathcal T_X}_{\eta,\theta_X}(\mu^X_\eta)=\mu_X$ and $\mu_X$ is measurable in $\mathcal M_X$; see \cite[Lemma~3.7]{MS2}.
The same calculation shows that $C_X$ generates $U(\mathcal S_X,\nu,0)$.

We now look at different $X\in\mathcal F$. 
Recall that we have fixed a stationary subset $\mathcal G\subseteq \mathcal F$.
We now need to further shrink $\mathcal G$ to another stationary subset $\mathcal H$.
\begin{itemize}
    \item Suppose $\mathrm{cf}(\nu)>\omega$. 
    Let $D\subseteq \nu$ be a club subset such that $\mathrm{type}(D) = \mathrm{cf}(\nu)$, and every $\alpha\in D$ has cofinality $<\mathrm{cf}(\nu)$. 
    Let $\mathcal H = \{X\in \mathcal G\mid D\subseteq X\}$, and $B_X = C_X\cap D\cap X$. 
    \item Suppose $\mathrm{cf}(\nu) = \omega$. 
    For each $X\in\mathcal G$, pick an unbounded $B_X\subseteq C_X$ such that $\mathrm{type}(B_X) = \omega$.
    Since $X$ is the union of an internally approachable chain, we may find $D_X\in X$ such that $B_X\subseteq D_X$ and $D_X\subseteq X$.
    Use Fodor's Lemma to find a stationary subset $\mathcal H\subseteq \mathcal G$ such that the function $X\mapsto D_X$ is fixed on $\mathcal H$.
\end{itemize}

We may now sketch the proof of $\nu$ being measurable in $W$. 

\begin{proof}[Sketch.]
    Let
    $$U = \bigcup_{X\in\mathcal H}U(\mathcal S_X,\nu,0).$$
    Then $U$ is an ultrafilter over $\mathcal P(\nu)\cap W$ amenable to $W$.
    Moreover, $\mathrm{Ult}(W,U)$ is wellfounded and the phalanx $((W,\mathrm{Ult}(W,U)),\nu)$ is iterable. 
    We compare $W$ against this phalanx and let $\mathcal U$ and $\mathcal V$ be the coiteration trees, respectively.
    Then 
    $$\mathrm{root}^{\mathcal V} = \mathrm{Ult}(W,U),$$
    and $\mathcal M^{\mathcal U}_\infty = \mathcal M^{\mathcal V}_\infty$.
    Therefore, $\mathrm{crit}(i^{\mathcal U}_{0,\infty}) = \nu$, and the derived ultrafilter of $i^{\mathcal U}_{0,\infty}$ is $U$.
    Therefore, $U\in W$.
    $U$ is the order $0$ measure of $\nu$ in $W$, since there exists a substructure $X\in\mathcal H$ such that
    $$A = \{\alpha<\nu\mid \alpha\text{ is not measurable in }W\}\in X,$$
    and $A\in U(\mathcal S_X,\nu,0)$.
    By Theorem 1.2 of \cite{Schl}, the extender equivalent with $U$ is on the $W$-sequence.
\end{proof}

We now outline a more concise proof of Theorem \ref{T1-1} under the additional assumption that $|\nu|$ is $\mathrm{cf}(\nu)$-closed.
This approach may serve as a useful reference for the discussions in Chapters~\ref{Ch3}, \ref{Ch5}, and~\ref{Ch6}.
For this argument choose a closed substructure directly. Using the
cardinal arithmetic assumption, take
$X\prec(H_{\Omega^+},\in,U_\Omega)$ with
$|X|<|\nu|$, $\{\nu,\Omega_0,W\}\subseteq X$,
$\sup(X\cap\nu)=\nu$ and ${}^{\mathrm{cf}(\nu)}X\subseteq X$.
Theorem \ref{T2-1} gives Requirements 1 and 2, which suffice for
the calculation of $C_X$ and $U(\mathcal S_X,\nu,0)$.
As in \cite[Section~3.1]{MS2}, choose a cofinal $B_X\subseteq C_X$
of order type $\mathrm{cf}(\nu)$ from the order-zero steps if these
are cofinal, and from the remaining steps otherwise. Since
$B_X\subseteq X$, closure gives $B_X\in X$.
For $A\in\mathcal P(\nu)\cap W\cap X$, Lemma~3.10 of \cite{MS2}
ensures that $U(\mathcal S_X,\nu,0)$ measures $A$.
In the first case $A$ belongs to this measure exactly when
every sufficiently large member of $B_X$ lies in $A$. In the second case the criterion is that
$A\cap\beta\in U(W,\beta,0)$ for all sufficiently large $\beta\in B_X$,
by the local measure calculation in \cite[Section~3.1]{MS2}.
In either case the criterion is definable from $B_X$, $\nu$ and $W$.
By elementarity it defines a total ultrafilter over $W$.
The amenability, iterability and absorption argument of \cite[Sections~3.1 and~5]{MS2}
then gives measurability in $W$.

\ChapterLabel{Ch3}
\chapter{Unique Normal Measure with Closure Assumption}

We prove the following theorem.
\begin{theorem}\label{T3-1}
    Assume there is no inner model with a Woodin cardinal. Suppose that $\nu>\omega_2$ is a cardinal in the Jensen-Steel core model $K$ such that 
    $$\mathrm{cf}(\nu)<|\nu|\leq\nu = \mathrm{cf}^K(\nu),$$
    and $o^K(\nu) = 1$.
    Further assume that $|\nu|$ is countably closed, that is,
    $$\forall \kappa<|\nu|(\kappa^\omega<|\nu|).$$
    Then there exists a maximal Prikry sequence $C_0$ over the unique total normal measure on $\nu$ in $K$. Moreover, if measurable cardinals in $K$ are bounded in $\nu$, then for every $A\subseteq \nu$ with $|A|<|\nu|$, there exists $B\in K[C_0]$ such that $B\supseteq A$ and $|B|^{K[C_0]}<\nu$.
\end{theorem}
Note that the hypotheses of Theorem \ref{T3-1} imply that $\mathrm{cf}(\nu) = \omega$ by Theorem \ref{T1-1}.
Theorem \ref{T3-1} is a special case of Theorems \ref{T1-2} and \ref{T1-3}, with the additional assumption that $|\nu|$ is countably closed. 

We use the notation of Chapter~\ref{Ch2}.
Fix $\Omega,\Omega_0,W,\nu$ as stated. 
Fix $X\in\mathcal F$ such that ${}^\omega X\subseteq X$.
Throughout this chapter, we fix $X$ and suppress the subscript $X$.
We write $W_X$ as $\overline W$ and $\mathcal M^{\mathcal T_X}_{\theta_X}$ as $\mathcal M_X$.
Recall that Lemma \ref{L2-10} defines $\zeta$ and the critical sequence $C$. 
The following lemma describes $C$ 
under the hypothesis that $o^K(\nu) = 1$. 

\begin{lemma}\label{L3-2}
    There exists a $\zeta_0\in (\zeta,\theta)_{\mathcal T}$ such that 
    $$E^{\mathcal T}_{\eta} = E(\mathcal M^{\mathcal T}_{\eta},\mu_\eta,0),$$
    and $\eta+1<_{\mathcal T}\theta$ for all $\eta\in I\setminus \zeta_0$. Here, $E(\mathcal M^{\mathcal T}_{\eta},\mu_\eta,0)$ denotes the zeroth total extender on the $\mathcal M^{\mathcal T}_{\eta}$-sequence with critical point $\mu_\eta$, ordered by its length.
\end{lemma}

\begin{proof}
    For $\eta\in I$, let $\xi(\eta)+1$ be the immediate successor of $\eta$ on the branch to $\theta$.
    If $E^{\mathcal T}_{\xi(\eta)}$ is the order-zero extender over $\mu_\eta$, then $\xi(\eta)=\eta$ by \cite[Section~3.1, p.~18]{MS2}.
    It therefore suffices to rule out a cofinal set of $\eta\in I$ for which the extender on this branch edge is not the order-zero extender.
    We record why a backward edge still supplies a total measure on its predecessor. Suppose $\beta=\operatorname{pred}_{\mathcal T}(\xi+1)<\xi$, there is no drop on this edge, and the successor of $\kappa=\operatorname{crit}(E^{\mathcal T}_\xi)$ exists in $\mathcal M^{\mathcal T}_\beta$. Write $\tau=\operatorname{lh}(E^{\mathcal T}_\beta)$. The predecessor and carrier agree below $\tau$ and have the same subsets of $\kappa$. Moreover, $\tau$ is a cardinal of the carrier and $\kappa^+\leq\tau$ there \cite[p.~61, proof of Lemma~6.1.5]{MSt}. Equality would make $\tau$ a cardinal strictly below $\mathcal M^{\mathcal T}_\beta\cap\mathrm{Ord}$, although the extender at index $\tau$ makes it a noncardinal there \cite[p.~8, Remark~(e)]{MSt}. Hence $\kappa^+<\tau$.

    The active segment carrying $E^{\mathcal T}_\xi$ has $\tau$ as a cardinal and has no cardinal above the natural length of its extender \cite[pp.~7--8, Remark~(d)]{MSt}. Thus restriction to $\kappa^+$ gives a proper normal-measure fragment. Its trivial completion $F$ occurs earlier on the carrier sequence by clause (a) of \cite[Definition~2.4]{SSZ}. The fragment has only one generator, and clause (b) would put an extender at the cardinal $\kappa^+$. By \cite[p.~8, Remark~(e)]{MSt}, the carrier has a surjection onto $\operatorname{lh}(F)$ from a set of size at most $\kappa^+$, using subsets of $\kappa$ and finite tuples below the natural length $\kappa^+$ of $F$. As $\tau$ is a cardinal in the carrier, $\operatorname{lh}(F)<\tau$. In particular $F$ is on the predecessor sequence and measures all its subsets of $\kappa$.

    Now suppose the bad edges are cofinal. Since $\operatorname{cf}(\theta)=\operatorname{cf}(\mu)=\omega$, choose a cofinal set $J$ of their predecessors of order type $\omega$.
    For $\eta\in J$, $\mu_\eta$ already has a zeroth total extender by branch elementarity, so its successor cardinal exists in $\mathcal M^{\mathcal T}_\eta$.
    If $\xi(\eta)=\eta$, the bad edge itself is a different total extender on that sequence. Otherwise apply the preceding observation at $\kappa=\mu_\eta$. By coherence $F$ still occurs in $\mathcal M^{\mathcal T}_{\xi(\eta)+1}$, and weak amenability preserves the subsets of $\kappa$, so this model sees $\kappa$ as measurable. The factor from the normal ultrapower by $F$ fixes $\kappa$. Hence that ultrapower also sees $\kappa$ as measurable, whereas the ultrapower by the zeroth extender does not. Thus the predecessor has two total extenders in either case, and its branch map gives two at $\mu$ in $\mathcal M_X$.
    We first record a fact about the first two total extenders.

    \begin{claim}\label{c3-2-1}
    Let $\mathcal M$ be an extender model. Let $E_0 = E(\mathcal M,\kappa,0)$ and $E_1 = E(\mathcal M,\kappa,1)$. Then both $E_0$ and $E_1$ have no generator greater than $\kappa$.

    \end{claim}
    
    \begin{proof}
        Suppose that the claim fails for $E_0$. The initial segment condition at $\kappa^{+\mathcal M}$ \cite[Definition~2.4]{SSZ} for extender sequences implies that the trivial completion of $E_0\upharpoonright\kappa^{+\mathcal M}$ is a total extender on the $\mathcal M$-sequence with length strictly less than that of $E_0$, contradicting the definition of $E_0$.
        
        Suppose that the claim fails for $E_1$. Let $\xi$ be the least such generator, and let $F$ be the trivial completion of $E_1\upharpoonright\kappa^{+\mathcal M}$. By the initial segment condition, $F$ is on the $\mathcal M$-sequence; hence it is $E(\mathcal M,\kappa,0)$. Let $n$ be the least such that $\rho_{n+1}(\mathcal M)\leq \kappa$ if it exists. If there is no such $n$, let $n = \omega$. Let $\mathcal M_1 = \mathrm{Ult}_n(\mathcal M,E_1)$ and $\mathcal M_0 = \mathrm{Ult}_n(\mathcal M,F)$. Let $i_0:\mathcal M\to \mathcal M_0$, $i_1:\mathcal M\to \mathcal M_1$ and $k:\mathcal M_0\to \mathcal M_1$ be the natural embeddings. We have $\mathrm{crit}(i_0) = \mathrm{crit}(i_1) = \kappa$ and $\mathrm{crit}(k) = \xi$.
        
        By coherence, $F$ appears on the sequence of $\mathcal M_1$. Thus, $o^{\mathcal M_1}(\kappa)\geq 1$, which implies $o^{\mathcal M_0}(\kappa)\geq 1$. However, this contradicts the fact that there is no total extender on the $\mathcal M_0$-sequence with critical point $\kappa$.
    \end{proof}

    As a result, for $i\in\{0,1\}$, $E(\mathcal M_X,\mu,i)\simeq U(\mathcal M_X,\mu,i)$, and the same holds for every $\mu_\eta$, where $\eta\in J$. Fix any $A\in \mathcal{P}(\mu)\cap \mathcal M_X$. By the direct limit, $A\in\mathrm{ran}(i^{\mathcal T}_{\eta,\theta})$ for every sufficiently late $\eta\in J$. For such $\eta$ we distinguish the two possibilities for $E^{\mathcal T}_{\xi(\eta)}$. 

    \noindent \textbf{Case 1.} $E^{\mathcal T}_{\xi(\eta)} = E(\mathcal M^{\mathcal T}_{\eta},\mu_{\eta},1)$.
    In this case, $\xi(\eta) = \eta$, and the coherence of the iteration implies:
    \[
    \begin{cases}
        A\in U(\mathcal M_X,\mu,0) \iff A\cap \mu_{\eta}\in U(\mathcal M^{\mathcal T}_{\eta},\mu_\eta,0);\\
        A\in U(\mathcal M_X,\mu,1) \iff \mu_\eta\in A.
    \end{cases}
    \]
    
   \noindent \textbf{Case 2.} $E^{\mathcal T}_{\xi(\eta)} \neq E(\mathcal M^{\mathcal T}_{\eta},\mu_{\eta},1)$.
    Here, the length of the extender satisfies $\mathrm{lh}(E^{\mathcal T}_{\xi(\eta)}) > \mathrm{lh}(E(\mathcal M^{\mathcal T}_{\eta},\mu_{\eta},1))$, and we have:
    \[
    \begin{cases}
        A\in U(\mathcal M_X,\mu,0) \iff A\cap \mu_{\eta}\in U(\mathcal M^{\mathcal T}_{\eta},\mu_\eta,0);\\
        A\in U(\mathcal M_X,\mu,1) \iff A\cap \mu_{\eta}\in U(\mathcal M^{\mathcal T}_{\eta},\mu_\eta,1).
    \end{cases}
    \]
    Thus, for every $A\in \mathcal P(\mu)\cap \mathcal M_X$ and $i\in\{0,1\}$, $A\in U(\mathcal M_X,\mu,i)$ if and only if
$$
        \exists\,\eta'<\theta\,\forall\,\eta\in J\setminus\eta'\,\left\{\begin{aligned}
        &\mu_{\eta}\in A&\mbox{if }i = o^{\mathcal M_X}(\mu_{\eta})\\
        &A\cap \mu_{\eta}\in U(\mathcal M_X,\mu_{\eta},i)&\mbox{if }i < o^{\mathcal M_X}(\mu_{\eta})\\
    \end{aligned}\right.
$$
    Thus the set 
    $$\overline C_J = \{\mu_\eta\mid \eta\in J\}$$
    generates the measure $U(\mathcal M_X,\mu,i)$, for $i\in\{0,1\}$.
    Since 
    $$\pi[\overline C_J]\in X\prec H_{\Omega^+}$$
    and 
    $$\mathcal{P}(\mu)\cap \overline W\subseteq \mathcal{P}(\mu)\cap \mathcal M_X,$$
    it follows that $\pi[\overline C_J]$ generates two normal measures, $U_0$ and $U_1$, on $\nu$ over $W$.
    The absorption argument of \cite[Sections~3.1 and~5]{MS2}, with the first correction in the Appendix, puts the extender equivalent to $U_0$ on the $W$-sequence.
    A similar argument applies to $U_1$, using the second extender in Claim~\ref{c3-2-1} as in the proof of \cite[Lemma~6.2]{MS2}.
    To see that $U_0\neq U_1$, consider the set
    $$A = \{\alpha<\mu\mid \alpha\mbox{ is measurable in }\overline W\},$$
    for which $\pi(A)\in U_1\setminus U_0$. 
    Consequently, $W$ contains at least two measures on $\nu$, which implies $o^K(\nu) \ge 2$, contradicting the assumption that $o^K(\nu) = 1$.
\end{proof}

Let $I_0 = I\setminus \zeta_0$, $\overline C_0 = \{\mu_\eta\mid \eta\in I_0\}$ and $C_0 = \pi[\overline C_0]$. 
By discarding a bounded initial segment we may assume that $\pi(\mu_\eta)>\omega_2$ for every $\eta\in I_0$. If $|\nu|<\nu$, we may also choose $I_0$ so that for all $\eta\in I_0$, $|\pi(\mu_\eta)| = |\nu|$.
We next show that $C_0$ is a maximal Prikry sequence.

\begin{lemma}\label{L3-3}
    $\mathrm{type}(C_0) = \omega$.
\end{lemma}

\begin{proof}
    Otherwise, let $\eta$ be the $\omega$-th member of $I_0$. 
    The critical points at the first $\omega$ members of $I_0$ increase to $\mu_\eta$. Their images form a sequence in $X$ by countable closure, so their supremum is $\pi(\mu_\eta)$. Thus $\pi(\mu_\eta)>\omega_2$ has countable cofinality in $V$ and is regular in $K$. Theorem \ref{T1-1} makes it measurable in $K$.
    However, since $\mathcal M_X$ agrees with $\overline W$ below $\mu$, and $E^{\mathcal T}_{\eta} = E(\mathcal M^{\mathcal T}_\eta,\mu_\eta,0)$, it follows that there is no total extender on the $\overline W$-sequence with critical point $\mu_\eta$.
    Therefore, $\pi(\mu_\eta)$ is not measurable in $W$. 
    A contradiction.
\end{proof}

By the above two lemmas, $\overline C_0$ generates $U(\mathcal M_X,\mu,0)$, meaning that for any $A\in \mathcal P(\mu)\cap \mathcal M_X$, 
$$A\in U(\mathcal M_X,\mu,0)\iff \overline C_0\subseteq^* A.$$
Thus $\overline C_0$ is a Prikry sequence for $U(\mathcal M_X,\mu,0)$ by the Mathias criterion. 
Notice that $U(\mathcal M_X,\mu,0)\cap \overline W$ is also a total measure over $\mathcal P(\mu)\cap\overline W$. Hence
$$N_X\vDash \overline C_0\text{ generates a total measure over }\mathcal P(\mu)\cap\overline W.$$
By elementarity, 
$$H_{\Omega^+}\vDash C_0\text{ generates a total measure over }\mathcal P(\nu)\cap W.$$
Call this measure $U$.
By \cite[Section~3.1 and Lemmas~5.2--5.3]{MS2}, with the first correction in the Appendix, $((W,\mathrm{Ult}(W,U)),\nu)$ is iterable and $U\in W$.
Moreover, since every member of $C_0$ is not measurable in $W$, $U = U(W,\nu,0)$. 
Therefore, $C_0$ is a Prikry sequence for $U(W,\nu,0)$.

It remains to prove the maximality of $C_0$; that is, $C_0$ almost contains every Prikry sequence for $U(W,\nu,0)$. 
This follows from the next lemma.

\begin{lemma}\label{L3-4}
    Let $\mathbf{d}\subseteq \nu\cap X$ with order type $\omega$.
    Then $\mathbf{d}\subseteq^* C_0$ if and only if for any club subset $A\in \mathcal P(\nu)\cap W$, $\mathbf{d}\subseteq^* A$. 
\end{lemma}

\begin{proof}
    Since $C_0$ is a Prikry sequence for $U(W,\nu,0)$, it is enough to show that if  for any club subset $A\in \mathcal P(\nu)\cap W$, $\mathbf{d}\subseteq^* A$, then $\mathbf{d}\subseteq^* C_0$.
    Assume otherwise and fix a witness $\mathbf d$. The final segments of $\nu$ are clubs in $W$, so $\mathbf d$ is cofinal in $\nu$.
    Without loss of generality, assume $C_0\cap \mathbf{d} = \emptyset$ and $\min \mathbf{d}>\min C_0$.
    Let $\overline{\mathbf d} = \pi^{-1}(\mathbf d)$.
    We will find a club subset $A\in \mathcal{P}(\nu)\cap W$ such that $A\cap \mathbf{d}$ is finite. 
    Recall that we have introduced the $\mathcal Q$-structures and $\mathcal S$-structures.
    In this chapter, we use $\mathcal Q_X$ for $\mathcal Q_\mu$, and $\mathcal S_X$ for $\mathcal S_\mu$.
    We consider four cases, according to whether $\mathcal Q_X$ is a set mouse or a weasel, and whether $\mathcal Q_X$ falls into the protomouse case.
    By the definitions in Chapter~\ref{Ch2}, we fix $m = m(\mu)$ and $n = n(\mu)$.

    \subsubsection*{Case 1.}
    Suppose $\mathcal Q_{X}$ is a set mouse,
    and $\mathcal Q_X = \mathcal M_X$.
    By Lemma \ref{L2-4}, $\mathcal S_X\lhd W$. 
    Since this is not the protomouse case, $n=m$ by the definitions in Chapter~\ref{Ch2}. Thus $\mathcal M_X$ is $m+1$-sound above $\mu$; let $p = p_{m+1}(\mathcal M_X)\setminus \mu$.
    By the fact that $\mathcal M_X$ is sound above $\mu$ relative to the parameter $p$, we have that $\mathcal S_X$ is sound above $\nu$ relative to the parameter $\Pi(p)$.
    Let $q = \Pi(p)$, and
    $$A = \{\alpha<\nu\mid \alpha = \nu\cap \mathrm{Hull}^{\mathcal S_X}_{{m+1}}(\alpha\cup q)\}.$$
    Observe that $A\in W$ and it is a closed subset of $\nu$. 
    To show it is unbounded in $\nu$, notice that for each $\eta\in I_0$, $i^{\mathcal T}_{\eta,\theta}(\mu_\eta) = \mu$ and thus
    $$\mu_\eta = \mu\cap \mathrm{Hull}^{\mathcal M_X}_{{m+1}}(\mu_\eta\cup p).$$
    Recall that $\Pi\upharpoonright\mu = \pi\upharpoonright\mu$.
    Pushing forward to $\mathcal S_X$, this gives
    $$\pi(\mu_\eta) = \Pi(\mu_\eta) = \nu\cap \mathrm{Hull}^{\mathcal S_X}_{{m+1}}(\pi(\mu_\eta)\cup q).$$
    Therefore, $\pi(\mu_\eta)\in A$ and $A$ is unbounded.

    We now check that $\mathbf d\cap A$ is finite.
    Let $d\in \overline{\mathbf{d}}$ and $\eta<\xi$ be consecutive members of $I_0$ such that 
    $${\mu}_\eta<d<{\mu}_{\xi}.$$
    It is enough to show that
    $$
        d\neq {\mu}\cap \mathrm{Hull}^{\mathcal M_X}_{{m+1}}(d\cup p).
    $$
    We may take $d$ large enough that $p_{m+1}(\mathcal M_X)\cap\mu\subseteq d$. Since the branch preserves the standard parameter and fixes ordinals below $\mu_\xi$, we then have $p_{m+1}(\mathcal M^{\mathcal T}_\xi)\cap\mu_\xi\subseteq d$.
    By the elementarity of $i^{\mathcal T}_{\xi,\theta}$, it suffices to show
    \begin{equation}\label{E3-1}
        d\neq {\mu}_\xi\cap \mathrm{Hull}^{\mathcal M^{\mathcal T}_{\xi}}_{{m+1}}(d\cup p_{m+1}(\mathcal M^{\mathcal T}_{\xi})).
    \end{equation}
    If $\xi = \eta+1$, the inequality holds trivially. 
    Otherwise, we have $\eta<_{\mathcal T}\eta+1<_{\mathcal T}\xi$. The iteration gives
    $$i^{\mathcal T}_{\eta,\eta+1}({\mu}_{\eta}) = \sup\left(i^{\mathcal T}_{\eta,\eta+1}({\mu}_{\eta})\cap \mathrm{Hull}^{\mathcal M^{\mathcal T}_{\eta+1}}_{{m+1}}\left(({\mu}_\eta+1)\cup p_{m+1}(\mathcal M^{\mathcal T}_{\eta+1})\right)\right).$$
    Using the fact that $\mathrm{crit}(i^{\mathcal T}_{\eta+1,\xi})>{\mu}_\eta$ and 
    $$\sup\left(i^{\mathcal T}_{\eta+1,\xi}\left[i^{\mathcal T}_{\eta,\eta+1}({\mu}_\eta)\right]\right) = {\mu}_\xi,$$
    we obtain
    $${\mu}_\xi = \sup\left({\mu}_\xi\cap \mathrm{Hull}^{\mathcal M^{\mathcal T}_{\xi}}_{{m+1}}\left(({\mu}_\eta+1)\cup p_{m+1}(\mathcal M^{\mathcal T}_{\xi})\right)\right).$$
    By ${\mu}_\eta<d<{\mu}_\xi$, (\ref{E3-1}) follows. 
    Therefore, $\pi(d)\not\in A$.

    \subsubsection*{Case 2.} Suppose $\mathcal Q_X$ is a set mouse, but the protomouse case takes place, that is, $\mathcal Q_X\neq \mathcal M_X$.
    In this case, we still have $\mathcal S_{X}\lhd W$, but the analysis needs to be modified for the protomouse case.
    Lemma \ref{L2-3} gives the relation between $\mathcal Q_X$ and $\mathcal M_X$, namely, there exists a parameter $\overline q$ such that for every large enough $\alpha<\mu$,
    \begin{equation}\label{E3-2}
        \mu\cap \mathrm{Hull}^{\mathcal M_X}_{1}(\alpha\cup p_{1}(\mathcal M_X)) = \mu\cap \mathrm{Hull}^{\mathcal Q_X}_{n+1}(\alpha\cup \overline q),
    \end{equation}
    where $\mathcal Q_X$ is $n+1$-sound above $\mu$.
    Moreover, $\mathcal Q_X$ is sound above $\mu$ relative to $\overline q$, that is, 
    $$\mathcal Q_X = \mathrm{Hull}^{\mathcal Q_X}_{n+1}(\mu\cup \overline q).$$
    Therefore we still have the property that $\mathcal S_X$ is sound above $\nu$ relative to $\Pi(\overline q)$ as in Case 1. 
    Notice that $\overline q\cap \mu = \emptyset$.
    Let $q =\Pi(\overline q)$.
    We define 
    $$A = \{\alpha<\nu\mid \alpha = \nu\cap \mathrm{Hull}^{\mathcal S_X}_{n+1}(\alpha\cup q)\}.$$
    As in Case 1, $A\in W$ and it is a closed subset of $\nu$.
    To show it is unbounded in $\nu$, notice that for each $\eta\in I_0$, we still have 
    $$\mu_\eta = \mu\cap \mathrm{Hull}^{\mathcal M_X}_{1}(\mu_\eta\cup [p_1(\mathcal M_X)\setminus \mu]).$$
    Suppose that $\mu_\eta$ is large enough to satisfy equation (\ref{E3-2}). 
    Then
    $$\mu_\eta = \mu\cap \mathrm{Hull}^{\mathcal Q_X}_{{n+1}}(\mu_\eta\cup \overline q).$$
    Pushing forward to $\mathcal S_X$ gives $\pi(\mu_\eta)\in A$. Therefore $A$ is unbounded in $\nu$.
    
    To show that $\mathbf d\cap A$ is finite, it is enough to verify that for every large enough $d\in\overline{\mathbf d}$,
    $$d\neq {\mu}\cap \mathrm{Hull}^{\mathcal Q_X}_{{n+1}}(d\cup \overline q).$$
    Since $\overline{\mathbf{d}}$ is cofinal in $\mu$, we may choose $d$ sufficiently large to satisfy equation (\ref{E3-2}).
    The argument of Case 1 gives
    $$d\neq {\mu}\cap \mathrm{Hull}^{\mathcal M_X}_{1}(d\cup [p_{1}(\mathcal M_X)\setminus \mu]).$$
    Therefore, 
    $$d\neq {\mu}\cap \mathrm{Hull}^{\mathcal Q_X}_{n+1}(d\cup \overline q).$$
    Pushing forward to $\mathcal S_X$ gives the result.
    
    \subsubsection*{Case 3.}
    Suppose $\mathcal Q_X$ is a weasel, and $\mathcal M_X = \mathcal Q_X$.
    In this case, there is no drop in $[0,\theta)_{\mathcal T}$, and $c(\mathcal Q_X) = c(\mathcal M_X) = \emptyset$.
    Let $G\in W$ be the short extender of Lemma~\ref{L2-5}, so $j:W\to\mathcal S_X=\mathrm{Ult}(W,G)$ is its ultrapower map.
    Let $\kappa = \mathrm{crit}(j)$.
    For $\alpha<\nu$, let $h(\alpha)$ be the supremum of those ordinals below $\nu$ of the form $j(f)(b)$, where $f:[\kappa]^{<\omega}\to\kappa$ belongs to $W$ and $b\in[\alpha]^{<\omega}$.
    This evaluation is computed in $W$ from $G$. In $W$ there are $(2^\kappa)^W<\nu$ possible functions, since $\nu$ is measurable there, so regularity gives $h(\alpha)<\nu$.
    The function is nondecreasing and continuous at limits by finiteness of the supports. Evaluating the successor function at singleton supports gives $h(\alpha)\geq\alpha$.
    Thus $A=\{\alpha<\nu:h(\alpha)=\alpha\}$ is closed and belongs to $W$. We show that $C_0\subseteq^* A$ and $A\cap\mathbf d$ is finite.

    Let $\Gamma$ be a thick class of common fixed points of $i^{\mathcal T}_{0,\theta}$, $\Pi$ and $j$.
    Take any sufficiently large $\eta\in I_0$ with $\pi(\mu_\eta)>\kappa$. Since $\mathcal M^{\mathcal T}_\eta$ is sound above $\mu_\eta$, and $i^{\mathcal T}_{\eta,\theta}(\mu_\eta) = \mu$, we have
    $$\mu_\eta = \mu\cap\mathrm{Hull}^{\mathcal M_X}(\mu_\eta\cup \Gamma).$$
    Pushing forward to $\mathcal S_X$ gives
    \begin{equation}\label{E3-3}
        \pi(\mu_\eta)=\nu\cap\mathrm{Hull}^{\mathcal S_X}(\pi(\mu_\eta)\cup\Gamma).
    \end{equation}
    If $\pi(\mu_\eta)\notin A$, then $h(\pi(\mu_\eta))>\pi(\mu_\eta)$. Thus some $f,b$ as above satisfy
    $\pi(\mu_\eta)<j(f)(b)<\nu$.
    The graph of $f$ has a code in $W\upharpoonright(\kappa^+)^W$, below $\Omega_0$. The definability property of $W$ at every ordinal below $\Omega_0$, together with canonical coding, gives $f\in\mathrm{Hull}^W(\Gamma)$.
    As $j$ fixes $\Gamma$ pointwise, $j(f)\in\mathrm{Hull}^{\mathcal S_X}(\Gamma)$. Evaluation at $b\in[\pi(\mu_\eta)]^{<\omega}$ contradicts \eqref{E3-3}.
    Therefore, $C_0\subseteq^* A$ and $A$ is unbounded.  
    
    For the opposite direction, let $d\in \overline{\mathbf{d}}$.
    The calculation in Case 1 gives
    $$d\neq \mu\cap \mathrm{Hull}^{\mathcal M_X}(d\cup \Gamma).$$
    Pushing forward to $\mathcal S_X$ by $\Pi$, we obtain
    $$\pi(d)\neq \nu\cap \mathrm{Hull}^{\mathcal S_X}(\pi(d)\cup \Gamma).$$
    This hull contains every ordinal below $\pi(d)$ and is closed under successor. Hence there is a Skolem term $\tau$, a finite tuple $a$ from $\pi(d)$ and a finite tuple $\gamma$ from $\Gamma$ such that
    $$\pi(d)<\tau^{\mathcal S_X}[a,\gamma]<\nu.$$
    In $W$ let $f$ agree with $\tau^W[-,\gamma]$ on inputs of the required arity whose value lies in $\kappa$, and take value $0$ on all other inputs. This is a set function in $W$ by Separation for the fixed term. Since $j(\gamma)=\gamma$ and $j(\kappa)>\nu$, we have $j(f)(a)=\tau^{\mathcal S_X}[a,\gamma]>\pi(d)$. Thus $\pi(d)\notin A$.

    \subsubsection*{Case 4.} Suppose $\mathcal Q_X$ is a weasel, but $\mathcal M_X\neq \mathcal Q_X$.
    This results in the protomouse case, and thus $\mathcal M_X$ is a set mouse.
    Use the relative parameter $t_\mu=c(\mathcal Q_X)\setminus\mu$, whose image is $\Pi(t_\mu)=c(\mathcal S_X)\setminus\nu$ in the notation of Chapter~\ref{Ch2}. By Lemma~\ref{L2-3}, for every thick class $\Gamma_0$ fixed pointwise by the composite from $W$ to $\mathcal Q_X$ and all sufficiently large $\xi<\mu$,
    \begin{equation}\label{E3-4}
        \mu\cap \mathrm{Hull}^{\mathcal M_X}_1(\xi\cup [p_{1}(\mathcal M_X)\setminus \mu]) = \mu\cap \mathrm{Hull}^{\mathcal Q_X}(\xi\cup t_\mu\cup \Gamma_0).
    \end{equation}
    By Lemma~\ref{L2-6}, $j:W\to\mathcal S_X=\mathrm{Ult}(W,G)$ is the ultrapower by $G\in W$ with support $\nu\cup \Pi(t_\mu)$. Put $\kappa=\mathrm{crit}(j)$. For $\alpha<\nu$ let $h(\alpha)$ be the supremum of those ordinals below $\nu$ of the form
    $$j(f)(b,\Pi(t_\mu)),$$
    where $b\in[\alpha]^{<\omega}$ and $f\in W$ is a $\kappa$-valued function on pairs of finite tuples from $\kappa$. We use increasing enumerations for the finite supports.
    The fixed tuple $\Pi(t_\mu)$ belongs to $W$, as it is a finite set of ordinals, and lies below $j(\kappa)$ by Lemma~\ref{L2-6}. Thus $W$ computes $h$ from this tuple and $G$.
    As in Case~3, $h$ maps $\nu$ to $\nu$, is continuous at limits, and satisfies $h(\alpha)\geq\alpha$ (ignore the second argument in the successor function). Let $A$ be its closed set of fixed points.

    Choose $\Gamma_0$ as in \eqref{E3-4} and take a thick subclass $\Gamma$ of common fixed points of $j$ and $\Pi$ above $\nu$. Equation~\eqref{E3-4} applies to $\Gamma$ as well.
    Take $\eta\in I_0$ large enough that $\pi(\mu_\eta)>\kappa$ and
    \eqref{E3-4} applies with $\xi=\mu_\eta$.
    Since $\mathcal M^{\mathcal T}_\eta$ is
    $\mu_\eta$-sound, 
    and 
    $i^{\mathcal T}_{\eta,\theta}(\mu_\eta) = \mu$,
    we have
    $$\mu_\eta = \mu\cap\mathrm{Hull}_1^{\mathcal M_X}(\mu_\eta\cup [p_{1}(\mathcal M_X)\setminus \mu]).$$
    Equation (\ref{E3-4}) tells us 
    $$\mu_\eta = \mu\cap\mathrm{Hull}^{\mathcal Q_X}(\mu_\eta\cup t_\mu\cup \Gamma).$$
    Pushing forward to $\mathcal S_X$, we have
    $$
        \pi(\mu_\eta) = \nu\cap\mathrm{Hull}^{\mathcal S_X}(\pi(\mu_\eta)\cup \Pi(t_\mu)\cup \Gamma).
    $$
    If $\pi(\mu_\eta)\notin A$, some $f\in W$ and $b\in[\pi(\mu_\eta)]^{<\omega}$ satisfy
    $$\pi(\mu_\eta)< j(f)(b,\Pi(t_\mu))<\nu.$$
    As in Case~3, $f\in\mathrm{Hull}^W(\Gamma)$, so $j(f)\in\mathrm{Hull}^{\mathcal S_X}(\Gamma)$. Both $b$ and the finite tuple $\Pi(t_\mu)$ belong to the displayed hull. Evaluating $j(f)$ there contradicts the displayed equality. Thus $C_0\subseteq^* A$ and $A$ is unbounded.

    On the other hand, let $d\in \overline{\mathbf{d}}$ 
    be large enough to satisfy equation (\ref{E3-4}).
    The calculation in Case 1 gives
    $$d\neq \mu\cap \mathrm{Hull}_1^{\mathcal M_X}(d\cup [p_{1}(\mathcal M_X)\setminus \mu]).$$
    By equation (\ref{E3-4}), we have 
    $$d\neq \mu\cap \mathrm{Hull}^{\mathcal Q_X}(d\cup t_\mu\cup \Gamma).$$
    Pushing forward to $\mathcal S_X$ by $\Pi$, we obtain
    $$\pi(d)\neq \nu\cap \mathrm{Hull}^{\mathcal S_X}(\pi(d)\cup \Pi(t_\mu)\cup \Gamma).$$
    Take a Skolem term $\tau$ and finite tuples $a$ from $\pi(d)$ and $\gamma$ from $\Gamma$ with
    $$\pi(d)<\tau^{\mathcal S_X}[a,\Pi(t_\mu),\gamma]<\nu.$$
    In $W$, truncate and totalize $\tau^W[-,-,\gamma]$ as in Case~3 to obtain a $\kappa$-valued function $f$ on pairs of finite tuples. The coordinates of $a$ and $\Pi(t_\mu)$ and the displayed value all lie below $j(\kappa)$. Thus
    $$j(f)(a,\Pi(t_\mu))=\tau^{\mathcal S_X}[a,\Pi(t_\mu),\gamma]>\pi(d),$$
    so $\pi(d)\notin A$.
\end{proof}

We use the same division into cases according to the $\mathcal Q$-structures throughout the rest of the paper.

\begin{corollary}\label{C3-5}
    $C_0$ is the maximal Prikry sequence for $U(W,\nu,0)$.
\end{corollary}

\begin{proof}
    For maximality, suppose the corollary is false.
    An infinite subsequence of a Prikry sequence is again Prikry, so we may take a counterexample $\mathbf d$ disjoint from $C_0$.
    By the elementarity of $X$, since $C_0\in X$, we can assume that $\mathbf d\in X$.
    By the Mathias condition of $\mathbf d$, $\mathbf d\subseteq^* A$ for all club subsets in $\mathcal P(\nu)\cap W$.
    However, Lemma \ref{L3-4} tells us that $\mathbf d\subseteq^* C_0$. 
    A contradiction.
\end{proof}

We turn to the second half of Theorem \ref{T3-1}.
The argument is similar to the proof of Lemma \ref{L3-4}.

\begin{lemma}\label{L3-6}
    Assume that measurable cardinals in $K$ are bounded in $\nu$.
    Let $A\subseteq \nu$ with $|A|<|\nu|$. Then there exists
    $B\in K[C_0]$ such that $A\subseteq B$ and 
    $|B|^{K[C_0]}<\nu$.
\end{lemma}

\begin{proof}
    Choose a countably closed $X\in\mathcal F$ with $A\subseteq X$.
    Indeed, take $\lambda=\max(\omega_1,|A|)^{\aleph_0}<|\nu|$ and use the chain of Theorem~\ref{T2-2} with $\varepsilon=\lambda^+$.
    Here $\lambda^+\leq|\nu|$ and $\lambda^+<\nu$ since $\mathrm{cf}(\nu)=\omega$.
    Start the chain with all members of $A$, all ordinals at most $\lambda$, a cofinal $\omega$-sequence in $\nu$, and the fixed parameters.
    Models of size $\lambda$ suffice, and the successor models can be chosen countably closed since $\lambda^{\aleph_0}=\lambda$.
    At a club point of cofinality $\omega_1$ as in the definition of $\mathcal F$, the union is still countably closed.
    By Corollary~\ref{C3-5}, the sequence obtained from this $X$ and the previously fixed $C_0$ almost contain one another. Their finite symmetric difference belongs to $K$, so they give the same $K[C_0]$. In the proof we use $C_0$ for the sequence obtained from $X$.
    We seek $B\in K[C_0]$ with $|B|^{K[C_0]}<\nu$ and $X\cap \nu\subseteq B$.
    The definition of $B$ will depend on whether $\mathcal Q_{X}$ is a set mouse or a weasel, and whether $\mathcal Q_X$ falls into the protomouse case. 
    The following claim simplifies the definition of $B$.
    \begin{claim}\label{c3-6-1}
        Let $\eta\in I_0$. 
        Then
        $\eta+1 = \min(I_0\setminus (\eta+1))$. 
    \end{claim}
    \begin{proof}[Proof of Claim.]
        Otherwise, let 
        $$\eta<_{\mathcal T}\eta+1<_{\mathcal T}\eta'+1$$
        be consecutive members of $[\zeta_0,\theta)_{\mathcal T}$.
        We have
        \begin{equation}\label{E3-5}
            \mu_\eta = \mathrm{crit}(E^{\mathcal T}_{\eta}) 
        <\mathrm{crit}(E^{\mathcal T}_{\eta'})<i^{\mathcal T}_{\eta,\eta+1}(\mu_\eta).
        \end{equation}
        For the second inequality, if the next critical point were greater than $i^{\mathcal T}_{\eta,\eta+1}(\mu_\eta)$, normality would make every subsequent branch map fix this image.
        The image would then equal $i^{\mathcal T}_{\eta,\theta}(\mu_\eta)=\mu$, although there are later branch critical points $\mu_\xi<\mu$ for $\xi\in I_0$.
        Equality would put $\eta+1$ in $I_0$, since its critical point would map to $\mu$.
        The next critical point is measurable in $\mathcal M^{\mathcal T}_{\eta+1}$. Indeed, it lies below the measurable $i^{\mathcal T}_{\eta,\eta+1}(\mu_\eta)$, so its successor cardinal exists there. If $\eta'>\eta+1$, use the backward-edge argument in Lemma~\ref{L3-2}. If $\eta'=\eta+1$, the extender is total on its own model. A proper sequence extender gives its normal measure in that model. For a top extender the larger measurable bounds the successor of its critical point strictly and lies at or below its natural length, so the same initial segment condition again gives the normal measure on the sequence below it.
        Thus for all $\alpha<\mu_\eta$,
        $$\mathcal M^{\mathcal T}_{\eta+1}\vDash 
        ``\text{There is a measurable cardinal in }
        (\alpha,i^{\mathcal T}_{\eta,\eta+1}(\mu_\eta))".$$
        Pulling back by elementarity, we have
        $$\mathcal M^{\mathcal T}_{\eta}\vDash 
        ``\text{There is a measurable cardinal in }
        (\alpha,\mu_\eta)".$$
        Therefore there are unboundedly many measurables in 
        $\mathcal M^{\mathcal T}_{\eta}$ below 
        $\mu_\eta$.
        Pushing forward to 
        $\mathcal M_X$,
        there are unboundedly many measurables below 
        $\mu$. 
        Since $\mathcal M_X$ 
        agrees with $\overline W$, 
        it is also true in $\overline W$. 
        A contradiction.
    \end{proof}
    Taking $\zeta_0$ to be the first member of $I_0$, we obtain $I_0 = \{\zeta_0+i\mid i<\omega\}$ and 
    $\theta = \zeta_0+\omega$. 
    The iteration from $\mathcal M^{\mathcal T}_{\zeta_0}$ to
    $\mathcal M_X$ is the linear
    iteration in which
    the images of one measure are applied repeatedly. 
    We now define $B$ in each of the four cases.

    \subsubsection*{Case 1.} Suppose
    $\mathcal M_X = \mathcal Q_X$ 
    and it is a set mouse.
    Let $n<\omega$ be such that $\mathcal M_X$ is $n+1$-sound above $\mu$. 
    Recall that the $\mu_{\zeta_0}$-soundness of 
    $\mathcal M^{\mathcal T}_{\zeta_0}$ tells us
    $$\mathcal M^{\mathcal T}_{\zeta_0} = 
    \mathrm{Hull}^{\mathcal M^{\mathcal T}_{\zeta_0}}_{n+1}
    (\mu_{\zeta_0}\cup 
    [p_{n+1}(\mathcal M^{\mathcal T}_{\zeta_0})\setminus
    \mu_{\zeta_0}]).$$
    Since $E^{\mathcal T}_{\zeta_0}$ is equivalent to a
    total measure on $\mu_{\zeta_0}$, we have
    $$\mathcal M^{\mathcal T}_{\zeta_0+1} = 
    \mathrm{Hull}^{\mathcal M^{\mathcal T}_{\zeta_0+1}}_{n+1}
    ((\mu_{\zeta_0}+1)\cup
    i^{\mathcal T}_{\zeta_0,\zeta_0+1}
    (p_{n+1}(\mathcal M^{\mathcal T}_{\zeta_0})\setminus\mu_{\zeta_0})).$$
    The same argument applies to all other members of $I_0$.
    Therefore,
    $$\mathcal M_X = 
    \mathrm{Hull}^{\mathcal M_X}_{n+1}
    (\mu_{\zeta_0}\cup \overline C_0\cup
    [p_{n+1}(\mathcal M_X)\setminus
    \mu]).$$
    We let 
    $$B = \nu\cap \mathrm{Hull}^{\mathcal S_X}_{n+1}
    (\pi(\mu_{\zeta_0})\cup C_0
    \cup [p_{n+1}(\mathcal S_X)\setminus
    \nu]).$$
    Since $\mathcal S_X\lhd W$, in $W$ we may tabulate the values below $\nu$ of its Skolem functions on finite tuples from $\nu$ with the displayed parameter fixed.
    This table belongs to $H_{(\nu^+)^W}^W\subseteq K$. Thus $B\in K[C_0]$.
    There are at most $\max(\aleph_0,|\pi(\mu_{\zeta_0})|)$ such inputs in $K[C_0]$, where $\nu$ remains a cardinal by Prikry forcing. Hence $|B|^{K[C_0]}<\nu$.
    Also, since $\Pi\upharpoonright\mu = 
    \pi\upharpoonright\mu$,
    it follows that 
    $$X\cap\nu = \pi[\overline W\cap \mu] 
    = \Pi[\mathcal M_X\cap \mu]
    \subseteq B.$$

    \subsubsection*{Case 2.} Suppose
    $\mathcal M_X\neq\mathcal Q_X$ 
    and both are set mice. 
    We are in the protomouse case and thus $\mathcal M_X$ is type II with $m_X(\mu)=0$.
    By a similar argument as in Case 1, we still have
    $$\mathcal M_X = 
    \mathrm{Hull}^{\mathcal M_X}_{1}
    (\mu_{\zeta_0}\cup \overline C_0\cup
    [p_{1}(\mathcal M_X)\setminus
    \mu]).$$
    To make use of $\mathcal S_X$,
    we need to transfer the above equality to 
    $\mathcal Q_X$.
    Put $k=n_X(\mu)$ and use $q_\mu=(p_{k+1}(\mathcal Q_X)\setminus\mu)\cup t_\mu$ as in Chapter~\ref{Ch2}.
    \begin{claim}\label{c3-6-2}
        For all sufficiently large $n<\omega$,
        \begin{equation}\label{E3-6}
        \mu\cap \mathrm{Hull}
        ^{\mathcal M_X}_{1}
        (\mu_{\zeta_0+n}
        \cup \overline C_0\cup [p_{1}(\mathcal M_X)
        \setminus \mu]) \subseteq 
        \mu\cap \mathrm{Hull}
        ^{\mathcal Q_X}_{k+1}(\mu_{\zeta_0+n}\cup \overline C_0\cup q_\mu).
    \end{equation}
    \end{claim}
    Indeed, choose $n$ so that $\xi=\mu_{\zeta_0+n}$ meets the five bounds of Lemma~\ref{LA-1}. Apply its forward inclusion with the additional ordinal parameters from $\overline C_0$, as proved immediately after that lemma.
    With this $n$, let
    $$B = \nu\cap \mathrm{Hull}^{\mathcal S_X}_{k+1}
    (\pi(\mu_{\zeta_0+n})\cup C_0
    \cup \Pi(q_\mu)).$$
    The conclusion now follows as in Case 1.

    \subsubsection*{Case 3.} Suppose that 
    $\mathcal M_X = \mathcal Q_X$ 
    and it is a weasel. 
    As in Case~3 of Lemma~\ref{L3-4}, let $j:W\to\mathcal S_X=\mathrm{Ult}(W,G)$ be the ultrapower map, with $G\in W$ and $\kappa=\mathrm{crit}(j)$.

    Take
    $$B=\nu\cap\{j(f)(a,s): f\in W,\ f:\kappa^{|a|+|s|}\to\kappa,\ 
        a\in\pi(\mu_{\zeta_0})^{<\omega},\ s\in C_0^{<\omega}\},$$
    where the arguments in the evaluation are concatenated.
    In $W$, the short extender $G$ computes the ordinal values below $\nu$ of these evaluations for all finite tuples from $\nu$. The resulting table belongs to $H_{(\nu^+)^W}^W\subseteq K$: there are only $(2^\kappa)^W<\nu$ functions and $\nu$ finite tuples. Restricting this table to the indicated tuples gives $B\in K[C_0]$. In that model the number of allowed inputs is at most
    $$\max((2^\kappa)^W,|\pi(\mu_{\zeta_0})|,\aleph_0)<\nu.$$

    Fix $\alpha<\mu$. Expanding representatives along the non-dropping branch gives
    $$\alpha=i^{\mathcal T}_{0,\theta}(F)(a,s),\qquad
      a\in\mu_{\zeta_0}^{<\omega},\quad s\in\overline C_0^{<\omega},\quad F\in W.$$
    Indeed, by the predecessor rule, the generators used before $\zeta_0$ lie below $\mu_{\zeta_0}$ and are fixed by the subsequent branch maps. Each step from $\zeta_0$ onward uses a normal measure and adds only its critical point. Successor representations compose, and every element of a direct limit comes from an earlier stage.
    Applying $\Pi$ gives $\pi(\alpha)=j(F)(\pi(a),\pi(s))$, since $j=\Pi\circ i^{\mathcal T}_{0,\theta}$. All coordinates and the value are below $\nu<j(\kappa)$. Restrict $F$ in $W$ to tuples from $\kappa$ of this arity and use value $0$ wherever $F$ has no value below $\kappa$. The resulting $f$ still gives $j(f)(\pi(a),\pi(s))=\pi(\alpha)$, so $X\cap\nu\subseteq B$.

    \subsubsection*{Case 4.} Suppose $\mathcal M_X\ne\mathcal Q_X$ and $\mathcal Q_X$ is a weasel.
    By Case~1 applied to the type II set mouse $\mathcal M_X$,
    $$\mathcal M_X=\mathrm{Hull}^{\mathcal M_X}_1
        (\mu_{\zeta_0}\cup\overline C_0\cup[p_1(\mathcal M_X)\setminus\mu]).$$
    Use $G,j,\kappa$ and the thick class $\Gamma$ of common fixed points from Case~4 of Lemma~\ref{L3-4}.
    \begin{claim}
        For all sufficiently large $n<\omega$,
        \begin{equation}\label{E3-7}
        \begin{aligned}
        \mu\cap\mathrm{Hull}^{\mathcal M_X}_1
          (\mu_{\zeta_0+n}\cup\overline C_0\cup[p_1(\mathcal M_X)\setminus\mu])
        \\ {}\subseteq\mu\cap\mathrm{Hull}^{\mathcal Q_X}
          (\mu_{\zeta_0+n}\cup\overline C_0\cup t_\mu\cup\Gamma).
        \end{aligned}
        \end{equation}
    \end{claim}
    This is the weasel case of the forward inclusion following Lemma~\ref{LA-1}, again with the additional parameters from $\overline C_0$.
    Fix such an $n$. In the definition of $B$ in Case~3 use $a\in\pi(\mu_{\zeta_0+n})^{<\omega}$ and include the fixed tuple $\Pi(t_\mu)$ as a final argument of $j(f)$, increasing the arity of $f$ accordingly. The table with this tuple fixed is again in $H_{(\nu^+)^W}^W$, so $B\in K[C_0]$ and $|B|^{K[C_0]}<\nu$.

    For each $\alpha<\mu$, \eqref{E3-7} gives a representation
    $$\alpha=\tau^{\mathcal Q_X}[a,s,t_\mu,\gamma],$$
    where $a$ is a finite tuple from $\mu_{\zeta_0+n}$, $s$ from $\overline C_0$, and $\gamma$ from $\Gamma$.
    As in Case~3, take $f$ in $W$ to agree with $\tau^W[-,-,-,\gamma]$ on tuples from $\kappa$ where its value lies below $\kappa$, and to have value $0$ elsewhere.
    Applying $\Pi$ to the displayed equation and using $j(\gamma)=\gamma$ gives
    $$\pi(\alpha)=\tau^{\mathcal S_X}[\pi(a),\pi(s),\Pi(t_\mu),\gamma]
      =j(f)(\pi(a),\pi(s),\Pi(t_\mu))\in B.$$
    The second equality uses $\Pi(t_\mu)\subseteq j(\kappa)$ from Lemma~\ref{L2-6} and $\pi(\alpha)<\nu<j(\kappa)$. Hence $X\cap\nu\subseteq B$.
\end{proof}
This completes the proof of Theorem \ref{T3-1}.

\ChapterLabel{Ch4}
\chapter{Optimality examples}

We give two forcing constructions illustrating the roles of the hypotheses
in Theorems~\ref{T1-2} and~\ref{T1-3}. The stationary-tower example is due
to the first author, and the tree-Prikry example is due to Tom Benhamou and the first author.
The first example separates the existence of a Prikry sequence from the
existence of a maximal one.  The second separates maximality from the
covering property of the corresponding Prikry extension.  Throughout this
chapter, a Prikry sequence for a measure belonging to an inner model is
understood relative to that model.

\section{The stationary tower example}

\begin{theorem}\label{T4-1}
Suppose that $U$ is a normal measure over $\nu$ and $\delta>\nu$ is a
Woodin cardinal.  There is a forcing extension in which $\nu$ is a strong
limit cardinal, $U$ has a Prikry sequence over the ground model, and $U$
has no maximal Prikry sequence over the ground model.
\end{theorem}

\begin{proof}
Write $H=H_{\nu^+}$ and regard
$$\mathcal N=(H,\in,\nu,U)$$
as a structure with a predicate for $U$.  In particular, we do not require
$U$ itself to belong to $H$.  We first construct a stationary collection
of elementary substructures of $\mathcal N$ of size $\nu$.

Fix $X\prec\mathcal N$ with $|X|<\nu$.  Starting with $X_0=X$, recursively
define, for $\alpha<\nu$,
$$
 d_\alpha=\min\bigcap(U\cap X_\alpha),\qquad
 X_{\alpha+1}=\{f(d_\alpha):f\in X_\alpha\text{ and }\mathrm{dom}(f)=\nu\},
$$
and take unions at limit stages.  The intersection defining $d_\alpha$
is nonempty because $|X_\alpha|<\nu$ and $U$ is $\nu$-complete.
The usual ultrapower calculation gives
$$X_\alpha\prec\mathcal N,\qquad X_\alpha\subseteq X_{\alpha+1}.$$
Here is the calculation in a form that will also establish stationarity.
For a formula $\varphi$ and functions $f_1,\ldots,f_k\in X_\alpha$, the set
$$
 \{\xi<\nu:\mathcal N\vDash\varphi(f_1(\xi),\ldots,f_k(\xi))\}
$$
belongs to $X_\alpha$. For any such set $A$, either $A$ or
$\nu\setminus A$ belongs to $U\cap X_\alpha$, so $d_\alpha\in A$
if and only if $A\in U$. This is
\L o\'s's calculation for the realization map.  Constant functions show
that the previous range is contained in the next range.

Let $M_\alpha$ be the transitive collapse of $X_\alpha$, and write
$\nu_\alpha,U_\alpha$ for its distinguished cardinal and measure.
The realization calculation identifies
$$
 M_{\alpha+1}=\mathrm{Ult}(M_\alpha,U_\alpha),
$$
with the usual direct limit interpretation at limit stages.  These
ultrapowers are wellfounded, since their realization maps embed them in
$\mathcal N$.  If $\pi_\alpha:M_\alpha\longrightarrow\mathcal N$ is the
realization map, then
$$\pi_\alpha=\pi_\beta\circ i_{\alpha,\beta}\quad(\alpha\leq\beta\leq\nu).$$
Moreover, $d_\alpha$ is above every member of $X_\alpha\cap\nu$:
every final segment of $\nu$ with endpoint in $X_\alpha$ belongs to
$U\cap X_\alpha$.  Thus the critical points $\nu_\alpha$ increase
strictly, and
$$\nu_\nu=\sup_{\alpha<\nu}\nu_\alpha=\nu.$$
Put $Y(X)=X_\nu$.

The collection
$$S=\{Y(X):X\prec\mathcal N\text{ and }|X|<\nu\}$$
is stationary in $[H]^\nu$.  To verify this, fix a function
$F:H^{<\omega}\longrightarrow H$ and choose $X$ elementary in the
expansion of $\mathcal N$ by $F$.  The same realization construction
preserves elementarity in this expansion.  Indeed, its truth sets are
subsets of $\nu$ and hence belong to $H$, and the pointwise application
of $F$ to finitely many $H$-valued functions on $\nu$ is again a member
of $H$.  Consequently $Y(X)$ is closed under $F$.

Now force with the full stationary tower $\mathbb P_{<\delta}$ below
$S$, and let
$$j:V\longrightarrow W\subseteq V[G]$$
be the generic ultrapower.  We use Woodin's stationary tower theorem:
the ultrapower is wellfounded and
$$ {}^{<\delta}W\cap V[G]\subseteq W.$$
We also use the standard seed identity
$$j[H]\in j(S).$$
These facts can be found in \cite[Fact~1.7 and Theorem~1.16]{LOmega}.

Every member $Y$ of $S$ satisfies $\mathrm{type}(Y\cap\nu)=\nu$.
Applying this observation inside $W$ to $j[H]\in j(S)$ yields
$$
 j(\nu)=\mathrm{type}(j[H]\cap j(\nu))
        =\mathrm{type}(j[\nu])=\nu.
$$
The closure of $W$ implies that $W$ and $V[G]$ have the same subsets
of $\nu$, the same functions from subsets of $\nu$ into $\nu$, and
the same sequences of subsets of $\nu$ of length less than $\nu$.
Therefore $j(U)$ is a normal measure over $\nu$ in $V[G]$.
In particular, $\nu$ is still measurable there.

On the substructure of $j(\mathcal N)$ with universe $j[H]$, the
measure predicate is $j(U)\cap j[H]=j[U]$, since $U\subseteq H$.
Its transitive collapse is therefore $\mathcal N$, with uncollapse map
$j\upharpoonright H$. Since $j[H]\in j(S)$ and $j(\nu)=\nu$, the definition
of $S$, applied in $W$, supplies an iteration
$$\langle Q_\alpha,k_{\alpha,\beta}:\alpha\leq\beta\leq\nu\rangle$$
whose terminal structure is $\mathcal N$ and whose strictly increasing
critical sequence $\langle\gamma_\alpha:\alpha<\nu\rangle$ is cofinal
in $\nu$.  This iteration belongs to $V[G]$.
For each $A\in U$ there is $\alpha_0<\nu$ such that
$$\gamma_\alpha\in A\qquad(\alpha_0\leq\alpha<\nu).$$
To see this, represent $A$ at a stage below the limit $\nu$.  At every
later stage the corresponding set belongs to the current normal
measure, so the next critical point belongs to its ultrapower image;
all subsequent embeddings fix that critical point.

Over $V[G]$, force with ordinary Prikry forcing for $j(U)$.  In the
resulting model $V[G][C]$, the cardinal $\nu$ is a strong limit of
cofinality $\omega$, since Prikry forcing preserves cardinals and adds
no bounded subsets of $\nu$.  Every increasing cofinal
$f:\omega\longrightarrow\nu$ now gives a Prikry sequence
$$\{\gamma_{f(n)}:n<\omega\}$$
for the original $U$ over $V$.

Finally, let $E$ be any Prikry sequence for $U$ in $V[G][C]$.
The range of the critical sequence has order type $\nu$, whereas $E$
is countable.  Since $\nu$ remains an uncountable cardinal, every tail
of the critical sequence has a member outside $E$.  Recursively choose
an increasing cofinal $f:\omega\longrightarrow\nu$ with
$\gamma_{f(n)}\notin E$ for all $n$.  The resulting Prikry sequence
is disjoint from $E$.  Hence $E$ is not maximal.
\end{proof}

The forcing over $V[G]$ uses the measure $j(U)$. In the resulting
extension, every Prikry sequence for the original $U$ over $V$ has a
disjoint Prikry sequence for the same measure.

\section{The tree-Prikry example}
Identify each member of $[\nu]^2$ with its increasing enumeration $(a,b)$.
For an ultrafilter $D$ on $[\nu]^2$, a condition in the tree-Prikry
forcing $\mathrm{Pr}(D)$ is a pair $(s,T)$. Here $T$ is a tree of finite
sequences of pairs, closed under initial segments. Every node is
comparable with the trunk $s\in T$ under end-extension, and for every
$t\in T$ extending $s$,
$$\{d\in[\nu]^2:t^\frown\langle d\rangle\in T\}\in D.$$
An extension $(s',T')\leq(s,T)$ has $s'$ extending $s$ in $T$ and
$T'\subseteq T$. It is a direct extension, written $\leq^*$, if $s'=s$.
The generic sequence is the union of the trunks in the generic filter.

\begin{theorem}\label{T4-2}
Let $M$ be a transitive class model of ZFC and $U$ a normal
measure over $\nu$ in $M$. Let
$$i:M\longrightarrow M'=\mathrm{Ult}(M,U).$$
Let $\nu<\nu'<i(\nu)$, let $U'$ be a normal measure over
$\nu'$ in $M'$, and put
$$i':M'\longrightarrow\mathrm{Ult}(M',U'),\qquad j=i'\circ i.$$
Define
$$
 D=\{A\in M:A\subseteq[\nu]^2\text{ and }\{\nu,\nu'\}\in j(A)\}.
$$
If
$$C=\langle(a_n,b_n):n<\omega\rangle$$
is $M$-generic for $\mathrm{Pr}(D)^M$, put
$$C_{\mathrm{even}}=\{a_n:n<\omega\},\qquad
  C_{\mathrm{odd}}=\{b_n:n<\omega\}.$$
Then $C_{\mathrm{even}}$ is a maximal Prikry sequence for $U$ in
$M[C]$, but there is no $B\in M[C_{\mathrm{even}}]$ such that
$$C_{\mathrm{odd}}\subseteq B\quad\hbox{and}\quad
  |B|^{M[C_{\mathrm{even}}]}<\nu.$$
\end{theorem}

We first describe $D$ in the forcing ground model $M$.
Choose representatives, with $f:\nu\longrightarrow\nu$,
$$\nu'=[f]_U,\qquad U'=[\alpha\mapsto V_\alpha]_U.$$
After restricting to a set $A_0\in U$, we have, for every $\alpha\in A_0$,
$$\alpha<f(\alpha)<\nu,$$
and $V_\alpha$ is a normal measure over the measurable cardinal
$f(\alpha)$.  For $A\subseteq[\nu]^2$ in $M$, write
$$A_\alpha=\{\beta\in(\alpha,f(\alpha)):(\alpha,\beta)\in A\}.$$
The two ultrapower calculations give
\begin{equation}\label{E4-1}
 A\in D\quad\Longleftrightarrow\quad
 \{\alpha\in A_0:A_\alpha\in V_\alpha\}\in U.
\end{equation}
Indeed, since $i'$ fixes $\nu$, the seed condition says that the
$\nu$-section of $i(A)$, restricted to $\nu'$, belongs to $U'$;
applying \L o\'s's theorem
to $i$ gives precisely \eqref{E4-1}.
In particular, $D$ belongs to $M$ and is $\nu$-complete.  For the
completeness assertion, given fewer than $\lambda<\nu$ many members
of $D$, first restrict to $\alpha>\lambda$ and intersect their
$U$-large sets of acceptable first coordinates.  Each corresponding
$V_\alpha$ is $\lambda$-complete, since $f(\alpha)>\alpha>\lambda$,
so their sectionwise intersection still belongs to $V_\alpha$.

By \eqref{E4-1}, the first-coordinate projection of $D$ is $U$.
In particular, $\{(a,b):a>\xi\}\in D$ for every $\xi<\nu$. Below any
condition we can thus thin the tree at each node so that each new pair
starts above all preceding coordinates. We work below such conditions.
To use the ordinal presentation of \cite[Definitions~4.2 and~4.4]{B},
enumerate $[\nu]^2$ by second coordinate, breaking ties by first
coordinate. This order has type $\nu$. After the fixed trunk is removed,
the remaining tree has increasing ordinal codes, and its extensions are
exactly the extensions of the coded tree for the transported ultrafilter.

For every $A\in U$, a direct refinement restricts all future first
coordinates to $A$. Thus $C_{\mathrm{even}}$ is eventually in every
member of $U$, including every final segment of $\nu$. Its increasing
enumeration is cofinal of order type $\omega$, so it is an ordinary
Prikry sequence for $U$ over $M$ by the Mathias criterion.

\begin{lemma}\label{L4-3}
Suppose $f,A_0\in M$, $A_0\in U$, and for every $\alpha\in A_0$,
$f(\alpha)$ is a regular cardinal of $M$ with $\alpha<f(\alpha)<\nu$.
Let $a=\langle a_n:n<\omega\rangle$ be an ordinary Prikry sequence
for $U$ over $M$. If $B\in M[a]$, $B\subseteq\nu$, and
$|B|^{M[a]}<\nu$, then there is $h\in M$ with
$h(\alpha)<f(\alpha)$ for $\alpha\in A_0$ such that
$$\sup(B\cap f(a_n))<h(a_n)$$
for all sufficiently large $n$.
\end{lemma}

\begin{proof}
Choose a name $\dot B$, a cardinal $\lambda<\nu$, and a condition
$p=(s_0,A)$ in the generic filter forcing
$\dot B\subseteq\nu$ and $|\dot B|\leq\lambda$.
We can require $A\subseteq A_0\setminus(\lambda+1)$.
The following construction works below every such condition, so its
resulting conditions form a dense set below $p$.

For each finite stem $s$ extending $s_0$ by points of $A$, with a
nonempty added part and last point $\alpha$, the condition with stem
$s$ forces
$$\sup(\dot B\cap f(\alpha))<f(\alpha).$$
Indeed, ordinary Prikry forcing preserves the regularity of
$f(\alpha)<\nu$, and $\lambda<\alpha<f(\alpha)$.
The Prikry property and the $\nu$-completeness of $U$ give a direct
extension $(s,A_s)$ and an ordinal $h_s<f(\alpha)$ such that
$$ (s,A_s)\Vdash\sup(\dot B\cap f(\alpha))<h_s.$$
For a fixed $\alpha$, there are at most $|\alpha|+\aleph_0$
relevant stems ending in $\alpha$.  By the regularity of $f(\alpha)$,
we may choose $h(\alpha)<f(\alpha)$ strictly above all their $h_s$.
Define $h$ arbitrarily below $f(\alpha)$ at the remaining points of
$A_0$.

Normality gives $A^*\in U$, $A^*\subseteq A$, such that
$$\beta\in A^*,\quad\max(s)<\beta\quad\Longrightarrow\quad\beta\in A_s$$
for every relevant stem $s$.  Explicitly, for each $\xi<\nu$ first
intersect the fewer than $\nu$ many sets $A_s$ with
$\max(s)\leq\xi$, and then take their normal diagonal intersection.
The condition $(s_0,A^*)$ forces the desired inequality at every new
Prikry point: if $s$ ends at that point, the entire remaining generic
sequence belongs to $A_s$, so $(s,A_s)$ belongs to its generic filter.
Since the construction can be performed below every extension of $p$,
one of these strengthened conditions belongs to the given generic
filter.  Its associated ground-model function $h$ proves the lemma.
\end{proof}

\begin{lemma}\label{L4-4}
In the situation of Theorem~\ref{T4-2}, no set
$B\in M[C_{\mathrm{even}}]$ of cardinality less than $\nu$ in that
model contains $C_{\mathrm{odd}}$.
\end{lemma}

\begin{proof}
Replacing $B$ by $B\cap\nu$, apply Lemma~\ref{L4-3}
and fix the resulting $h\in M$.  By \eqref{E4-1},
$$
 \{(\alpha,\beta):\alpha\in A_0,
           h(\alpha)<\beta<f(\alpha)\}\in D.
$$
A tree-Prikry generic sequence is eventually contained in each
ground-model member of $D$: the conditions whose successor sets are
all contained in that member form a dense set.
Hence, for every sufficiently large $n$,
$$\sup(B\cap f(a_n))<h(a_n)<b_n<f(a_n).$$
In particular, $b_n\notin B$.
\end{proof}

We will also use the following description of a condition. At a node $t$ whose
length is an integral number of pairs, the allowed next pairs have the form
$$a\in A_t,\qquad b\in B_{t,a},$$
where $A_t\in U$, $A_t\subseteq A_0$, and
$B_{t,a}\in V_a$ consists of ordinals in $(a,f(a))$. Every condition can be
directly refined to have this form: discard the first coordinates whose
vertical sections are not in $V_a$.  After any further paired-tree
refinement, discard such first coordinates again as needed.
This permits us to view its tree as an
alternating tree of individual ordinals. At a pair boundary the successor
measure is $U$; after choosing $a$, it is $V_a$. Refining this alternating tree
at any of its nodes gives a valid direct refinement of the paired tree.

We use two elementary facts about finite fat trees. First, if all successor
ultrafilters in a finite-height tree are $\delta^+$-complete, then any
coloring of its terminal nodes by at most $\delta$ colors becomes constant
after shrinking the tree without changing its root. This follows by induction
on the height: first make the coloring constant above each immediate successor,
and then use completeness at the root to make those constants agree. In
particular, a finite coloring can be made constant in any of the trees above.
Second, given $A_t\in U$ for each finite paired stem $t$ in a tree, there is
$A\in U$ such that
\begin{equation}\label{E4-2}
 a\in A\text{ and }\max(t)<a\quad\Longrightarrow\quad a\in A_t
\end{equation}
for every nonempty such $t$; the set $A$ can also be required to lie in the
successor set at an empty stem. To see this, set
$$H_\gamma=\bigcap\{A_t:t\text{ is such a stem and every coordinate of }t
 \text{ is at most }\gamma\}.$$
There are fewer than $\nu$ sets in this intersection, so $H_\gamma\in U$.
The normal diagonal intersection $\mathop{\triangle}_{\gamma<\nu}H_\gamma$
has the required property. This construction applies to stems including the
fixed stem of the initial condition.

Let $\dot C_{\rm even}$ be the name for the sequence of first coordinates.

\begin{lemma}\label{L4-5}
Suppose that $p\in\mathrm{Pr}(D)$ and
$$p\Vdash\dot x<\nu\text{ and }\dot x\notin\dot C_{\rm even}.$$
Then there are a direct extension $q\leq^*p$ and a set $A\in U$ such that
$$q\Vdash\dot x\notin A.$$
\end{lemma}

\begin{proof}
Choose $\gamma<\nu$ above every coordinate of the stem of $p$, and refine
so that all new first coordinates exceed $\gamma$. By the strong tree-Prikry
property, applied to the dense set of conditions deciding $\dot x$, there is a
direct refinement $p_0=(s,T_0)$ and an integer $n$ such that, for every extension
$t$ of $s$ by exactly $n$ pairs in $T_0$, the condition $(t,(T_0)_t)$ decides
$\dot x$. This is the strong property in \cite[Lemma 4.17]{B}, applied to $D$.
Let $F(t)<\nu$ be the value so decided. If $n=0$, the conclusion follows by
removing this one ordinal from $\nu$, so assume $n>0$.

For a leaf $t$, write the newly chosen coordinates as
$$a_0<b_0<\cdots<a_{n-1}<b_{n-1}.$$
Color $t$ according to whether $F(t)\leq\gamma$, whether it equals one of these
coordinates, or in which of the intervening open intervals (including the
interval above $b_{n-1}$) it lies. Shrink the finite tree so that this color is
constant, retaining the tails above its leaves. This is a direct refinement of
$p_0$ and preserves all decisions. Call the resulting tree $T$.

If the common color is $F(t)\leq\gamma$, take $A=\nu\setminus(\gamma+1)$.
The color $F(t)=a_i$ is impossible, because the condition at $t$ forces
$a_i\in\dot C_{\rm even}$ and $\dot x\notin\dot C_{\rm even}$.

Let
$$E_f=\{\delta<\nu:\delta\text{ is a limit ordinal and }
 (\forall\alpha<\delta)\ f(\alpha)<\delta\}.$$
This is a club of $\nu$, and hence belongs to $U$. If the common color is
$a_i<F(t)\leq b_i$, then
$$a_i<F(t)<f(a_i),$$
so $F(t)\notin E_f$. In this case take $A=E_f$.

It remains to consider a common color of one of the following forms:
$$\gamma<F(t)<a_0,\qquad
 b_{r-1}<F(t)<a_r\ (0<r<n),\qquad
 b_{n-1}<F(t).$$
We claim that the tree can be further refined so that, for every prefix $u$
extending $s$ by exactly $r$ pairs, the condition $(u,T_u)$ decides
$\dot x=y_u$, where $y_u$ exceeds every coordinate of $u$. For the first case
take $r=0$ and require $y_s>\gamma$. For the last case take $r=n$; the claim
then already holds with $y_u=F(u)$.

For the other cases fix such a prefix $u$. For each allowed next first
coordinate $a$, consider the finite alternating tree above $u{}^\frown a$
through the old terminal level. Every value of $F$ on its leaves is less than
$a$. Every successor ultrafilter in this finite tree is $a^+$-complete: the
first one is $V_a$ on $f(a)>a$, and subsequent ones are either $U$ or $V_{a'}$
with $a'>a$. The finite-tree fact therefore gives a refinement on which
$F$ is constantly $z_u(a)<a$. By normality of $U$, after restricting the
allowed $a$'s to a set in $U$, the regressive function $z_u$ has a constant
value $y_u$. Perform these refinements separately above every prefix $u$ at
this level. Their subtrees are disjoint, so together they give a direct
refinement of the original paired tree.

All terminal conditions above $u$ now force $\dot x=y_u$. Since extensions
reaching the terminal level are dense below $(u,T_u)$, that condition itself
forces $\dot x=y_u$. The fixed interval color gives $y_u>\max(u)$ if $u$ is
nonempty, and $y_s>\gamma$ in the first case. This proves the claim.

For the refined tree choose $A\in U$ as in
\eqref{E4-2}, also restricting it above $\gamma$. If some
$y_u$ belonged to $A$, it would be an allowed next first coordinate above $u$.
Choose an allowed second coordinate $b\in B_{u,y_u}$. The condition with stem
$u{}^\frown(y_u,b)$ would then force both
$$\dot x=y_u\qquad\text{and}\qquad y_u\in\dot C_{\rm even},$$
contradicting the hypothesis on $p$. Thus $y_u\notin A$ for every prefix $u$
under consideration. Those prefixes form a full finite level of the tree,
so their conditions are predense below its root. The final direct refinement
therefore forces $\dot x\notin A$.
\end{proof}

\begin{lemma}\label{L4-6}
Let $C=\langle(a_n,b_n):n<\omega\rangle$ be $\mathrm{Pr}(D)$-generic over $M$ and
let $C_{\rm even}=\{a_n:n<\omega\}$. If $X\in M[C]$, $X\subseteq\nu$,
$|X|^{M[C]}<\nu$, and $X\cap C_{\rm even}=\emptyset$, then there is
$A\in U$ with $A\cap X=\emptyset$.
\end{lemma}

\begin{proof}
Choose a condition in the generic filter, an ordinal $\tau<\nu$, and a name
$\langle\dot x_\xi:\xi<\tau\rangle$ which that condition forces to enumerate
$X$ and to be disjoint from $\dot C_{\rm even}$. Below any extension of this
condition, apply Lemma~\ref{L4-5} successively to these names,
taking direct lower bounds at limit stages. Direct extension is
$\nu$-closed, because $D$ is $\nu$-complete. Thus the recursion gives a
condition $q$ and sets $A_\xi\in U$ such that
$$q\Vdash\dot x_\xi\notin A_\xi\qquad(\xi<\tau).$$
The set $A=\bigcap_{\xi<\tau}A_\xi$ belongs to $U$, and $q$ forces
$X\cap A=\emptyset$. Conditions with such a witness $A$ are therefore dense
below the chosen condition. The generic filter meets this dense set.
\end{proof}

\begin{corollary}\label{C4-7}
The sequence $C_{\rm even}$ is a maximal Prikry sequence for $U$ over $M$ in
$M[C]$.
\end{corollary}

\begin{proof}
The first-coordinate projection of $D$ is $U$, so the generic first coordinates
are eventually in every member of $U$. They are eventually increasing and cofinal in
$\nu$. Thus they satisfy the Mathias condition for $U$.

Suppose that $d$ were another Prikry sequence for $U$ over $M$ and that
$X=d\setminus C_{\rm even}$ were infinite. The set $X$ is countable and is
disjoint from $C_{\rm even}$. By Lemma~\ref{L4-6}, some
$A\in U$ misses $X$. This contradicts the Mathias condition for $d$, since
every infinite subsequence of $d$ is eventually in $A$.
\end{proof}

\begin{proof}[Proof of Theorem~\ref{T4-2}]
The sum representation \eqref{E4-1} puts $D$ in the setting
of Corollary~\ref{C4-7}, which proves the maximality of
$C_{\mathrm{even}}$.  Lemma~\ref{L4-4} proves that
$C_{\mathrm{odd}}$ has no cover of the stated size in
$M[C_{\mathrm{even}}]$.
\end{proof}

The ultrafilter $D$ is $\nu$-complete on a set of size $\nu$.
Thus tree-Prikry forcing preserves cardinals and adds no bounded
subsets of $\nu$; see \cite[Corollaries~4.9 and~4.10]{B}.
Consequently $C_{\mathrm{odd}}$ is countable in $M[C]$ while $\nu$
remains a cardinal there, so the failure of covering concerns a set
of size strictly less than $\nu$ in the outer model.

Finally, the existence of $\nu'$ and $U'$ in the theorem is equivalent
to the measurable cardinals of $M$ being unbounded in $\nu$.
For the forward implication, for each $\alpha<\nu$ the cardinal $\nu'$
witnesses in $M'$ that there is a measurable between $i(\alpha)=\alpha$
and $i(\nu)$, and elementarity reflects this to $M$.
Conversely, choose in $M$ a measurable $f(\alpha)\in(\alpha,\nu)$
and a normal measure $V_\alpha$ on it, for each $\alpha<\nu$.
Then $[f]_U$ and $[\alpha\mapsto V_\alpha]_U$ have the required
properties. Suppose in particular that $M=K^M$, there is no inner
model with a Woodin cardinal in $M$, and $U$ is its unique normal
measure on $\nu$. By generic absoluteness, $K^{M[C]}=M$
\cite[Theorem~1.1(3)]{JS}. Every maximal Prikry sequence for $U$ over
$M$ in $M[C]$ differs only finitely from $C_{\mathrm{even}}$ and hence
generates the same intermediate model. Thus $C_{\mathrm{odd}}$ witnesses
the failure of the covering conclusion of Theorem~\ref{T1-3} when the
measurable cardinals below $\nu$ are unbounded.

When these measurable cardinals are unbounded, is there another
criterion for choosing $C\subseteq\nu$ canonically relative to $K$
so that $K[C]$ satisfies the covering conclusion of Theorem~\ref{T1-3}?
\ChapterLabel{Ch5}
\chapter{Unique Normal Measure without Closure Assumption}

This section is for the proofs of Theorems \ref{T1-2} and \ref{T1-3}
with no additional assumptions.
We first state their conclusions together below.
\begin{theorem}\label{T5-1}
    Assume there is no inner model with a Woodin cardinal.
    Suppose that in the Jensen-Steel core model $K$, 
    $\nu>\omega_2$ is a cardinal such that 
    $$\mathrm{cf}(\nu)<|\nu|\leq\nu = \mathrm{cf}^K(\nu),$$
    and $o^K(\nu) = 1$.
    By Theorem \ref{T1-1}, $\mathrm{cf}(\nu) = \omega$.
    Then there exists a maximal Prikry sequence $C_0$ 
    over $U(K,\nu,0)$. 
    Moreover, if measurable cardinals in $K$ are bounded in $\nu$, 
    then for every $A\subseteq \nu$ with $|A|^V<|\nu|^V$, 
    there exists $B\in K[C_0]$ such that
    $B\supseteq A$ and $|B|^{K[C_0]}<\nu$.
\end{theorem}
We first give a brief overview of our strategy in proving Theorem \ref{T5-1}.
As in Chapter~\ref{Ch3}, we first choose a substructure $X$.
In that chapter, countable closure ensures that $X$ contains every countable sequence of its elements, in particular the increasing enumeration of $C_0$.
Without countable closure of $|\nu|$, we use the construction of $\mathcal F$ from Chapter~\ref{Ch2}
to choose a weakly internal approachable chain $\langle Y_i\mid i<\omega_1\rangle$ with limit $X$ and $Y_i\in X$ for every $i<\omega_1$.
We will analyze not only the coiteration process defined by $X$ but also by each $Y_i$.
It turns out that $\mathcal M_X$ and
$W_X$ share the same order 0 total measure of $\mu_X$, and an end segment of the critical sequence of $X$ is a Prikry sequence for $U(W,\nu,0)$.
This sequence is also an end segment of the critical sequence of $Y_i$ for some $i<\omega_1$, therefore it is a member of $X$.
Maximality then follows as in Lemma~\ref{L3-4}. 

Recall that in Lemma \ref{L3-2}, the hypothesis $o^K(\nu)=1$ gives us the property that
$$E^{\mathcal T_X}_\eta\simeq U(\mathcal M^{\mathcal T_X}_{\eta},\mu_\eta,0)$$
for almost all $\eta\in I_X$.
The following lemma is a similar result and it is the first step in our proof of Theorem \ref{T5-1}.
 
\begin{lemma}\label{L5-2}
    Under the hypothesis of Theorem \ref{T5-1}, there exists a function $F$ on $H_{\Omega^+}$ such that,
    if we let 
    $$\mathcal C_F = \{X\prec H_{\Omega^+}\mid X\mbox{ is closed under }F\},$$
    then for every $X\in \mathcal C_F\cap \mathcal F$,
    there exists a $\zeta_0^X\in [\zeta_X,\theta_X)_{\mathcal T_X}$ such that
    $$E^{\mathcal T_X}_{\eta} = 
    E(\mathcal M^{\mathcal T_X}_{\eta},\mu^X_\eta,0)$$
    and $\eta+1\in [\zeta_X,\theta_X)_{\mathcal T_X}$ 
    for all $\eta\in I_X\setminus \zeta^X_0$. 

\end{lemma}

\begin{proof}
    Suppose otherwise. 
    Let $\overline{\mathcal F}$ be the collection of $X\in\mathcal F$
    for which such $\zeta^X_0$ does not exist.
    Then $\overline{\mathcal F}$ is a
    stationary subset of $\mathcal F$.
    Recall that by Lemma \ref{L2-10}, 
    $I_X\subseteq [0,\theta_X)_{\mathcal T_X}$ is defined. 
    For each $\eta\in I_X$, 
    let $\mathrm{crit}(i^{\mathcal T_X}_{\eta,\theta_X}) = \mu^X_\eta$.
    Therefore, the hypothesis implies that there is a cofinal subset
    $J_X\subseteq I_X$ such that, 
    for every $\eta\in J_X$, 
    $E^{\mathcal T_X}_{\eta} \neq E(\mathcal M^{\mathcal T_X}_{\eta},\mu^X_\eta,0)$.
    In particular, 
    $o^{\mathcal M^{\mathcal T_X}_{\theta_X}}(\mu_X)\geq 2$.
    Let $D_X = \{\mu^X_\eta\mid \eta\in J_X\}$.
    Replace $\mathcal F$ by $\overline{\mathcal F}$ in the analysis of Chapter~\ref{Ch2} and define $\mathcal G$ and $\mathcal H$ from it.
    Notice that the definition of $\mathcal H$ involves 
    choosing $B_X\subseteq \nu$ for each $X\in \mathcal G$.
    Since we already have $\mathrm{cf}(\nu) = \omega$, 
    we may choose $B_X\subseteq \pi_X[D_X]$ cofinally and let 
    $\mathrm{type}(B_X) = \omega$. 
    The arguments lead to a proof of the following claim.
    \begin{claim}\label{c5-2-1}
        For each $X\in\overline{\mathcal F}$, let $\mathcal S_X$ be as defined in Chapter~\ref{Ch2}, which is the $\mathcal S$-structure corresponds to $X$ and $\mu_X$. Then
        $$U_0 = \bigcup_{X\in\mathcal H} U(\mathcal S_X,\nu,0)$$
        is the Mitchell order $0$ total measure of $K$ on $\nu$.
    \end{claim}
    Notice that the argument framework does not rely on
    the Mitchell order of the measures, as shown by the proof of \cite[Lemma 6.1]{MS2}.
    The measure $U(\mathcal S_X,\nu,1)$ exists,
    because $o^{\mathcal M^{\mathcal T_X}_{\theta_X}}(\mu_X)\geq 2$,
    therefore $o^{\mathcal Q_X}(\mu_X)\geq 2$, 
    and so $o^{\mathcal S_X}(\nu)\geq 2$.
    By the proof of \cite[Lemma 6.2]{MS2}, we have
    $$U_1 = \bigcup_{X\in\mathcal H}U(\mathcal S_X,\nu,1)$$
    is also a total measure of $K$ on $\nu$. 
    Let 
    $$A = \{\alpha<\nu\mid \alpha\text{ is measurable in }K\}.$$
    Then $A\in U_1$ and $A\not\in U_0$.
    Hence, $U_1$ has Mitchell order $\geq 1$. 
    This contradicts the assumption $o^K(\nu) = 1$.
\end{proof}

Without loss of generality, let $\zeta^X_0 = \zeta_X$ for all $X\in\mathcal F\cap \mathcal C_F$.
We use the following form of the pressing-down lemma from \cite[Lemma~1.1]{MS1}.

\begin{lemma}\label{L5-3}
    Let $\varepsilon\geq\omega_2$ be a regular cardinal and let
    $\langle Y_j\mid j<\varepsilon\rangle$ be a continuous internally approachable chain
    of elementary substructures of $H_{\Omega^+}$ with $|Y_j|<\varepsilon$ for each $j$.
    If $\mathrm{dom}(f)$ is stationary in $\varepsilon$ and $f(j)\in Y_j$ for each $j\in\mathrm{dom}(f)$,
    then $f$ is constant on a stationary subset of its domain.
\end{lemma}

Fix a cardinal $\omega_1\leq\kappa<|\nu|$. Since $\mathrm{cf}(\nu)=\omega$, we have $\kappa^+<\nu$.
Choose a continuous internally approachable chain $\langle Y_j\mid j<\kappa^+\rangle$
as in Theorem~\ref{T2-2}, with $\kappa+1\subseteq Y_0$, each $Y_j$ of size $\kappa$,
and each $Y_j$ closed under the function $F$ of Lemma~\ref{L5-2}.
Let $\mathcal C\subseteq\kappa^+$ be the club given by Theorem~\ref{T2-2}.
Let $S$ be the set of $j\in\lim\mathcal C$ of cofinality $\omega_1$ such that
$j=Y_j\cap\kappa^+=\mathrm{type}(\mathcal C\cap j)$.
For each $j\in S$, we have $Y_j\in\mathcal F\cap\mathcal C_F$.
Increase $\zeta_{Y_j}$ if necessary so that $\pi_{Y_j}(\mu_\eta^{Y_j})>\omega_2$ for every $\eta\in I_{Y_j}$.
We temporarily pause the construction and show that $\mathrm{type}(C_{Y_j})=\omega$.

Suppose $I_{Y_j}$ has an $\omega$-th member $\eta$, and write $\beta=\pi_{Y_j}(\mu_\eta^{Y_j})$.
By the continuity of the critical sequence in Lemma~\ref{L2-10}, the images of the first $\omega$ critical points are cofinal in $Y_j\cap\beta$.
Since $\mathrm{cf}(j)=\omega_1$ and the chain is continuous, some $Y_t$ with $t<j$ contains all these points.
By internal approachability $Y_t\in Y_j$, so $\sup(Y_j\cap\beta)=\sup(Y_t\cap\beta)\in Y_j$.
This supremum is $\beta$, since $Y_j\cap\beta$ has no largest element. Thus $\mathrm{cf}^V(\beta)=\omega$.
By Lemma~\ref{L5-2} the extender used at $\eta$ has order zero.
As in the proof of Lemma~\ref{L3-3}, it follows that $\beta$ is regular but not measurable in $K$, contrary to Theorem~\ref{T1-1}.

Therefore, both $I_{Y_j}$ and $C_{Y_j}$ have order type $\omega$. Write
$$I_{Y_j} = \{\eta^{Y_j}_{k}\mid k<\omega\}.$$
For each $j\in S$ and $k<\omega$, we write $\mathcal M^{Y_j}_k$ to denote $\mathcal M^{\mathcal T_{Y_j}}_{\eta^{Y_j}_k}$, and
$$i^{Y_j}_k:\mathcal M^{Y_j}_k\to \mathcal M_{Y_j}$$
denote the tree embedding. 
Let $f:S\to H_{\Omega^+}$ be a function such that $f(j)\cup\{f(j)\}\subseteq Y_j$ and $C_{Y_j}\subseteq f(j)$. 
By Lemma \ref{L5-3}, there exists a stationary $S'\subseteq S$ such that the value of $f(j)$ is fixed for $j\in S'$.
\footnote{The function $f$ is to make sure that $C_{Y_j}\subseteq Y_{j'}$ for any $j,j'\in S'$. Readers are referred to the construction of $\mathcal H$ in Chapter~\ref{Ch2}.}
Find a stationary subset $S''\subseteq S'$ such that the answer to each of the following questions is fixed for $j\in S''$.

\noindent \textbf{(List B)}
\begin{itemize}[label=\textbullet, leftmargin=2em, nosep, topsep=3pt]
    \item Is $\mathcal{M}_{Y_j}$ a set mouse or a weasel?   
    \item What is the value of $m_{Y_j}(\mu_{Y_j})$?
    \item Is $\mathcal{Q}_{Y_j}$ defined by the protomouse case?
    \item Is $\mathcal{Q}_{Y_j}$ a set mouse or a weasel?
    \item What is the value of $n_{Y_j}(\mu_{Y_j})$?
\end{itemize}
Write $m=m_{Y_j}(\mu_{Y_j})$ and $n=n_{Y_j}(\mu_{Y_j})$ for the common values on $S''$. 
Let $\ell\in S''$ with $\ell = \mathrm{type}(\ell\cap S'')$.
Let $X = Y_\ell$. 
Together with $\ell$, we fix a cofinal subsequence 
$\langle j_i\mid i<\omega_1\rangle$ of $\ell\cap S''$. 
We will write $\langle Y_i\mid i<\omega_1\rangle$ instead of $\langle Y_{j_i}\mid i<\omega_1\rangle$, and $\mu^{Y_i}_k$ as $\mathrm{crit}(i^{Y_i}_k)$.

By our construction of $X$, $C_X\subseteq Y_i$ for every $i<\omega_1$ as well as $C_{Y_i}\subseteq Y_j$ for every $j<i<\omega_1$. 
The following lemma adapts \cite[Lemma~4.3]{MS2} to this setting using the correction in the Appendix. 

\begin{lemma}\label{L5-4}
    Let $i<j<\omega_1$.
    Then $C_{Y_j}\subseteq^* C_{Y_i}$.
    Moreover, $C_X\subseteq^* C_{Y_i}$ for every $i<\omega_1$.
\end{lemma}

The following lemma is a generalization of Lemma \ref{L2-3}.
The proof is similar, and we will temporarily omit it.
We will also record the proof in the appendix.

\begin{lemma}\label{L5-5}
    Suppose that $\mathcal M_X\neq \mathcal Q_X$. 
    Then there exists some $k_0<\omega$ such that for all $k>k_0$, one of the following holds:
    \begin{itemize}
        \item $\mathcal Q_X$ is a set mouse, and 
        $$\mathcal P(\mu_X)\cap \mathrm{Hull}^{\mathcal M_X}_1(\mu^X_k\cup p_1(\mathcal M_X))=\mathcal P(\mu_X)\cap \mathrm{Hull}^{\mathcal Q_X}_{n+1}(\mu^X_k\cup q_{\mu_X}).$$
        \item $\mathcal Q_X$ is a weasel. 
        For every thick class $\Gamma$ of ordinals fixed by the embedding
        $$W\xrightarrow{}\mathcal P^X_{\lambda_\ell}\xrightarrow{}\mathcal Q_{X},$$
        we have
        $$\mathcal P(\mu_X)\cap \mathrm{Hull}_1^{\mathcal M_X}(\mu^X_k\cup p_1(\mathcal M_X)) = \mathcal P(\mu_X)\cap \mathrm{Hull}^{\mathcal Q_X}(\mu^X_k\cup t_{\mu_X}\cup\Gamma).$$
    \end{itemize}
    The lemma also holds if $X$ is replaced by $Y_i$ for any $i<\omega_1$.
\end{lemma}

In the following lemma, we will look at the definability of subsets of $\mu_X$ between $\mathcal M_X$ and $\mathcal S_{Y_i X}$.

\begin{lemma}\label{L5-6}
    Suppose that $\mathcal M_X$ is a set mouse and is $m+1$-sound above $\mu_X$. 
    Then there exists some $k_0<\omega$ such that for all $k>k_0$, if $A\in \mathcal P(\mu_X)\cap \mathcal M_X$ such that
    $$A\in \mathrm{Hull}^{{\mathcal M_X}}_{m+1}(\mu^X_k\cup p_{m+1}({\mathcal M_X})),$$
    then there exist unboundedly many $i<\omega_1$ such that
    \begin{itemize}
        \item[1.] If ${\mathcal S_{Y_i X}}$ is a set mouse, then 
        $$A\in \mathrm{Hull}^{{\mathcal S_{Y_i X}}}_{n+1}(\mu^X_k\cup \Pi_{Y_iX}(q_{\mu_{Y_i}})).$$
        \item[2.] If ${\mathcal S_{Y_i X}}$ is a weasel and $\Gamma$ is a thick class, then 
        $$A\in \mathrm{Hull}^{{\mathcal S_{Y_i X}}}(\mu^X_k\cup c({\mathcal S_{Y_i X}})\cup \Gamma).$$
    \end{itemize}
    Suppose that $\mathcal M_X$ is a weasel. 
    Fix a thick class $\Gamma$ of common fixed points of all relevant embeddings.
    Then there exists some $k_0<\omega$ such that for all $k>k_0$, if $A\in \mathcal P(\mu_X)\cap \mathcal M_X$ such that
    $$A\in \mathrm{Hull}^{{\mathcal M_X}}(\mu^X_k\cup\Gamma),$$
    then there exist unboundedly many $i<\omega_1$ such that
    $$A\in \mathrm{Hull}^{{\mathcal S_{Y_i X}}}(\mu^X_k\cup \Gamma).$$
\end{lemma}

\begin{proof}
    We will only demonstrate the proof when $\mathcal M_X$ is a set mouse, and leave the rest to the reader.
    Recall that our construction of $\langle Y_i\mid i<\omega_1\rangle$ and $X$ ensure that $\mathcal M_{Y_i}$ are all set mice.
    Moreover, they are coherent in a way that, for example, if $\mathcal Q_X$ is a weasel, then $\mathcal Q_{Y_i}$ is a weasel, for every $i<\omega_1$.
    We will split the discussion into different cases, based on whether $\mathcal Q_X$ is a set mouse or a weasel, and whether $\mathcal Q_{X}$ is defined by the protomouse case.

    \subsection*{Case 1} 
    Suppose that $\mathcal Q_X = \mathcal M_X$ is a set mouse. 
    Here $m=n$. By \cite[Section~2.3]{MS2},
    \[
        \Pi_{Y_iX}(q_{\mu_{Y_i}})=p_{n+1}(\mathcal S_{Y_iX})\setminus\mu_X.
    \] As in Claim~\ref{c5-6-1}, choose the large enough $k<\omega$ so that
    $p_{n+1}(\mathcal S_{Y_iX})\cap\mu_X\subseteq\mu_k^X$ for unboundedly many $i$.
    By Lemma \ref{L2-8}, for every $i<\omega_1$, either $\mathcal S_{Y_i X}\unlhd \mathcal M_X$, or there exists some $\lambda_i<\mu_X$ and an extender $E_i$ on the $\mathcal M_X$-sequence such that 
    $$\mathcal S_{Y_i X} = \mathrm{Ult}_n(\mathcal P^{X}_{\lambda_i},E_i).$$
    If there exist unboundedly many $i<\omega_1$ such that $\mathcal S_{Y_i X} = \mathcal M_X$, we may take $k_0<\omega$ large enough such that for unboundedly many $i<\omega_1$,
    $$p_{m+1}(\mathcal M_X)\in \mathrm{Hull}_{n+1}^{\mathcal S_{Y_i X}}(\mu^X_{k_0}\cup p_{n+1}(\mathcal S_{Y_i X})),$$
    and the result follows immediately.
    Suppose otherwise and let us assume there exist unboundedly many $i<\omega_1$ such that $\mathcal S_{Y_i X}\lhd \mathcal M_X$.
    The following claim provides the combinatorial property that $k_0$ and $i$ need to satisfy.
    \begin{claim}\label{c5-6-1}
        There exists $k_0<\omega$ such that for every $k>k_0$, there exist unboundedly many $i<\omega_1$ such that
        \begin{itemize}
            \item[1.] $\pi_X(\mu^X_k)\in C_{Y_i}$;
            \item[2.] $\{{\mathcal S_{Y_i X}};p_{m+1}({\mathcal S_{Y_i X}})\}\subseteq \mathrm{ran}(i^{X}_{k})$.
        \end{itemize}
    \end{claim}
    \begin{proof}
    By assumption, $\mathcal S_{Y_iX}\lhd\mathcal M_X$ for unboundedly many $i<\omega_1$. Fix such an $i$. By Lemma~\ref{L5-4},
    $C_X\subseteq^* C_{Y_i}$, so $\pi_X(\mu^X_k)\in C_{Y_i}$
    for all sufficiently large $k$.
    Moreover, both $\mathcal S_{Y_iX}$ and $p_{m+1}(\mathcal S_{Y_iX})$ belong to
    $\mathcal M_X$. The ranges of the maps $i_k^X$ are increasing,
    and their union is $\mathcal M_X$, so both objects belong to
    $\operatorname{ran}(i_k^X)$ for all sufficiently large $k$.
    Thus, for each such $i$ there is a least $k_i<\omega$ such that
    conditions 1 and 2 hold for every $k>k_i$. Let $k_0<\omega$ be the common value of $k_i$ for unboundedly many $i$.
    \end{proof}
    Fix $k>k_0$ and $A\subseteq\mu_X$ such that
    $$A\in \mathrm{Hull}_{m+1}^{\mathcal M_X}
    (\mu_k^X\cup p_{m+1}(\mathcal M_X)).$$
    Since this hull is contained in $\mathrm{ran}(i_k^X)$, write
    $A=i_k^X(A')$. As $\mathrm{crit}(i_k^X)=\mu_k^X$ and
    $i_k^X(\mu_k^X)=\mu_X$, we have $A'=A\cap\mu_k^X$.
    By acceptability, $A'\in J_{\mu_X}^{\mathcal M_X}=J_{\mu_X}^{W_X}$.

    Since $X=\bigcup_{i<\omega_1}Y_i$, we have $\pi_X(A')\in Y_i$
    for all sufficiently large $i$. Fix any such $i$ satisfying
    conditions 1 and 2 of the claim. By condition 1, choose $\ell$ with
    $\pi_{Y_i}(\mu_\ell^{Y_i})=\pi_X(\mu_k^X)$.
    Then $B'=\pi_{Y_iX}^{-1}(A')$ is a subset of $\mu_\ell^{Y_i}$,
    and $\Pi_{Y_iX}(B')=A'$ since $\Pi_{Y_iX}$ agrees with $\pi_{Y_iX}$
    on $J_{\mu_{Y_i}}^{W_{Y_i}}$.
    Since 
    $$\mathcal P(\mu^{Y_i}_\ell)\cap \mathcal M^{Y_i}_\ell = \mathcal{P}(\mu^{Y_i}_\ell)\cap {\mathcal M_{Y_i}},$$
    we have that $B'\in \mathcal P(\mu^{Y_i}_\ell)\cap \mathcal M^{Y_i}_\ell$. 
    Thus,
    $$B'\in \mathrm{Hull}_{m+1}^{{\mathcal M^{Y_i}_\ell}}(\mu^{Y_i}_\ell \cup p_{m+1}({\mathcal M^{Y_i}_\ell})).$$
    Pushing forward to $\mathcal M_{Y_i}$ and let $D=i_\ell^{Y_i}(B')$, we have
    $$D \in \mathrm{Hull}_{m+1}^{{\mathcal M_{Y_i}}}(\mu^{Y_i}_\ell \cup p_{m+1}({\mathcal M_{Y_i}})).$$
    Let $B = \Pi_{Y_i X}(D)$.
    We have
    $$B\in \mathrm{Hull}^{{\mathcal S_{Y_i X}}}_{m+1}(\mu^X_k\cup p_{m+1}({\mathcal S_{Y_i X}})).$$
    The lemma follows if $A = B$. 
    By our assumption on $i$ and $k$, 
    $$\{{\mathcal S_{Y_i X}},p_{m+1}({\mathcal S_{Y_i X}})\}\cup\mu^X_k\subseteq \mathrm{ran}(i^X_k).$$
    Thus, $B\in \mathrm{ran}(i_k^X)$ and
    $$\begin{aligned}
    B &= i_k^X(B\cap \mu^X_{k})\\
    &= i_k^X(\Pi_{Y_i X}(D)\cap \mu^X_{k})\\
    &= i_k^X(\Pi_{Y_i X}(D\cap \mu^{Y_i}_\ell))\\
    &= i_k^X(\Pi_{Y_i X}(B'))\\
    &= i_k^X(A') = A.
    \end{aligned}$$
    Therefore,
    $$A\in \mathrm{Hull}^{{\mathcal S_{Y_i X}}}_{m+1}(\mu^X_k\cup p_{m+1}({\mathcal S_{Y_i X}})).$$
    The above calculation only relies on the fact that $B\in \mathrm{ran}(i_k^X)$ and will be used several times in the proof.
    
    Recall that we have assumed there exist unboundedly many $i<\omega_1$ such that $\mathcal S_{Y_i X}\lhd \mathcal M_X$ at the beginning of this case. 
    If this is false, then for almost all $i<\omega_1$, there exist $\lambda_i<\mu_X$ and an extender $E_i$ on the $\mathcal M_X$-sequence such that 
    $$\mathcal S_{Y_i X} = \mathrm{Ult}_m(\mathcal P^{X}_{\lambda_i},E_i).$$
    Fix a finite support $s_i$ for ultrapower representatives of the
    coordinates of $\Pi_{Y_iX}(q_{\mu_{Y_i}})$. Replace condition 2 by
    \begin{itemize}
    \item[2*.] $\lambda_i<\mu_k^X$ and $s_i\in\operatorname{ran}(i_k^X)$.
    If $E_i\in\mathcal M_X$, also require $E_i\in\operatorname{ran}(i_k^X)$.
    \end{itemize}
    The rest of the argument is unchanged.

    \subsection*{Case 2.} Suppose that $\mathcal Q_X\neq \mathcal M_X$ but both are set mice.
    Still, we first assume that for all large enough $i<\omega_1$, $\mathcal S_{Y_i X}\unlhd \mathcal M_X$. 
    If $\mathcal S_{Y_i X}=\mathcal M_X$ for unboundedly many $i<\omega_1$, we take $k_0<\omega$ large enough such that for unboundedly many $i<\omega_1$,
    $$p_{m+1}(\mathcal M_X)\in \mathrm{Hull}_{n+1}^{\mathcal S_{Y_i X}}(\mu^X_{k_0}\cup \Pi_{Y_i X}(q_{\mu_{Y_i}})),$$
    and the result follows immediately.
    For the case where $\mathcal S_{Y_i X}\lhd \mathcal M_X$ holds for almost all $i<\omega_1$, we modify the requirements for $k_0<\omega$ as follows: For every $k>k_0$, there exist unboundedly many $i<\omega_1$ such that
    \begin{itemize}
    \item[1.] $\pi_X(\mu^X_k)\in C_{Y_i}$;
    \item[2.] $\{\mathcal S_{Y_i X},\Pi_{Y_iX}(q_{\mu_{Y_i}})\}\subseteq \mathrm{ran}(i^X_k)$;
    \item[3.] Let $\ell<\omega$ be such that $\Pi_{Y_i X}(\mu^{Y_i}_\ell) =\mu^X_k$. 
    Then $\mu^{Y_i}_\ell$ is large enough in the sense of Lemma \ref{L5-5}.
    \end{itemize}
    The existence of such $k_0$ can be shown using a similar method as in Case 1. 
    Fix $k>k_0$ and $A\subseteq\mu_X$ in
    $\mathrm{Hull}_1^{\mathcal M_X}(\mu_k^X\cup p_1(\mathcal M_X))$.
    As in Case 1, $A'=A\cap\mu_k^X$ satisfies $i_k^X(A')=A$,
    and $\pi_X(A')\in Y_i$ for all sufficiently large $i$.
    Fix any such $i$ satisfying conditions 1--3, and take $\ell$ as in
    condition 3. The same argument as in Case 1 gives
    $B'=\pi_{Y_iX}^{-1}(A')\in\mathcal M_\ell^{Y_i}$
    with $\Pi_{Y_iX}(B')=A'$.
    A similar calculation shows
    $$D=i_\ell^{Y_i}(B')\in \mathrm{Hull}_{1}^{{\mathcal M_{Y_i}}}(\mu^{Y_i}_\ell \cup p_{1}({\mathcal M_{Y_i}})).$$
    By Lemma \ref{L5-5} and condition 3, we have 
    $$D\in \mathrm{Hull}_{n+1}^{{\mathcal Q_{Y_i}}}(\mu^{Y_i}_\ell \cup q_{\mu_{Y_i}}).$$
    Let $B = \Pi_{Y_i X}(D)$.
    A similar calculation shows $B\in \mathrm{ran}(i^{X}_k)$ and $B = A$. 

    In the case where $\mathcal S_{Y_i X}$ is not an initial segment of $\mathcal M_X$ for unboundedly many $i<\omega_1$, we modify condition 2 to be 2* as in Case 1 and argue as before.

    \subsection*{Case 3.}
    $\mathcal Q_{Y_i}$ is a weasel for every $i<\omega_1$.
    Since $Y_i\in X$, Corollary~\ref{CA-4} provides $R_i\in\mathcal M_X$ and $\rho_i<\mu_X$.
    As in Claim~\ref{c5-6-1}, increase $k_0$ so that for every $k>k_0$
    there are unboundedly many $i<\omega_1$ satisfying the following conditions.
    \begin{itemize}
    \item[1.] $\pi_X(\mu^X_k)\in C_{Y_i}$.
    \item[2.] $R_i\in\mathrm{ran}(i_k^X)$ and $\rho_i<\mu_k^X$.
    \item[3.] The index $\ell$ with $\pi_{Y_iX}(\mu_\ell^{Y_i})=\mu_k^X$
    is large enough for Lemma~\ref{L5-5} at $Y_i$.
    \end{itemize}
    Fix $k>k_0$ and $A$ as in the statement, and put $A'=A\cap\mu_k^X$.
    For every sufficiently large $i$ satisfying 1--3, $\pi_X(A')\in Y_i$.
    As in Case~1, $B'=\pi_{Y_iX}^{-1}(A')$ satisfies $\Pi_{Y_iX}(B')=A'$. As in Case~2,
    $D=i_\ell^{Y_i}(B')$ belongs to
    $$\mathcal P(\mu_{Y_i})\cap\mathrm{Hull}_1^{\mathcal M_{Y_i}}
      (\mu_\ell^{Y_i}\cup p_1(\mathcal M_{Y_i})).$$

    Shrink any thick class to a thick class $\Gamma$ of common fixed points of
    $h_{Y_i}\colon W\to\mathcal Q_{Y_i}$ and $\Pi_{Y_iX}$.
    By condition~3 and Lemma~\ref{L5-5}, there are a Skolem term $\tau$ and finite
    tuples $a_0$ from $\mu_\ell^{Y_i}$ and $b$ from $\Gamma$ such that
    $D=\tau^{\mathcal Q_{Y_i}}(a_0,t_{Y_i},b)$.
    Applying $\Pi_{Y_iX}\colon\mathcal Q_{Y_i}\to\mathcal S_{Y_iX}$ gives
    $$B=\Pi_{Y_iX}(D)=\tau^{\mathcal S_{Y_iX}}(a,t_{Y_iX},j_{Y_iX}(b)),$$
    where $a=\Pi_{Y_iX}(a_0)$ is a tuple from $\mu_k^X$ and $j_{Y_iX}(b)=\Pi_{Y_iX}(h_{Y_i}(b))=b$.
    Since $t_{Y_iX}=c(\mathcal S_{Y_iX})\setminus\mu_X$, this puts $B$ in the required hull.
    Corollary~\ref{CA-4} now gives $B=\{\beta<\mu_X:R_i(e,a,\beta)\}$ for some $e<\rho_i$.
    By condition~2, $R_i,e,a$ are in $\mathrm{ran}(i_k^X)$, as is $\mu_X=i_k^X(\mu_k^X)$.
    Thus $B\in\mathrm{ran}(i_k^X)$. Since $B\cap\mu_k^X=\Pi_{Y_iX}(B')=A'$,
    we have $B=i_k^X(A')=A$.
\end{proof}

In fact, the above proof method works in a more general sense:
Assume that $\mathrm{cf}(\nu) = \alpha$ and $\langle Y_i\mid i<\alpha^+\rangle$ is a $\subseteq$-increasing sequence such that $Y_i\in Y_{i+1}$ for all $i<\alpha^+$. 
Let $X = \bigcup_{i<\alpha^+}Y_i$ and assume that all $Y_i$ and $X$ gives the same set of answers for the questions in List B. 
Assume that $B_{Y_i}$ and $B_X$ is given.
Furthermore, $B_X\subseteq^* C_{Y_i}$ and $B_{Y_i}\subseteq^* C_{Y_j}$ for all $i,j<\alpha^+$. 
Readers can check that by the same argument, Lemma \ref{L5-6} holds for $X$ and $\langle Y_i\mid i<\alpha^+\rangle$, with $\omega$ replaced by $\alpha$, and $\omega_1$ replaced by $\alpha^+$.
This is crucial in proving Lemma \ref{L6-2}.

The above lemma can be used to show the following corollary.

\begin{corollary}\label{C5-7}
    $\mathcal P(\mu_X)\cap \mathcal M_X = \mathcal P(\mu_X)\cap W_X$.
\end{corollary}

\begin{proof}
    By the coiteration process, $\mathcal P(\mu_X)\cap \mathcal M_X \supseteq \mathcal P(\mu_X)\cap W_X$.
    For the other direction, let $A\in \mathcal P(\mu_X)\cap \mathcal M_X$ and let $i<\omega_1$ be such that $A\in {\mathcal S_{Y_i X}}$. 
    We need to show that $A\in W_X$.

    \subsection*{Case 1.} 
    ${\mathcal Q_{Y_i}}$ is a set mouse. 
    Since $Y_i\cup\{Y_i\}\subseteq X$, $N_{Y_i}\cup\{N_{Y_i}\}\subseteq X$ and $\pi_{Y_i}\cup\{\pi_{Y_i}\}\subseteq X$.
    Thus $\mathrm{crit}(\pi_X)>N_{Y_i}\cap\mathrm{Ord}$ and $\pi_{Y_i X} = \pi_{X}^{-1}(\pi_{Y_i})\in N_X$.
    Notice that $|{\mathcal Q_{Y_i}}| = |{\mu_{Y_i}}|\leq \mathrm{crit}(\pi_X)$ and ${\mathcal Q_{Y_i}}$ is definable with $Y_i$.
    Choose in $X$ a surjection from $\mu_{Y_i}$ onto the universe of $\mathcal Q_{Y_i}$. Since $\mu_{Y_i}\subseteq X$, every member of $\mathcal Q_{Y_i}$ belongs to $X$, so the collapse fixes $\mathcal Q_{Y_i}$.
    Thus, ${\mathcal Q_{Y_i}}\in N_X$ and ${\mathcal S_{Y_i X}}\in N_X$.
    Moreover, $\pi_X({\mathcal S_{Y_i X}}) = {\mathcal S_{Y_i}}$.
    Thus by Lemma \ref{L2-4}, $\mathcal S_{Y_i}\lhd W$ and thus ${\mathcal S_{Y_i X}}\lhd W_X$.
    We are done.

    \subsection*{Case 2.} 
    $\mathcal Q_{Y_i}$ is a weasel.
    By elementarity, we have
    $$\pi_X^{-1}({\mathcal S_{Y_i}}) = \mathrm{Ult}(\pi^{-1}_X({\mathcal Q_{Y_i }}),\pi_{Y_i X},\mu_X).$$
    By Lemma \ref{L2-5} or \ref{L2-6}, there exists an extender $G\in W$ such that 
    $$\mathcal P(\nu)\cap \mathcal S_{Y_i} = \mathcal P(\nu)\cap \mathrm{Ult}(W,G).$$
    Elementarity gives
    $$\mathcal P(\mu_X)\cap \pi^{-1}_X({\mathcal S_{Y_i}})\subseteq\mathcal P(\mu_X)\cap W_X.$$
    Let $\delta_X = \mathrm{crit}(\pi_X)$.
    Then $\delta_X>{\mu_{Y_i}}$.
    Since $\delta_X$ is a cardinal in $N_{X}$, $\delta_X$ is a cardinal in $\pi_X^{-1}({\mathcal Q_{Y_i}})$ and thus $\delta_X\geq (\mu_{Y_i}^{+})^{\pi_X^{-1}({\mathcal Q_{Y_i}})}$. 
    Towards a contradiction, assume $\delta_X= (\mu_{Y_i}^{+})^{\pi_X^{-1}({\mathcal Q_{Y_i}})}$.
    Now $\pi_X(\delta_X)$ is a cardinal in $H_{\Omega^+}$, and thus
    $$\pi_X(\delta_X) = \pi_X((\mu_{Y_i}^{+})^{\pi_X^{-1}({\mathcal Q_{Y_i}})}) = (\mu_{Y_i}^{+})^{{\mathcal Q_{Y_i}}} = |\mu_{Y_i}|^+.$$
    However,
    $$(\mu_{Y_i}^{+})^{{\mathcal Q_{Y_i}}} = (\mu_{Y_i}^{+})^{{\mathcal M_{Y_i}}} = \bigcup_{k<\omega}i_k^{Y_i} \bigl[((\mu_k^{Y_i})^+)^{\mathcal M_k^{Y_i}}\bigr],$$
    which implies that $|(\mu_{Y_i}^{+})^{{\mathcal Q_{Y_i}}}| = |\mu_{Y_i}|$. Contradiction.
    Thus 
    $$\delta_X>(\mu_{Y_i}^+)^{\pi_X^{-1}(\mathcal Q_{Y_i})} = (\mu_{Y_i}^+)^{\mathcal Q_{Y_i}}.$$ 
    Therefore $\pi_X^{-1}(\mathcal Q_{Y_i})$ and $\mathcal Q_{Y_i}$ compute the same powerset of $\mu_{Y_i}$.
    Thus, we have
    $$\mathcal P(\mu_X)\cap \pi^{-1}_X({\mathcal S_{Y_i}}) = \mathcal P(\mu_X)\cap {\mathcal S_{Y_i X}}.$$
    Recall that we assumed $A\in \mathcal S_{Y_i X}$, which leads to $A\in \pi^{-1}_X({\mathcal S_{Y_i}})$ and thus $A\in W_X$.
    Exactly what we want.
\end{proof}

By the acceptability of $W_X$ and $\mathcal M_X$, the following corollary is obvious.

\begin{corollary}\label{C5-8}
    $(\mu^+_X)^{W_X} = (\mu^+_X)^{\mathcal M_X}$, and $W_X\upharpoonright (\mu_X^{++})^{W_X}\unlhd \mathcal M_X$.
    Especially, for any $\xi<o^{W_X}(\mu_X)$, $U(W_X,\mu_X,\xi) = U(\mathcal M_X,\mu_X,\xi)$.
\end{corollary}

We have now finished all the prerequisites. 
Our next lemma will tell us that $C_X\in X$.

\begin{lemma}\label{L5-9}
     $C_X =^* C_{Y_j}$ for unboundedly many $j<\omega_1$.
\end{lemma}

\begin{proof}
    By Lemma~\ref{L5-4}, $C_X\subseteq^* C_{Y_j}\subseteq^* C_{Y_0}$ for every $0<j<\omega_1$.
    We need only consider points $\pi_{Y_0}(\mu_n^{Y_0})$ in $C_{Y_0}\setminus C_X$ above $\min C_X$.

\begin{claim}\label{c5-9-1}
    There is a club $D\subseteq\nu$ in $W\cap X$ such that $D\cap(C_{Y_0}\setminus C_X)$ is finite.
\end{claim}

\begin{proof}[Proof of Claim.]
    Fix $n$ such that $\pi_{Y_0}(\mu_n^{Y_0})\notin C_X$ and $\pi_{Y_0}(\mu_n^{Y_0})>\min C_X$.
    Choose $\ell$ with
    $$\pi_X(\mu_\ell^X)<\pi_{Y_0}(\mu_n^{Y_0})<\pi_X(\mu_{\ell+1}^X).$$
    By Lemma~\ref{L5-2}, the first step at $\mu_\ell^X$ is a normal-measure ultrapower.
    The branch map from that ultrapower to $\mathcal M_{\ell+1}^X$ fixes $\mu_\ell^X$
    and maps the ordinals below its first image cofinally into $\mu_{\ell+1}^X$,
    as in the proof of Lemma~\ref{L3-4}. Thus we may choose
    $f:\mu_\ell^X\to\mu_\ell^X$ in $\mathcal M_\ell^X$ such that, writing $F_n=i_\ell^X(f)$,
    $$\pi_{Y_0X}(\mu_n^{Y_0})<F_n(\mu_\ell^X)<\mu_{\ell+1}^X.$$

    The graph of $F_n$ is coded by a subset of $\mu_X$, so $F_n\in W_X$ by Corollary~\ref{C5-7}.
    Choose $t<\omega_1$ such that $Y_t$ contains all the $\pi_X(F_n)$.
    By weak covering, $\mathrm{cf}^V((\nu^+)^W)\geq|\nu|>|Y_t|$.
    Hence a single function in $W$ eventually dominates all the functions from $\nu$ to $\nu$ in $Y_t\cap W$.
    Since $Y_t,W,\nu\in X$, we may choose such an $h\in W\cap X$ by elementarity.
    Its club of closure points $D=\{\delta<\nu:h[\delta]\subseteq\delta\}$ belongs to $W\cap X$.

    The function $\bar h=\pi_X^{-1}(h)$ eventually dominates each $F_n$
    and belongs to $\mathcal M_X$ by Corollary~\ref{C5-7}.
    By the direct limit, $\bar h\in\operatorname{ran}(i_\ell^X)$ for all sufficiently large $\ell$.
    For $n$ in any such gap, $F_n$ is in the same range. The least $\xi<\mu_X$ such that
    $F_n(\gamma)\leq\bar h(\gamma)$ for all $\gamma\in[\xi,\mu_X)$ therefore belongs to
    $\operatorname{ran}(i_\ell^X)\cap\mu_X=\mu_\ell^X$. It follows that
    $$\bar h(\mu_\ell^X)\geq F_n(\mu_\ell^X)>\pi_{Y_0X}(\mu_n^{Y_0})>\mu_\ell^X.$$
    Thus $\pi_{Y_0}(\mu_n^{Y_0})\notin D$. Only finitely many points of $C_{Y_0}$ precede these gaps.
\end{proof}
    
    For all sufficiently large $j$, $D\in Y_j$. By \cite[Lemma~3.10]{MS2},
    $U(\mathcal S_{Y_j},\nu,0)$ measures $D$ and contains it since $D$ is a club.
    Lemma~\ref{L5-2} and \cite[Section~3.1]{MS2} now give $C_{Y_j}\subseteq^*D$.
    The claim and Lemma~\ref{L5-4} give
    $$C_X\subseteq^* C_{Y_j}\subseteq^* C_{Y_0}\cap D\subseteq^* C_X.$$
\end{proof}

We can thus show that $C_X$ is the maximal Prikry sequence for $U(W,\nu,0)$.
Pick $j<\omega_1$ as in the above lemma. 
Since $Y_j\in X$, $C_{Y_j}\in X$.
$C_X$ only differs from $C_{Y_j}$ by finitely many points, and thus $C_X\in X$.
From the iteration process, we know that $\overline C_X$ is a Prikry sequence for $U(\mathcal M_X,\mu_X,0)$. 
By results in \cite{MS2}, $U(W_X,\mu_X,0)$ exists, and thus by Corollary \ref{C5-8}, $U(\mathcal M_X,\mu_X,0) = U(W_X,\mu_X,0)$. 
Elementarity of $\pi_X$ implies that $C_X$ is a Prikry sequence over $U(W,\nu,0)$. 
The maximality of $C_X$ relies on the following lemma, which follows from a similar argument as Lemma \ref{L3-4}.
\begin{lemma}\label{L5-10}
    Let $\mathbf{d}\subseteq \nu\cap X$ with order type $\omega$.
    Then $\mathbf{d}\subseteq^* C_X$ if and only if for any club subset $A\in \mathcal P(\nu)\cap W$, $\mathbf{d}\subseteq^* A$. 
\end{lemma}

\begin{proof}[Sketch.]
    Following the arguments in the proof of Lemma \ref{L3-4}, pick a club subset $A\in \mathcal P(\nu)\cap W$ under different cases.
    Then $\mathbf{d}\subseteq^* A$ if and only if $\mathbf{d}\subseteq^* C_X$.

    For maximality, assume towards a contradiction that there exists a Prikry sequence $\mathbf{d}\subseteq \nu$ over $U(W,\nu,0)$ such that $\mathbf{d}\not\subseteq^* C_X$.
    Elementarity of $X$ ensures that we may assume $\mathbf{d}\in X$, and thus $\mathbf{d}\subseteq X$, contradicting with the lemma above.
\end{proof}

Readers may check that our proof for Lemma \ref{L3-4} does not rely on the assumption that $|\nu|$ is countably closed, and all the lemmas and corollaries used in that proof is shown for $X$. 

\begin{lemma}\label{L5-13}
    Assume that measurable cardinals in $K$ are bounded in $\nu$.
    Let $A\subseteq \nu$ with $|A|<|\nu|$. Then there exists
    $B\in K[C_0]$ such that 
    $A\subseteq B$ and $|B|^{K[C_0]}<\nu$.
\end{lemma}

\begin{proof}
    In the construction above take $\kappa=\max(\omega_1,|A|)$ and start with $A\subseteq Y_0$.
    Then $A\subseteq X\cap\nu$. By maximality, $C_X=^*C_0$, so $K[C_X]=K[C_0]$.
    The proof of Lemma~\ref{L3-6} gives a set $B$ in this model containing $X\cap\nu$
    with $|B|^{K[C_0]}<\nu$.
\end{proof}

Therefore, Theorem \ref{T5-1} is shown.
\ChapterLabel{Ch6}
\chapter{Generating Sequence and Closure Assumption}

The following definition was used by Mitchell in \cite{M}. 
It generalizes the idea of the Mathias condition of Prikry sequences.
\begin{definition}\label{D6-1}
    Let $U$ be a nonprincipal $\mathcal M$-normal ultrafilter on $\mathcal P(\kappa)\cap\mathcal M$.
    We say that a cofinal subsequence $\mathbf{c}\subseteq \kappa$ is a \emph{generating sequence} for $U$, iff there exists a function $g:\mathbf{c}\to \kappa$ such that $g(\alpha)\leq o^\mathcal M(\alpha)$ for all $\alpha\in\mathbf{c}$, and for any $A\in \mathcal P(\kappa)\cap \mathcal M$, $A\in U$ iff there exists some $\delta<\kappa$ such that for all $\alpha\in\mathbf{c}\setminus \delta$, 
    \begin{itemize}
        \item either $g(\alpha) = o^\mathcal M(\alpha)$ and $\alpha\in A$, 
        \item or $g(\alpha)<o^\mathcal M(\alpha)$ and $A\cap \alpha\in U(\mathcal M,\alpha,g(\alpha))$.
    \end{itemize}
    If $U$ is not a member of $\mathcal M$, we say that $\mathbf{c}$ generates a measure on $\mathcal M$ over $\kappa$.
\end{definition}
Prikry sequences provide simple examples of generating sequences. 
Suppose $\nu$ is a measurable cardinal in $\mathcal M$ and $C$ is a Prikry sequence for $U(\mathcal M,\nu,0)$. 
Then $C$ is a generating sequence for $U(\mathcal M,\nu,0)$. 
Moreover, every generating sequence of any measure is almost contained in any club subset of $\nu$ in $\mathcal M$.
Let $\beta$ be the least ordinal, if there is any, such that $U(K,\nu,\beta)$ has no generating sequences. 
Otherwise, let it be $o^K(\nu)$.
We now prove Theorem~\ref{T1-4}.

\begin{theorem}\label{T6-1}
    Assume there is no inner model with a Woodin cardinal. 
    Suppose that $\nu>\omega_2$ is a cardinal in Jensen-Steel $K$ such that 
    $$\mathrm{cf}(\nu)<|\nu|\leq\nu = \mathrm{cf}^K(\nu),$$
    and $|\nu|$ is $\mathrm{cf}(\nu)$-closed.
    Moreover, assume that $\beta = \lambda+1$ is a successor ordinal.
    Then there exists a maximal Prikry sequence $C$ over $U(K,\nu,\lambda)$.
\end{theorem}

Our strategy is as follows:
We will first pick a suitable substructure $X$ as before.
Suppose that $\mathbf c\subseteq \nu$ is a generating sequence for $U(K,\nu,\lambda)$ with $\mathrm{type}(\mathbf{c}) = \mathrm{cf}(\nu)$ and $\mathbf c\subseteq X$.
By the argument of Lemma~\ref{L3-4}, a tail of $\mathbf c$ is contained in $C_X$, since every club subset of $\nu$ in $K$ contains a tail of $\mathbf c$.
Pick a large enough $\mu^X_\xi$ such that $\pi_X(\mu^X_\xi)\in \mathbf c$.
We will show that $E^{\mathcal T_X}_\xi$ is equivalent to a normal measure on $\mu^X_\xi$, and it has Mitchell order exactly at $\overline \lambda_\xi$.
Let the collection of all such $\mu^X_\xi$ be $\overline C_0$. 
We will show that the set of limit points of $\overline C_0$ is bounded in $\mu_X$, and every generating sequence $\mathbf d\subseteq \nu$ of $U(K,\nu,\lambda)$ is almost contained in $C_0 = \pi_X[\overline C_0]$.
Thus $\overline C_0$ has a tail of order type $\omega$, whose image under $\pi_X$ is the desired maximal Prikry sequence.
We start the proof from the following lemma, which defines the suitable structure we are going to use in this section.

\begin{lemma}\label{L6-2}
    There exists an $X\prec(H_{\Omega^+},\in,U_\Omega)$ with the following properties:
    \begin{itemize}
        \item[1.] $X$ satisfies Requirements 1 and 2 in Chapter~\ref{Ch2};
        \item[2.] ${}^{\mathrm{cf}(\nu)}X\subseteq X$;
        \item[3.] $X$ satisfies Corollary \ref{C5-7}, i.e., 
        $$\mathcal P(\mu_X)\cap \mathcal M_X = \mathcal P(\mu_X)\cap W_X.$$
    \end{itemize}
\end{lemma}

\begin{proof}
    We construct an internal approachable chain $\langle Y_{i}\mid i<\left(2^{\mathrm{cf}(\nu)}\right)^+\rangle$ of elementary substructures of $(H_{\Omega^+},\in,U_\Omega)$ with the following properties:
    For each $i<\left(2^{\mathrm{cf}(\nu)}\right)^+$,
    \begin{itemize}
        \item $|Y_i| = 2^{\mathrm{cf}(\nu)}$;
        \item ${}^{\mathrm{cf}(\nu)}Y_i\subseteq Y_{i+1}$;
        \item $2^{\mathrm{cf}(\nu)}+1\subseteq Y_0$ and $\{\nu,\Omega_0,W,\lambda\}\subseteq Y_0$.
    \end{itemize}
    These conditions make sense, since 
    $$\left(2^{\mathrm{cf}(\nu)}\right)^{\mathrm{cf}(\nu)} = 2^{\mathrm{cf}(\nu)\cdot \mathrm{cf}(\nu)} = 2^{\mathrm{cf}(\nu)},$$
    and if $j<\left(2^{\mathrm{cf}(\nu)}\right)^+$ is a limit ordinal, then $|j|\leq 2^{\mathrm{cf}(\nu)}$, and thus
    $$|Y_j| = \sum_{i<j}|Y_i| \leq 2^{\mathrm{cf}(\nu)}\cdot 2^{\mathrm{cf}(\nu)} = 2^{\mathrm{cf}(\nu)}.$$
    By Theorem \ref{T2-2}, there exists a club $\mathcal C$ such that for each $i\in \mathcal C$ with $\mathrm{cf}(i)>\omega$, $Y_i$ satisfies Requirements 1 and 2. 
    Choose a stationary $S\subseteq\{i\in\mathcal C:\mathrm{cf}(i)=\mathrm{cf}(\nu)^+\}$ on which the answers to List~B in Chapter~\ref{Ch5} are constant.
    Let $\varepsilon\in S\cap \lim(S)$ with $\mathrm{cf}(\varepsilon) = \mathrm{cf}(\nu)^+$. 
    Let $\{\varepsilon_\alpha\mid \alpha<\mathrm{cf}(\nu)^+\}\subseteq S\cap \varepsilon$ with $\sup_{\alpha<\mathrm{cf}(\nu)^+}\varepsilon_\alpha = \varepsilon$. 
    Let $X_\alpha = Y_{\varepsilon_\alpha}$ and $X = Y_{\varepsilon}$.
    We can conclude the following properties from our construction:
    \begin{itemize}
        \item ${}^{\mathrm{cf}(\nu)}X\subseteq X$;
        \item $X$ and each $X_\alpha$ satisfies Requirements 1 and 2.
    \end{itemize}
    It remains only to show that condition 3 holds for $X$. 
    Pick some cofinal sequence $B_X\subseteq C_X$ of order type $\mathrm{cf}(\nu)$.
    Then $B_X\in X$.
    Without loss of generality, we may assume that $\{B_X\}\cup B_X\subseteq X_\alpha$ for all $\alpha<\mathrm{cf}(\nu)^+$.
    This allows us to apply Lemma \ref{L5-4}, which implies that $B_X\subseteq^* C_{X_\alpha}$ for all $\alpha<\mathrm{cf}(\nu)^+$. 
    By adapting the argument in the proof of Lemma \ref{L5-6} to the context of $X_\alpha$ and $X$, we see that the same reasoning applies.
    Since Corollary \ref{C5-7} depends solely on the result of Lemma \ref{L5-6}, the same proof remains valid for $X$.
\end{proof}

From this point, we will only use $X$ in our following arguments and adapt the results in Chapter~\ref{Ch2}. 
As in Chapter~\ref{Ch3}, we will suppress the subscript $X$ on each symbol.
The following lemma is an observation which generalizes Claim \ref{c3-2-1}.

\begin{lemma}\label{L6-3}
    Let $\eta\in I$ and $\zeta+1 = \min(\eta,\theta)_{\mathcal T}$. If $E^{\mathcal T}_{\zeta}$ has a generator greater than $\mu_\eta$ and $\xi$ is the least such, then $E^{\mathcal T}_{\zeta}\upharpoonright \xi$ is equivalent with a total measure on $\mu_\eta$, $E^{\mathcal T}_{\zeta}\upharpoonright \xi$ is on the sequence of $\mathcal M^{\mathcal T}_\eta$, and 
    $$o^{\mathcal M^{\mathcal T}_{\eta+1}}(\mu_\eta)=o^{\mathcal M^{\mathcal T}_{\zeta+1}}(\mu_\eta)=o^{\mathcal M_X}(\mu_\eta)=(\mu_\eta^{++})^{\mathcal M_X}.$$
    If $E^{\mathcal T}_{\zeta}$ does not satisfy the above assumption, then $\zeta = \eta$ and $E^{\mathcal T}_{\eta}$ is equivalent with a normal measure. 
\end{lemma} 

\begin{proof}
    We need to show the following facts:
    \begin{itemize}
        \item[1.] $(\mu_\eta^{+})^{\mathcal M^{\mathcal T}_\eta} = (\mu_\eta^{+})^{\mathcal M_{X}}$;
        \item[2.] $\mathrm{lh}(E^{\mathcal T}_{\eta})\geq (\mu_\eta^{++})^{\mathcal M^{\mathcal T}_{\eta+1}}$
        \item[3.] Assume that $E^{\mathcal T}_{\eta}$ has critical point $\mu_\eta$ and has no other generator. Then $E^{\mathcal T}_{\eta}$ is equivalent with a normal measure on $\mathcal M^{\mathcal T}_{\eta}$;
        \item[4.] Assume that $E^{\mathcal T}_{\eta}$ has critical point $\mu_\eta$ and has more than one generator. Then $o^{\mathcal M^{\mathcal T}_{\eta+1}}(\mu_\eta) = (\mu_\eta^{++})^{\mathcal M^{\mathcal T}_{\eta+1}}$.
    \end{itemize}
    
    For 1 and 2, since $E^{\mathcal T}_{\zeta}$ is a total extender for $\mathcal M^{\mathcal T}_{\eta}$, $\mathrm{lh}(E^{\mathcal T}_{\eta})>(\mu_\eta^+)^{\mathcal M^{\mathcal T}_\eta}$. Therefore, 1 is shown by agreement between models. 2 is also obvious because of this. Thus, if $\zeta>\eta$ but $E^{\mathcal T}_\zeta$ is equivalent with a measure, then 
    $$\mathrm{lh}(E^{\mathcal T}_\zeta)<(\mu_\eta^{++})^{\mathcal M^{\mathcal T}_{\zeta}} = (\mu_\eta^{++})^{\mathcal M^{\mathcal T}_{\eta+1}}\leq \mathrm{lh}(E^{\mathcal T}_{\eta}),$$
    which violates the rule of iteration.

    For 3, $E^{\mathcal T}_{\eta}$ is equivalent with $(E^{\mathcal T}_{\eta})_{\{\mu_\eta\}}$, which is equivalent with the measure
    $$U = \left\{\bigcup A\mid A\in (E^{\mathcal T}_{\eta})_{\{\mu_\eta\}}\right\}.$$
    For 4, we let $E = E^{\mathcal T}_{\eta}$ and $\xi$ be its second generator. By agreement between models in $\mathcal T$, we only need to show that 
    $$\mathrm{Ult}(\mathcal M^{\mathcal T}_{\eta},E)\vDash o(\mu_\eta) = \mu_\eta^{++}.$$
    Let $i:\mathcal M^{\mathcal T}_{\eta}\to \mathrm{Ult}(\mathcal M^{\mathcal T}_{\eta},E)$, $j:\mathcal M^{\mathcal T}_{\eta}\to \mathrm{Ult}(\mathcal M^{\mathcal T}_{\eta},E\upharpoonright\xi)$ and $k: \mathrm{Ult}(\mathcal M^{\mathcal T}_{\eta},E\upharpoonright\xi)\to \mathrm{Ult}(\mathcal M^{\mathcal T}_{\eta},E)$ be the factor map. By the definition of generator, $\mathrm{crit}(k) = \xi$, and thus $\xi$ is a cardinal in $\mathrm{Ult}(\mathcal M^{\mathcal T}_{\eta},E\upharpoonright\xi)$ which is greater than $(\mu_\eta^+)^{\mathrm{Ult}(\mathcal M^{\mathcal T}_{\eta},E\upharpoonright\xi)} = (\mu_\eta^+)^{\mathcal M^{\mathcal T}_{\eta}}$.
    \begin{claim}\label{c6-3-1}
        $\xi = (\mu_\eta^{++})^{\mathrm{Ult}(\mathcal M^{\mathcal T}_{\eta},E\upharpoonright\xi)}$.
    \end{claim}
    \begin{proof}[Proof of Claim.]
        By initial segment property, since $E\upharpoonright\xi$ is equivalent with a measure, the trivial completion of $E\upharpoonright\xi$ has length $(\mu_\eta^{++})^{\mathrm{Ult}(\mathcal M^{\mathcal T}_{\eta},E\upharpoonright\xi)}$. Let this ordinal be $\xi^*$. Suppose the claim fails. Then $E^{\mathcal M^{\mathcal T}_\eta}_{\xi^*}$ does not exist on the $\mathrm{Ult}(\mathcal M^{\mathcal T}_\eta,E\upharpoonright\xi)$-sequence. However, since $\mathrm{crit}(k) = \xi>\xi^*$, $E^{\mathcal M^{\mathcal T}_\eta}_{\xi^*}$ is also not on the sequence of $\mathrm{Ult}(\mathcal M^{\mathcal T}_\eta,E)$, contradiction.
    \end{proof}
    By the claim, it follows that the trivial completion of $E\upharpoonright\xi$ is $E^{\mathcal M^{\mathcal T}_{\eta}}_\xi$. Notice that this extender, which is on the $\mathrm{Ult}(\mathcal M^{\mathcal T}_\eta,E)$-sequence, witnesses the following property: For each $\xi'<\xi$,
    \begin{equation*}
\begin{aligned}
    \mathrm{Ult}(\mathcal{M}^\mathcal{T}_\eta, E) \vDash& \exists \xi'' \in (\xi',k(\xi))\\ \big( &E_{\xi''}^{\mathrm{Ult}(\mathcal{M}^\mathcal{T}_\eta, E)} \text{ is equivalent with}
    \text{a total measure on } \mu_\eta \big).
\end{aligned}
\end{equation*}
    Thus, by elementarity, there exist unboundedly many $\xi''<\xi$ such that\linebreak[4] $E^{\mathrm{Ult}(\mathcal M^{\mathcal T}_\eta,E\upharpoonright\xi)}_{\xi''}$ is equivalent with a measure on $\mu_\eta$. Thus,
    $$\mathrm{Ult}(\mathcal M^{\mathcal T}_\eta,E\upharpoonright\xi)\vDash o(\mu_\eta) = \mu_{\eta}^{++}.$$
    By elementarity, this also holds in $\mathrm{Ult}(\mathcal M^{\mathcal T}_\eta,E)$. We are done.
\end{proof}

\begin{corollary}\label{C6-4}
    $(\mu^+)^{\overline W} = (\mu^+)^{\mathcal M_X}$, and $\overline W\upharpoonright (\mu^{++})^{\overline W}\unlhd \mathcal M_X$. Especially, for any $\xi<o^{\overline W}(\mu)$, $U(\overline W,\mu,\xi) = U(\mathcal M_X,\mu,\xi)$.
\end{corollary}

\begin{corollary}\label{C6-5}
Suppose that $\overline{\mathbf{c}}\subseteq\mu$ is a generating sequence of a measure $U(\mathcal M_X,\mu,\xi)$ with some function $\overline f$. Then $\xi\in\mathrm{dom}(\pi)$, $U(W,\nu,\pi(\xi))$ exists and $\mathbf{c} = \pi[\overline{\mathbf c}]$ is a generating sequence of this measure with function $f = \pi[\overline f]$.
\end{corollary}

\begin{proof}
    We first look at the case where $\mathrm{type}(\overline{\mathbf{c}}) = \mathrm{cf}(\nu)$. Since $X$ is closed under $\mathrm{cf}(\nu)$-sequences, $\overline{\mathbf{c}}\in N_X$ and $\overline{f}\in N_X$. Let $\mathbf{c} = \pi(\overline{\mathbf{c}})$ and $f = \pi(\overline f)$. By Corollary \ref{C6-4}, we have
    $$N_X\vDash \overline{\mathbf{c}}\mbox{ generates a measure on }\overline W\mbox{ over }\mu\mbox{ with }\overline{f}.$$
    Thus by elementarity, 
    $$X\vDash \mathbf{c}\mbox{ generates a measure on }W\mbox{ over }\nu\mbox{ with }f.$$
    Call this measure $U'$. Using the same technique as in Section 5 of \cite{MS2}, we can show that the phalanx 
    $$((\overline W,\mathrm{Ult}(\overline W,U(\mathcal M_X,\mu,\xi))),\mu)$$
    is iterable, and hence so is $((W,\mathrm{Ult}(W,U')),\nu)$. Comparing $W$ with this phalanx gives $U'\in W$. Thus $U'=U(W,\nu,\xi')$ for some $\xi'$.

    We next show that $\xi' = \pi(\xi)$. Since $\mathbf{c}\in \mathrm{ran}(\pi)$, $\xi'\in\mathrm{ran}(\pi)$, and thus $\overline{\mathbf{c}}$ is a generating sequence of $U(\overline W,\mu,\pi^{-1}(\xi'))$ with $\overline f$. Since $\mathcal P(\mu)\cap \overline W = \mathcal P(\mu)\cap \mathcal M_X$, 
    $$U(\mathcal M_X,\mu,\xi) = U(\overline W,\mu,\pi^{-1}(\xi')),$$
    and thus $\xi = \pi^{-1}(\xi')$. We are done.

    We then look at the case where $\mathrm{type}(\mathbf{\overline c})>\mathrm{cf}(\nu)$. For each cofinal subsequence $\overline{\mathbf c'}\subseteq \overline{\mathbf c}$ with $\mathrm{type}(\overline{\mathbf c'}) = \mathrm{cf}(\nu)$, $\overline{\mathbf c'}$ satisfies the corollary, and thus $\mathbf c' = \pi[\overline{\mathbf{c}'}]$ generates $U(W,\nu,\pi(\xi))$ by $f\upharpoonright \mathbf{c'}$. If $\mathbf{c}$ does not generate $U(W,\nu,\pi(\xi))$ by $f$, we can pick a cofinal subsequence with order type $\mathrm{cf}(\nu)$ witnessing this fact, but there is a contradiction.
\end{proof}

Let $\lambda^*$ be such that $E^W_{\lambda^*}$ is equivalent with $U(W,\nu,\lambda)$. 
Let $\pi(\overline\lambda^*) = \lambda^*$.
By Corollary \ref{C6-4}, $E^{\mathcal M_X}_{\overline\lambda^*} = E^{\overline W}_{\overline\lambda^*}$ and it is equivalent with $U(\mathcal M_X,\mu,\overline\lambda) = U(\overline W,\mu,\overline\lambda)$.
For large enough $\eta\in I$, let $i^{\mathcal T}_{\eta,\theta}(\overline\lambda_\eta) = \overline\lambda$ and let $i^{\mathcal T}_{\eta,\theta}(\overline\lambda_\eta^*) = \overline\lambda^*$.

\begin{lemma}\label{L6-6}
    Suppose that $\mathbf{d}\subseteq C$ is a generating sequence of $U(W,\nu,\lambda)$ via a function $f\in X$. 
    Then there exists some $\delta<\nu$ such that for all $\pi(\mu_\eta)\in \mathbf{d}\setminus \delta$, $\mathrm{lh}(E^{\mathcal T}_{\eta})\geq \overline\lambda_\eta^*$.
\end{lemma}

\begin{proof}
    Suppose otherwise and let $\mathbf{d}$ be a counterexample. 
    Shrink $\mathbf{d}$ if necessary, we may assume that $\mathrm{type}(\mathbf{d}) = \mathrm{cf}(\nu)$ and for all $\pi(\mu_\eta)\in\mathbf d$, $\mathrm{lh}(E^{\mathcal T}_{\eta})<\overline\lambda_\eta^*$. 
    Therefore, for all $\pi(\mu_\eta)\in\mathbf d$, 
    $$o^{\mathcal M_X}(\mu_\eta)=o^{\overline W}(\mu_\eta)\leq\overline\lambda_\eta.$$
    By the definition of generating functions, we know that for all $\pi(\mu_\eta)\in\mathbf d$, $f(\pi(\mu_\eta))\leq o^W(\pi(\mu_\eta))$. Thus we may further shrink $\mathbf d$ to make it satisfy one of the following assumptions:
    \begin{itemize}
        \item[a.] For all $\pi(\mu_\eta)\in \mathbf{d}$, $f(\pi(\mu_\eta)) = o^W(\pi(\mu_\eta))$;
        \item[b.] For all $\pi(\mu_\eta)\in \mathbf{d}$, $f(\pi(\mu_\eta)) < o^W(\pi(\mu_\eta))$.
    \end{itemize}
    Suppose that Case a holds. By the definition of generating sequence,
    $$\forall x\in\mathcal P(\nu)\cap W\exists\delta<\nu\forall \pi(\mu_\eta)\in\mathbf d\setminus \delta\left(x\in U(W,\nu,\lambda)\iff \pi(\mu_\eta)\in x\right).$$
    By the elementarity of $\pi$,
    $$N_X\vDash \forall x\in\mathcal P(\mu)\cap \overline W\exists\delta<\mu\forall \mu_\eta\in\overline {\mathbf d}\setminus \delta\left(x\in U(\overline W,\mu,\overline\lambda)\iff \mu_\eta\in x\right).$$
    By the agreement between $\mathcal M_X$ and $\overline W$,
    $$\forall x\in\mathcal P(\mu)\cap \mathcal M_X\exists\delta<\mu\forall \mu_\eta\in\overline {\mathbf d}\setminus \delta\left(x\in U(\mathcal M_X,\mu,\overline\lambda)\iff \mu_\eta\in x\right).$$
    For each $\eta\in I$ such that $\mu_\eta\in\overline{\mathbf d}$, let $U_\eta$ be the measure derived from $i^{\mathcal T}_{\eta,\theta}$. By Lemma \ref{L6-3}, there exists an ordinal $\overline{\xi}_\eta<\overline{\lambda}_\eta$ such that $U_\eta = U(\mathcal M_\eta,\mu_\eta,\overline{\xi}_\eta)$. Thus, 
    $$U(\mathcal M_X,\mu,i_{\eta,\theta}^{\mathcal T}(\overline{\xi}_\eta)) = U(\overline W,\mu,i_{\eta,\theta}^{\mathcal T}(\overline{\xi}_\eta)).$$
    Let $\xi_\eta = \pi(i^{\mathcal T}_{\eta,\theta}(\overline\xi_\eta))$. Notice that $\{\xi_\eta\mid \pi(\mu_\eta)\in \mathbf d\}\in X$ by the closure assumption.
    \begin{claim}\label{c6-6-1}
    There exists a subset $x\in U(W,\nu,\lambda)$ such that $x\not\in U(W,\nu,\xi_\eta)$ for all $\pi(\mu_\eta)\in \mathbf{d}$.
    \end{claim} 
    This would be enough to show a contradiction:
    By elementarity, we may choose such an $x$ in $X$. 
    Let $\overline x = \pi^{-1}(x)$. 
    Then $\overline x\in \mathcal M_X$ and $\overline x\in U(\mathcal M_X,\mu,\overline\lambda)\setminus U(\mathcal M_X,\mu,i^{\mathcal T}_{\eta,\theta}(\overline{\xi}_\eta))$. 
    Pick some $\delta<\nu$ such that for all $\pi(\mu_\eta)\in\mathbf{d}\setminus \delta$, $\pi(\mu_\eta)\in x$ and $\overline x\in\mathrm{ran}(i^{\mathcal T}_{\eta,\theta})$. 
    By elementarity, we have
    $$\mu_\eta\in \overline x\iff \overline x\cap \mu_\eta\in U(\mathcal M^{\mathcal T}_{\eta},\mu_\eta,\overline{\xi}_\eta)\iff \overline x\in U(\mathcal M_X,\mu,i^{\mathcal T}_{\eta,\theta}(\overline{\xi}_\eta)).$$
    A contradiction.
    \begin{proof}[Proof of Claim.]
         Let $A=\{\lambda\}\cup\{\xi_\eta:\pi(\mu_\eta)\in\mathbf d\}$, and in $W$ fix a surjection $h:(\nu^+)^W\to\lambda+1$.
         By weak covering \cite[Theorem~1.1]{MSS}, 
         $$|A|\leq\mathrm{cf}(\nu)<|\nu|\leq\mathrm{cf}((\nu^+)^W).$$ Choose $\gamma<(\nu^+)^W$ above the least $h$-preimage of each member of $A$. Then $B=h[\gamma]\in W$ contains $A$ and $|B|^W\leq\nu$.
         In $W$ choose pairwise disjoint sets $x_\alpha\in U(W,\nu,\alpha)$ for $\alpha\in B$. Take $x=x_\lambda$.
    \end{proof}

    Now suppose that $\mathbf{d}$ falls under Case b. 
    Let $\overline f=\pi^{-1}(f)$. For each $\pi(\mu_\eta)\in\mathbf d$, put $\overline\xi_\eta=i^{\mathcal T}_{\eta,\theta}(\overline f(\mu_\eta))$. Since $\overline f(\mu_\eta)<\overline\lambda_\eta$, we have $\overline\xi_\eta<\overline\lambda$ and
    $$U(\mathcal M_X,\mu,\overline\xi_\eta)=i^{\mathcal T}_{\eta,\theta}(U(\mathcal M_\eta^{\mathcal T},\mu_\eta,\overline f(\mu_\eta))).$$
    Thus $$U(\mathcal M_X,\mu,\overline{\xi}_\eta) = U(\overline W,\mu,\overline{\xi}_\eta).$$
    Let $\xi_\eta = \pi(\overline{\xi}_\eta)$.
    Following a similar reasoning as in case a would lead us to a contradiction.
\end{proof}

\begin{lemma}\label{L6-7}
Let $\mathbf d\subseteq C$ be a generating sequence of $U(W,\nu,\lambda)$ via a function $f\in X$.
Then there exists some $\delta<\nu$ such that for every $\pi(\mu_\eta)\in \mathbf{d}\setminus \delta$, $\eta+1=\min(\eta,\theta)_{\mathcal T}$ and $E^{\mathcal T}_{\eta}$ is equivalent to $U(\mathcal M^{\mathcal T}_{\eta},\mu_\eta,\overline\lambda_{\eta})$.
\end{lemma}

\begin{proof}
    Suppose otherwise and let $\mathbf d$ be a counterexample. 
    By our last lemma, we may shrink $\mathbf{d}$ to make sure that $\mathrm{type}(\mathbf{d}) = \mathrm{cf}(\nu)$ and for all $\pi(\mu_\eta)\in\mathbf{d}$, $\mathrm{lh}(E^{\mathcal T}_{\eta})>\overline\lambda_\eta^*$. 
    Define a function $g:\mathbf{d}\to \mathrm{Ord}$ as follows: 
    $$g:\pi(\mu_\eta)\mapsto\pi(\overline\lambda_\eta)+1.$$
    By the closure assumption, $g\in X$. 
    It suffices to obtain a measure of order $\overline\lambda+1$ in $\mathcal M_X$ generated by $\overline{\mathbf d}$ via $\overline g$: Corollary~\ref{C6-5} then gives a generating sequence for $U(W,\nu,\lambda+1)$, contrary to the definition of $\beta$.
    Let $\overline g = \pi^{-1}(g)$. It is enough to show that 
    \begin{itemize}
        \item[a.] For all $\mu_\eta\in\overline{\mathbf{d}}$, $\overline g(\mu_\eta)\leq o^{\mathcal M_X}(\mu_\eta)$;
        \item[b.] For all $x\in\mathcal P(\mu)\cap \mathcal M_X$, there exists some $\delta<\mu$ such that for all $\mu_\eta\in \overline{\mathbf{d}}\setminus\delta$, $x\in U(\mathcal M_X,\mu,\overline\lambda+1)$ if and only if
        $$\left\{\begin{aligned}&\mu_\eta\in x; &\text{ if }\overline{g}(\mu_\eta) = o^{\mathcal M_X}(\mu_\eta);\\
        &x\cap \mu_\eta\in U(\mathcal M_X,\mu_\eta,\overline g(\mu_\eta));&\text{ if }\overline g(\mu_\eta)<o^{\mathcal M_X}(\mu_\eta).\end{aligned}\right.$$
    \end{itemize}

    Clause a follows from the fact that
    $$U(\mathcal M^{\mathcal T}_{\eta},\mu_\eta,\overline\lambda_\eta) = U(\mathcal M_X,\mu_\eta,\overline\lambda_\eta),$$
    and thus $o^{\mathcal M_X}(\mu_\eta)\geq \overline\lambda_\eta+1$. 
    To validate clause b, we fix some $x\in \mathcal P(\mu)\cap \mathcal M_X$. Let $\mu_\eta\in \overline{\mathbf d}$ be large enough such that $x\in\mathrm{ran}(i^{\mathcal T}_{\eta,\theta})$ and let 
    $$x_\eta = x\cap \mu_\eta = (i^{\mathcal T}_{\eta,\theta})^{-1}(x).$$
    Suppose $\overline g(\mu_\eta) = o^{\mathcal M_X}(\mu_\eta)$. Then 
    $$o^{\mathcal M^{\mathcal T}_{\eta+1}}(\mu_\eta) = o^{\mathcal M_X}(\mu_\eta) = \overline{\lambda}_\eta+1,$$
    which, by Lemma \ref{L6-3}, tells us that $E^{\mathcal T}_{\eta}$ is equivalent with $U(\mathcal M^{\mathcal T}_{\eta},\mu_\eta,\overline{\lambda}_\eta+1)$, and $\eta+1 = \min(\eta,\theta)_{\mathcal T}$. Thus,
    $$x\in U(\mathcal M_X,\mu,\overline\lambda+1)\iff x_\eta\in U(\mathcal M^{\mathcal T}_\eta,\mu_\eta,\overline\lambda_\eta+1)\iff \mu_\eta\in x.$$
    The other part can be shown similarly.
\end{proof}

Let $C^*$ be the collection of all $\pi(\mu_\eta)\in C$ such that $\eta+1=\min(\eta,\theta)_{\mathcal T}$ and $E^{\mathcal T}_{\eta}$ is equivalent to $U(\mathcal M^{\mathcal T}_{\eta},\mu_\eta,\overline\lambda_{\eta})$. 
By Lemma~\ref{L6-7}, every generating sequence contained in $C$ with generating function in $X$ has a tail in $C^*$.
The only remaining argument against the theorem is the following one.

\begin{lemma}\label{L6-8}
    There exists a $\delta<\nu$ such that $\mathrm{type}(C^*\setminus \delta) = \omega$.
\end{lemma}

\begin{proof}
    Let $\overline{C}^* = \pi^{-1}[C^*]$. 
    If the lemma fails, $\lim \overline C^*$ is cofinal in $\mu$.
    Let $\overline C' = \lim \overline C^*$.
    Notice that since $\overline C$ is a club in $\mu$, $\overline C'\subseteq \overline C$. 
    Pick a cofinal subsequence $\overline C''\subseteq\{\gamma\in\overline C':\mathrm{cf}(\gamma)=\omega\}$ with $\mathrm{type}(\overline C'') = \mathrm{cf}(\nu)$ and let $C'' = \pi[\overline C'']$. 
    Then $C''\in X$.
    \begin{claim}\label{c6-8-1}
        For each $\mu_\eta\in\overline C''$ with $\overline\lambda\in \mathrm{ran}(i^{\mathcal T}_{\eta,\theta})$, $U(\mathcal M_X,\mu_\eta,\overline\lambda_\eta)$ exists.
    \end{claim}
    \begin{proof}
        Fix such a $\mu_\eta$ and choose an increasing $\omega$-sequence $\overline{\mathbf c}\subseteq\overline C^*\cap\mu_\eta$ cofinal in $\mu_\eta$. Put $\mathbf c=\pi[\overline{\mathbf c}]$. Then $\mathbf c\in X$, and $\overline{\mathbf c}$ generates $U(\mathcal M^{\mathcal T}_\eta,\mu_\eta,\overline\lambda_\eta)$ via $o^{\mathcal M^{\mathcal T}_\eta}\upharpoonright\overline{\mathbf c}$.
        
        Put $U=U(\mathcal M^{\mathcal T}_\eta,\mu_\eta,\overline\lambda_\eta)$. By (1) in the proof of Lemma~\ref{L6-3} and agreement below $\mu$, $\mathcal M^{\mathcal T}_\eta$ and $W_X$ have the same subsets of $\mu_\eta$. Thus $\overline{\mathbf c}\in N_X$ defines $U$ over $W_X$ by eventual membership.
        Apply the absorption argument in the proof of Corollary~\ref{C6-5}: the tree in the first correction of the Appendix contains $\mathcal M^{\mathcal T}_\eta$, its final model extends $K_X$, and the iteration from $K_X$ to $W_X$ is above $\mu_\eta$. Hence the extender corresponding to $\pi(U)$ is on the $W$-sequence. Pulling back and identifying the rank as in that proof gives $U=U(W_X,\mu_\eta,\overline\lambda_\eta)=U(\mathcal M_X,\mu_\eta,\overline\lambda_\eta)$.
    \end{proof}
    The index of this measure satisfies $\overline\lambda_\eta^*<(\mu_\eta^{++})^{\mathcal M_X}\leq\mathrm{lh}(E_\eta^{\mathcal T})$, by normal-measure indexing and the common initial segment at $\eta$.
    Using a similar technique as in the proof of Lemma \ref{L6-7}, $C''$ generates $U(W,\nu,\lambda+1)$, which is contradictory to the definition of $\lambda$.
\end{proof}

Therefore, $\mathrm{cf}(\nu) = \omega$. Let $\delta$ be as above and $\mathbf{c}=C^*\setminus \delta$. Then $\mathbf{c}$ is a Prikry sequence for $U(W,\nu,\lambda)$. Our earlier lemmas in this section together imply that:
\begin{lemma}\label{L6-9}
    For every generating sequence $\mathbf d$ for $U(W,\nu,\lambda)$, the set $\mathbf d\setminus\mathbf c$ is bounded in $\nu$. In particular, it is finite when $\mathbf d$ is a Prikry sequence.
\end{lemma}

Therefore, $\mathbf{c}$ is the maximal Prikry sequence of $U(W,\nu,\lambda)$.
Thus, Theorem \ref{T6-1} is shown.
\ChapterLabel{Ch7}
\chapter{Generating Sequence without Closure}

\begingroup
\def\cf{\operatorname{cf}}
\def\otp{\operatorname{otp}}
\def\ran{\operatorname{ran}}
\def\M{\mathcal M}
\def\Pow{\mathcal P}
\def\res{\mathbin{\upharpoonright}}

In this chapter we generalize the result of Chapter~\ref{Ch6} by removing the closure assumption
on $|\nu|$. For this we strengthen the hypotheses as follows.

\begin{theorem}\label{T7-1}
Assume there is no inner model with a Woodin cardinal. Suppose that in the Jensen--Steel
core model $K$, $\nu>\omega_2$ is a cardinal such that
\[
 \cf(\nu)^+<|\nu|\leq\nu=\cf^K(\nu).
\]
Assume further that
\[
 \sup\{\alpha<\nu\mid K\models o(\alpha)=\alpha^{++}\}<\nu
 \qquad\text{and}\qquad \beta=\lambda+1<(\nu^+)^V.
\]
Then $\cf(\nu)=\omega$ and there exists a maximal Prikry sequence for $U(K,\nu,\lambda)$.
\end{theorem}

As in Chapter~\ref{Ch6}, $\beta$ is the least $\xi<o^K(\nu)$ such that $U(K,\nu,\xi)$ has no
generating sequence in $V$, or $o^K(\nu)$ if every such measure has a generating sequence.
Let $\mathbf c=\langle\kappa_i\mid i<\cf(\nu)\rangle$ be a generating sequence for
$U(K,\nu,\lambda)$ with generating function $f$. Our aim is to show that $\mathbf c$ is a
Prikry sequence for $U(K,\nu,\lambda)$ and hence that $\cf(\nu)=\omega$.
Write $U=U(W,\nu,\lambda)$, where $W$ is the
soundness witness of Chapter~\ref{Ch2}. 

The first step is to choose $X\prec(H_{\Omega^+},\in,U_\Omega)$. Following the construction
in Chapter~\ref{Ch5}, fix a weakly internal approachable chain
$\langle Y_i\mid i<\cf(\nu)^+\rangle$ of elementary substructures with limit $X$.
Choose the chain with the following properties.
\begin{itemize}
\item For $i<j<\cf(\nu)^+$, $Y_i\cup\{Y_i\}\subseteq Y_j$. Each $Y_i$ is the limit
of a weakly internal approachable chain of elementary substructures of length $\cf(\nu)^+$.
\item Let $Z$ be $X$ or, for some $i$, either $Y_i$ or a member of the chain with
limit $Y_i$. Then
\[
 |Z|=\cf(\nu)^+,\quad \cf(\nu)^+ +1\subseteq Z,\quad
 \{\nu,\Omega_0,W,\lambda,\mathbf c,f\}\subseteq Z,\quad \sup(Z\cap\nu)=\nu.
\]
\item Each such $Z$ satisfies Requirements~1 and~2 of Chapter~\ref{Ch2}. The answers to
List~B of Chapter~\ref{Ch5} are independent of $Z$. 
\end{itemize}

For each such $Z$ use the notation of Lemma~\ref{L2-10}. Thus $\pi_Z:N_Z\to Z$ is the
uncollapse, $\mu_Z=\pi_Z^{-1}(\nu)$, and $\mu_\eta^Z$ is the critical point of
$i^{\mathcal T_Z}_{\eta,\theta_Z}$ for $\eta\in I_Z$. We have
$i^{\mathcal T_Z}_{\eta,\theta_Z}(\mu_\eta^Z)=\mu_Z$ and
$C_Z=\{\pi_Z(\mu_\eta^Z)\mid\eta\in I_Z\}$. When $Z$ is fixed we write
$\mathcal T,\theta,\mu_\eta$ for $\mathcal T_Z,\theta_Z,\mu_\eta^Z$, respectively,
and use $M_\eta=\mathcal M_\eta^{\mathcal T}$ and $i_{\eta,\theta}=i^{\mathcal T}_{\eta,\theta}$.
Put $s_\eta=o^{\M_Z}(\mu_\eta)$.

\begin{lemma}\label{L7-normal}
For all sufficiently large $\eta\in I_Z$ we have
$\eta+1=\min(\eta,\theta)_{\mathcal T}$ and
$E_\eta^{\mathcal T}\simeq U(M_\eta,\mu_\eta,s_\eta)$. Moreover,
\begin{equation}\label{coherence}
 U(M_\eta,\mu_\eta,r)=U(\M_Z,\mu_\eta,r)\quad(r<s_\eta),
 \qquad \pi_Z(s_\eta)=o^K(\pi_Z(\mu_\eta)).
\end{equation}
\end{lemma}
\begin{proof}
For $\eta\in I_Z$, let $\zeta+1=\min(\eta,\theta)_{\mathcal T}$. If $E_\zeta^{\mathcal T}$
has a generator above $\mu_\eta$, Lemma~\ref{L6-3} gives $\M_Z\models o(\mu_\eta)=\mu_\eta^{++}$.
Since $\mu_Z$ is measurable in $\M_Z$, $(\mu_\eta^{++})^{\M_Z}<\mu_Z$, so agreement
below $\mu_Z$ gives the same equality in $W_Z$. Applying $\pi_Z$ and using the agreement
of $W$ and $K$ below $\nu$, we obtain
\[
 K\models o(\pi_Z(\mu_\eta))=(\pi_Z(\mu_\eta))^{++}.
\]
Fix a bound in $Z\cap\nu$ for the ordinals below $\nu$ satisfying this equality in $K$.
For all sufficiently large $\eta\in I_Z$, $\pi_Z(\mu_\eta)$ exceeds the bound, so
$E_\zeta^{\mathcal T}$ has no generator above $\mu_\eta$. Lemma~\ref{L6-3} then gives
$\zeta=\eta$ and $E_\eta^{\mathcal T}$ is equivalent to its derived normal measure.

A further generator of a total extender at $\mu_\eta$ on the $M_{\eta+1}$-sequence
would give $o(\mu_\eta)=\mu_\eta^{++}$ in its ultrapower by the proof of Lemma~\ref{L6-3}.
By elementarity, $o(\alpha)=\alpha^{++}$ would then hold for unboundedly many
$\alpha<\mu_\eta$ in $M_{\eta+1}$, hence in $W_Z$ by agreement below $\mu_\eta$.
Taking such an $\alpha$ above the preimage of the fixed bound gives a contradiction.

Thus every total extender at $\mu_\eta$ on the $M_{\eta+1}$-sequence is equivalent
to a normal measure, with index below
$(\mu_\eta^{++})^{M_{\eta+1}}=\operatorname{lh}(E_\eta^{\mathcal T})$.
By coherence, this list consists exactly of the total extenders at $\mu_\eta$
preceding $E_\eta^{\mathcal T}$ in $M_\eta$. The remaining branch map fixes the list,
so its order type $s_\eta$ is the rank of $E_\eta^{\mathcal T}$, and the measure of
each rank $r<s_\eta$ is the same in $M_\eta$ and $\M_Z$. Finally, agreement below
$\mu_Z$ gives $s_\eta=o^{W_Z}(\mu_\eta)$, whence
$\pi_Z(s_\eta)=o^W(\pi_Z(\mu_\eta))=o^K(\pi_Z(\mu_\eta))$.
\end{proof}

\begin{lemma}\label{L7-capture}
If $d\in Z$ is a generating sequence for $U$ of order type $\cf(\nu)$, then a tail of
$d$ is contained in $C_Z$.
\end{lemma}
\begin{proof}
By Lemma~\ref{L7-normal}, the club constructed in the proof of Lemma~\ref{L3-4} omits every
sufficiently large point of $(Z\cap\nu)\setminus C_Z$. 
Since $d\in Z$ and $\otp(d)=\cf(\nu)\subseteq Z$, we have $d\subseteq Z$. Since $d$ is a generating sequence, it is almost contained in every club in $\mathcal P(\nu)\cap K$ and hence in $C_Z$.
\end{proof}

\begin{lemma}\label{L7-agreement}
If $Z=X$ or $Z=Y_i$ for some $i<\cf(\nu)^+$, then
\begin{equation}\label{powers}
 \Pow(\mu_Z)\cap \M_Z=\Pow(\mu_Z)\cap W_Z,
 \; U(\M_Z,\mu_Z,r)=U(W_Z,\mu_Z,r)\;(r<o^{W_Z}(\mu_Z)).
\end{equation}
\end{lemma}
\begin{proof}
In Lemma~\ref{L6-2}, closure under $\cf(\nu)$-sequences puts an arbitrary cofinal subsequence
of $C_X$ of order type $\cf(\nu)$ in $X$. Here we use the fixed generating sequence
$\mathbf c$. In the construction of Chapter~\ref{Ch5}, take the indices of the $Y_i$ to be
limit points of the stationary set on which List~B is constant, so that the weakly internal
approachable chain with limit $Y_i$ can be chosen from the same stationary set.

We write $B_Z$ for the image under $\pi_Z$ of the set chosen in
\cite[Section~4, pp.~22--23]{MS2}.
If $\cf(\nu)=\omega$, take a tail of $\mathbf c$ contained in $C_Z$ by
Lemma~\ref{L7-capture}. If $\cf(\nu)>\omega$, use the common club
$\lim(\mathbf c)$ and put $B_Z=C_Z\cap\lim(\mathbf c)$.
Each point of the club is the supremum of an initial segment of $\mathbf c$ in $Z$.
Since every index of the segment belongs to $Z$, its supremum belongs to $Z$ and
$\pi_Z$ is continuous at its preimage. By Lemma~\ref{L7-capture} and
closedness of the collapsed critical sequence, $C_Z$ contains a tail of the club.
Thus in both cases each $B_Z$ is almost contained in every model's critical sequence,
as required after Lemma~\ref{L5-6}. The lemma is thus proven by the same argument as in Corollaries~\ref{C5-7} and~\ref{C5-8}.
\end{proof}

Fix $Z=X$ or $Z=Y_i$. By \eqref{powers}, the collapse of $U$ is
$\bar U=U(\M_Z,\mu_Z,\bar\lambda)$, where $\bar\lambda=\pi_Z^{-1}(\lambda)$.
For sufficiently large $\eta\in I_Z$, let $\bar\lambda_\eta$ be the preimage of
$\bar\lambda$ under $i_{\eta,\theta}$, as in Chapter~\ref{Ch6}.

Our aim is to prove that $g(\alpha)=o^K(\alpha)$ on a tail of $d$ whenever
$(d,g)\in Z$ generates $U$ and $\otp(d)=\cf(\nu)$. Lemma~\ref{L7-3} first identifies $g(\pi_Z(\mu_\eta))$ with $\pi_Z(\bar\lambda_\eta)$.
In the proof of Lemma~\ref{L7-2}, we use this equality to rule out an unbounded set
of points where $g(\alpha)<o^K(\alpha)$.

\begin{lemma}\label{L7-3}
Let $(d,g)\in Z$ generate $U$ with $\otp(d)=\cf(\nu)$. Then on a tail of $d$,
\[
 g(\pi_Z(\mu_\eta))=\pi_Z(\bar\lambda_\eta)\leq o^K(\pi_Z(\mu_\eta)).
\]
\end{lemma}
\begin{proof}
By Lemma~\ref{L7-capture} we may restrict to a tail of $d$ contained in $C_Z$.
Write $\bar g=\pi_Z^{-1}(g)$. At each $\pi_Z(\mu_\eta)$ in this tail put
$r_\eta=\bar g(\mu_\eta)\leq s_\eta$ and
\[
 V_\eta=U(\M_Z,\mu_Z,i_{\eta,\theta}(r_\eta)).
\]
Fix $x\in\bar U$. For all sufficiently large $\eta$ we have
$x=i_{\eta,\theta}(x\cap\mu_\eta)$. Since $i_{\eta,\theta}(\mu_\eta)=\mu_Z$,
the definition of $V_\eta$ and elementarity give
\[
 x\in V_\eta\quad\Longleftrightarrow\quad
 x\cap\mu_\eta\in U(M_\eta,\mu_\eta,r_\eta).
\]
If $r_\eta=s_\eta$, Lemma~\ref{L7-normal} identifies the measure on the right
with the one derived from $E_\eta^{\mathcal T}$. Thus the right side is equivalent
to $\mu_\eta\in i_{\eta,\eta+1}(x\cap\mu_\eta)$, or equivalently to $\mu_\eta\in x$,
since the remaining branch map fixes $\mu_\eta$. If $r_\eta<s_\eta$, the right side
is equivalent to $x\cap\mu_\eta\in U(W_Z,\mu_\eta,r_\eta)$ by \eqref{coherence}
and agreement below $\mu_Z$. Since $(d,g)$ generates $U$, $x\in V_\eta$ at every sufficiently large point
of $d$.

We show that eventually $V_\eta=\bar U$. Otherwise choose $x_\eta\in\bar U\setminus V_\eta$ at
cofinally many points of $d$. As in the proof of Claim~\ref{c6-6-1}, we use weak covering and diagonal intersection
to find a set in $\bar U$ that belongs to none of the $V_\eta$ at the chosen points.
By \eqref{powers}, $A_\eta=\pi_Z(x_\eta)\in U\cap Z$. In $W\cap Z$
fix an enumeration $e:\sigma\twoheadrightarrow\Pow(\nu)^W$, where $\sigma=(\nu^+)^W$. 
The least indices of all the $A_\eta$ lie in a single constituent $Y\in Z$, since there are at most $\cf(\nu)$ indices and the chain has cofinality $\cf(\nu)^+$. 
By weak covering
(Jensen--Steel \cite[Theorem~1.1(5)]{JS}), $\cf^V(\sigma)\geq|\nu|>|Y|$. Thus $b=\sup(Y\cap\sigma)+1\in Z\cap\sigma$.
Enumerate $e[b]\cap U$ in $W\cap Z$ in length $\nu$. Its diagonal
intersection $A$ belongs to $U\cap Z$. For each $\eta$ the least index of $A_\eta$ in this
enumeration belongs to $Z\cap\nu$, so $\pi_Z^{-1}(A)\setminus x_\eta$ is bounded in $\mu_Z$.
Uniformity gives $\pi_Z^{-1}(A)\notin V_\eta$ at every bad point, contradicting the preceding
paragraph. Equality of the measures gives $i_{\eta,\theta}(r_\eta)=\bar\lambda$ by \eqref{powers} and
uniqueness of the rank. Thus $\bar\lambda_\eta=r_\eta\in N_Z$, and applying $\pi_Z$ gives the stated equality.
\end{proof}

\begin{lemma}\label{L7-2}
Let $Z=X$ or $Z=Y_i$ and suppose $(d,g)\in Z$ generates $U$ with $\otp(d)=\cf(\nu)$.
Then for all sufficiently large $\pi_Z(\mu_\eta)\in d$,
\[
  g(\pi_Z(\mu_\eta))=\pi_Z(\bar\lambda_\eta)=o^K(\pi_Z(\mu_\eta)).
\]
At these points $E_\eta^{\mathcal T}\simeq U(M_\eta,\mu_\eta,\bar\lambda_\eta)$.

If $(d,g)\in Y_0$, then for some $\gamma\in d$ the function
$\pi_X(\mu_\eta^X)\mapsto\pi_X(\bar\lambda_\eta^X)$ on $d\setminus\gamma$
equals $g\res(d\setminus\gamma)\in Y_0$.
For each $i<\cf(\nu)^+$, every sufficiently large $\alpha\in d$ can be written as
$\alpha=\pi_X(\mu_\eta^X)=\pi_{Y_i}(\mu_\zeta^{Y_i})$, and then
\[
  \pi_X(\bar\lambda_\eta^X)=g(\alpha)=\pi_{Y_i}(\bar\lambda_\zeta^{Y_i}).
\]
\end{lemma}
\begin{proof}
By Lemma~\ref{L7-3} it remains to show that $d'=\{\alpha\in d:g(\alpha)<o^K(\alpha)\}$ is bounded.
Suppose otherwise. Both $d'$ and the function $\alpha\mapsto g(\alpha)+1$ on $d'$ belong to $Z$. By Lemma~\ref{L7-3}, after removing a bounded part we have $\pi_Z^{-1}(g)(\mu_\eta)=\bar\lambda_\eta$ and
$\bar\lambda_\eta+1\leq s_\eta$. Thus $U(M_\eta,\mu_\eta,\bar\lambda_\eta+1)$ exists
by Lemma~\ref{L7-normal}, and $U^*=U(\M_Z,\mu_Z,\bar\lambda+1)$ exists by elementarity.
Fix $x\in\Pow(\mu_Z)\cap W_Z$. For all sufficiently large $\pi_Z(\mu_\eta)\in d'$,
we have $x=i_{\eta,\theta}(x\cap\mu_\eta)$ and
$$
  x\in U^* \quad\Longleftrightarrow\quad
  \left\{\begin{aligned}
    &x\cap\mu_\eta\in U(W_Z,\mu_\eta,\bar\lambda_\eta+1)
      &&\quad\text{if } \bar\lambda_\eta+1<s_\eta,\\
    &\mu_\eta\in x
      &&\quad\text{if } \bar\lambda_\eta+1=s_\eta.
  \end{aligned}\right.
$$

Both $\pi_Z^{-1}(d')$ and the function $\xi\mapsto\pi_Z^{-1}(g)(\xi)+1$ on
$\pi_Z^{-1}(d')$ belong to $N_Z$. By the above equivalence and \eqref{powers},
$\pi_Z^{-1}(d')$ generates $U^*$ over $W_Z$ via this function.
If $\lambda+1<o^K(\nu)$, then \eqref{powers} gives $U^*=U(W_Z,\mu_Z,\bar\lambda+1)$.
By elementarity, $d'$ generates $U(W,\nu,\lambda+1)$ via $\alpha\mapsto g(\alpha)+1$,
contrary to the choice of $\beta$.
If $\lambda+1=o^K(\nu)$, the comparison in the proof of Corollary~\ref{C6-5} gives $U^*\in W_Z$. But by \eqref{powers} every normal measure on $\mu_Z$
in $W_Z$ already occurs in $\M_Z$ at rank at most $\bar\lambda$, whereas $U^*$ has rank $\bar\lambda+1$.
This proves the asserted equality, and \eqref{coherence} gives
$s_\eta=\bar\lambda_\eta$ there. Lemma~\ref{L7-normal} therefore gives
$E_\eta^{\mathcal T}\simeq U(M_\eta,\mu_\eta,s_\eta)=U(M_\eta,\mu_\eta,\bar\lambda_\eta)$.

Finally suppose $(d,g)\in Y_0$. For $Z=X$, choose $\gamma\in d$ above the exceptional
points. Since $d\in Y_0$ and $\otp(d)=\cf(\nu)\subseteq Y_0$, we have $\gamma\in Y_0$
and hence $g\res(d\setminus\gamma)\in Y_0$.
For each $i$, the pair also belongs to $Y_i$. By Lemma~\ref{L7-capture}, a tail of $d$
is contained in $C_X\cap C_{Y_i}$. Applying the first assertion to each model gives
the stated equality for all $\alpha$ on this tail above both bounds.
\end{proof}

From now on fix $Z=X$, so $\mu_\eta=\mu_\eta^X$. Let $P$ consist of the points $\pi_X(\mu_\eta)$ such that $E_\eta^{\mathcal T}\simeq U(M_\eta,\mu_\eta,\bar\lambda_\eta)$.
By Lemma~\ref{L7-normal}, on a tail of $I_X$ these are
precisely the points where $s_\eta=\bar\lambda_\eta$. Lemma~\ref{L7-2} puts a tail of $\mathbf c$ in $P$. Therefore $P$ is cofinal in $\nu$
and every $A\in U\cap X$ contains a tail of $P$.

We say that $b\subseteq P$ is an $\omega$-block if $\otp(b)=\omega$, and that a sequence
$(b_i)$ of such blocks is increasing if $\sup b_i\leq\min b_j$ whenever $i<j$.

\begin{lemma}\label{L7-blocks}
There is no increasing sequence $(b_i:i<\cf(\nu))\in X$ of $\omega$-blocks in $P$
whose suprema form a strictly increasing cofinal sequence in $\nu$.
Consequently, $\cf(\nu)=\omega$.
\end{lemma}
\begin{proof}
Suppose such a sequence exists and put $\rho_i=\sup b_i$. Since $\cf(\nu)\subseteq X$,
each $b_i$ and its increasing enumeration belong to $X$. Hence $\pi_X^{-1}[b_i]$ is
cofinal in $\pi_X^{-1}(\rho_i)$. As $b_i\subseteq P\subseteq C_X$, the closedness of
$\{\mu_\zeta:\zeta\in I_X\}$ gives $\pi_X^{-1}(\rho_i)=\mu_\eta$ for some $\eta\in I_X$.
After discarding an initial segment, we may assume that $\bar\lambda_\eta$ is defined at every such $\eta$.

By the proof of Claim~\ref{c6-8-1}, $\pi_X^{-1}[b_i]$ is a Prikry sequence for
$U(W_X,\mu_\eta,\bar\lambda_\eta)$. Put $r_i=\pi_X(\bar\lambda_\eta)$.
Then $b_i$ is a Prikry sequence for $U(W,\rho_i,r_i)$. The Mathias criterion
determines this measure uniquely from $b_i$ and $W$, so the function $\rho_i\mapsto r_i$
belongs to $X$.

Every $A\in U\cap X$ contains a tail of $P$. Thus for all sufficiently large $i$ we have
$b_i\subseteq A$ and hence $A\cap\rho_i\in U(W,\rho_i,r_i)$. Applying this also to
complements gives the generating criterion in $X$. By elementarity, $(\rho_i)$ generates
$U$ via $\rho_i\mapsto r_i$. Lemma~\ref{L7-2} now gives $r_i=o^K(\rho_i)$ eventually,
contradicting $r_i<o^W(\rho_i)=o^K(\rho_i)$ for every $i$.

If $\cf(\nu)>\omega$, let $e:\cf(\nu)\to \mathbf c$ be the increasing enumeration. The blocks
\[
 b_i=\{e(\omega\cdot i+n+1):n<\omega\}\qquad(i<\cf(\nu))
\]
form a sequence in $X$. Their suprema are below $\nu$, and all sufficiently late blocks lie in
$P$. Removing the initial blocks gives the contradiction just proved. Hence $\cf(\nu)=\omega$.
\end{proof}

\begin{lemma}\label{L7-countable}
For every countable $B\in X$ with $B\subseteq\nu$, there are $A\in U\cap X$ and
$\gamma\in X\cap\nu$ such that $(B\cap A)\setminus\gamma=(B\cap P)\setminus\gamma$.
Moreover, there exists $\delta\in X\cap\nu$ such that $P\setminus\delta\in X$
and $\otp(P\setminus\delta)=\omega$.
\end{lemma}
\begin{proof}
Every $A\in U\cap X$ contains a tail of $P$, so we need one omitting all sufficiently
large points of $B\setminus P$. For each such $\alpha$, let $\mu_\eta$ be the last critical
point in the sequence at or below $p=\pi_X^{-1}(\alpha)$.
Choose $x_\alpha\in\bar U\cap\ran(i_{\eta,\theta})$ omitting $p$.
If $p>\mu_\eta$, use the closure points of a function in this range whose value at
$\mu_\eta$ exceeds $p$, as in the proof of Claim~\ref{c5-9-1}.
If $p=\mu_\eta$, then $\alpha\notin P$, so we may take the image of a set in
$U(M_\eta,\mu_\eta,\bar\lambda_\eta)$ outside the measure derived from $E_\eta^{\mathcal T}$.

Applying the diagonal intersection from the proof of Lemma~\ref{L7-3} to the countably many
$\pi_X(x_\alpha)$ gives $A\in U\cap X$ with $\pi_X^{-1}(A)\setminus x_\alpha$ bounded in
$\mu_X$ for each $\alpha$. For large $\eta$, $\pi_X^{-1}(A)$ also lies in
$\ran(i_{\eta,\theta})$, so the least bound lies below $\mu_\eta\le p$. Hence $\alpha\notin A$,
giving the required equality above some $\gamma\in\mathbf c$.

By Lemma~\ref{L7-blocks}, $\cf(\nu)=\omega$. Given a countable $T\subseteq P$, take
$j<\omega_1$ with $T\subseteq Y_j$ and a surjection $e:\omega_1\to Y_j$ in $X$.
Then $T\subseteq B=e[\xi]\cap\nu$ for some $\xi<\omega_1$.
Applying the first assertion to $B\in X$ gives $D\in X$ with
\[
  T\setminus\gamma\subseteq D=(B\cap A)\setminus\gamma\subseteq P.
\]

If $P$ had no tail of order type $\omega$, we could choose an increasing $\omega$-sequence
of $\omega$-blocks in $P$ whose suprema are cofinal in $\nu$. For their union $T$,
the set $D$ contains all sufficiently late blocks. By elementarity, $X$ contains such
a sequence in $D\subseteq P$, contrary to Lemma~\ref{L7-blocks}.
Now take $T$ to be a tail of $P$ of order type $\omega$. For $D$ and $\gamma$ as above,
choose $\delta\in X\cap\nu$ above $\gamma$ and $\min T$.
Then $P\setminus\delta=D\setminus\delta\in X$.
\end{proof}

\begin{proof}[Proof of Theorem~\ref{T7-1}]
By Lemmas~\ref{L7-blocks} and~\ref{L7-countable}, $\cf(\nu)=\omega$ and some tail $C$ of
$P$ of order type $\omega$ belongs to $X$. Every $A\in U\cap X$ contains a tail of $C$,
so elementarity gives the Mathias criterion over $W$. If a Prikry sequence for $U$ had
infinitely many points outside $C$, there would be such a sequence $d\in X$. With witness
$o^W\res d$ it generates $U$, so Lemma~\ref{L7-2} puts a tail of $d$ in $P$, hence in $C$.
Thus $C$ is maximal.
\end{proof}
\endgroup
\appendix
\ChapterLabel{ChA}
\chapter{Appendix}

For the following lemma, fix the structure $X$ constructed before
Lemma~\ref{L5-4} and write
$\mu=\mu_X$. In the notation of Chapter~\ref{Ch2}, the following
generalization of Lemma~\ref{L2-3} gives Lemma~\ref{L5-5}.

\begin{lemma}\label{LA-1}
    Suppose $\mathcal M_X\neq\mathcal Q_X$.
    Then $\mathcal M_X$ is a type II active set mouse.
    Let $\lambda=(\kappa^+)^{\mathcal M_X}$, $F=\dot F^{\mathcal M_X}$, $\kappa=\mathrm{crit}(\dot F^{\mathcal M_X})$, $s = s(\dot F^{\mathcal M_X})\setminus \mu$ and $p = p_1(\mathcal M_X)\setminus \mu$.
    Take a large enough ordinal $\xi<\mu$ such that
    \begin{itemize}
        \item $\lambda < \xi$;
        \item $p_1(\mathcal M_X)\cap\mu\subseteq\xi$;
        \item $s(\dot F^{\mathcal M_X})\cap\mu\subseteq\xi$;
        \item $p$ is $\Sigma_1$-definable over $\mathcal M_X$ from parameters in $\xi\cup s$;
        \item $s$ is $\Sigma_1$-definable over $\mathcal M_X$ from parameters in $\xi\cup p$.
    \end{itemize}
    Then one of the following holds.
    \begin{itemize}
        \item[1.] $\mathcal Q_X$ is a set mouse.
        Let $n=n_X(\mu)$ and put $q=(p_{n+1}(\mathcal Q_X)\setminus\mu)\cup t_\mu$. Then
        $$\mathcal P(\mu)\cap\mathrm{Hull}^{\mathcal M_X}_1(\xi\cup p)
        =\mathcal P(\mu)\cap\mathrm{Hull}^{\mathcal Q_X}_{n+1}(\xi\cup q).$$
        \item[2.] $\mathcal Q_X$ is a weasel.
        If $\Gamma$ is a thick class such that every member of $\Gamma$ is fixed by the embedding
        $$W\xrightarrow{i^{\mathcal T}_{0,\eta(\lambda_\ell)}}\mathcal P_{\lambda_\ell}
        \xrightarrow{\phi_\ell}\mathcal Q_X,$$
        then
        $$\mathcal P(\mu)\cap\mathrm{Hull}^{\mathcal M_X}_1(\xi\cup p)
        =\mathcal P(\mu)\cap\mathrm{Hull}^{\mathcal Q_X}(\xi\cup t_\mu\cup\Gamma).$$
    \end{itemize}
\end{lemma}

\begin{proof}
    Put $D=J_\lambda^{\mathcal M_X}$.
    Let $i:\mathcal Q_\lambda\to\mathcal Q_X$ be the ultrapower map by $F$.
    Since $\mathcal{Q}_X$ is an ultrapower by $F$, 
    $$\mathcal P(\mu)\cap\mathcal M_X=\mathcal P(\mu)\cap\mathcal Q_X.$$

    We first prove the left-to-right inclusion in both 1 and 2.
    Let $B\subseteq\mu$ belong to $\mathrm{Hull}_1^{\mathcal M_X}(\xi\cup p)$.
    By the choice of $\xi$, 
    $$\mathrm{Hull}_1^{\mathcal M_X}(\xi\cup p) = \mathrm{Hull}_1^{\mathcal M_X}(\xi\cup s).$$
    By the argument in the proof of \cite[Lemma~A.1]{MS2}, choose $h\in D$ and $r\in[\xi]^{<\omega}$ such that
    $$B=[r\cup s,h]^D_F=i(h)(r\cup s).$$
    Write $q=q_\mu$ in the set mouse case and $q=t_\mu$ in the weasel case.
    The decomposition in Chapter~\ref{Ch2} gives $q=s\cup i[q_0]$, where $q_0$ consists of the remaining parameters before applying $i$.
    Applying \cite[Lemma~2.6.1]{MSS} along the decomposition of $\mathcal Q_\lambda$, we get
    $h\in\mathrm{Hull}^{\mathcal Q_\lambda}_{n+1}(\kappa\cup q_0)$ in the set mouse case and
    $h\in\mathrm{Hull}^{\mathcal Q_\lambda}(\kappa\cup q_0\cup\Gamma)$ in the weasel case.
    Since $i$ fixes all ordinals below $\kappa$ and each member of $\Gamma$ in the weasel case, $i(h)$ belongs to the hull of $\mathcal Q_X$ on the right, and hence so does $B$.

    For the other direction, we illustrate the proof by showing the second equality, where $\mathcal Q_X$ is a weasel.
    The first equality can be shown similarly.
    Let $B\subseteq\mu$ belong to the hull on the right, say
    $$B=\tau^{\mathcal Q_X}[a,q,\gamma],$$
    where $\tau$ is a Skolem term, $a\in[\xi]^{<\omega}$, and $\gamma\in\Gamma^{<\omega}$.
    The range of the factor map from $\mathrm{Ult}(\mathcal Q_\lambda,F\upharpoonright(\xi\cup s))$ to $\mathcal Q_X$ contains every member of $\xi\cup s$ and of $i[\mathcal Q_\lambda]$, so it contains $a$, $q$, and $\gamma$.
    Thus $B=i(h)(b)$ for some function $h\in \mathcal{Q}_\lambda$ and $b\in[\xi\cup s]^{<\omega}$.
    Since $B\subseteq\mu<\nu(F)<i(\kappa)$, we may assume that the range of $h$ is contained in $\mathcal P(\kappa)$.
    Code $h$ into a subset of $\kappa$.
    Since $\mathcal P(\kappa)\cap D = \mathcal P(\kappa)\cap \mathcal Q_\lambda$, we have $h\in D$.

    It remains to show that $i(h)(b)\in\mathrm{Hull}_1^{\mathcal M_X}(\xi\cup p)$.
    Choose $\kappa<\varepsilon<\lambda$ so that $h\in J^{\mathcal M_X}_\varepsilon$ and $i(h)(b)$ is definable from $D$ and
    $$F\cap\bigl([\nu(F)]^{<\omega}\times J_\varepsilon^{\mathcal M_X}\bigr).$$
    This fragment is uniformly $\Sigma_1$-definable over $\mathcal M_X$ from $\varepsilon$ and $\nu(F)=\max(s)+1$, as we desired.
\end{proof}

The left-to-right inclusion also holds for ordinals below $\mu$ when the members of any $D_0\subseteq\mu$ are added to both hulls.
Indeed, the finitely many ordinals from $D_0$ used in a definition can be included in the support $r$ in the same argument.

The following two corrections to \cite{MS2} are due to the first author.

\subsection*{First correction} Fix $X$ as in Chapter~\ref{Ch2}. We replace Requirement 1 of \cite{MS2} by the following.

For each $\nu<\Omega_1<\Omega$ in $X$, let $W_1\in X$ be
an $A_0$-soundness witness for $K\upharpoonright\Omega_1$. Write $\overline\Omega_1=\pi_X^{-1}(\Omega_1)$, 
$\overline W_1=\pi_X^{-1}(W_1)$ and $\delta = \mathrm{crit}(\pi_X)$. Then $$(\delta^+)^{\overline W_1}<(\delta^+)^{W_1},$$
and the comparison of $\overline W_1$ with $W_1$ uses no extender
of length below $\overline\Omega_1$ on the $\overline W_1$ side.

For the canonical witness $W_1$, condition~$(1)_\alpha$ in \cite[Section~3]{MSS} gives a lower bound on the lengths of extenders used on the $\overline W_1$ side of the comparison.
The original Requirement~1 asks that $\overline W$ not move in its comparison with $W$. It is unclear how this follows from the arguments in \cite{MSS} and \cite{MS1} cited in \cite[Section~2.5]{MS2}.
Using the argument in \cite{MS1} to remove the countable closure assumption and treating limit cardinals as in \cite[Theorem~4]{MS2}, we can choose $X$ satisfying the revised Requirement~1 and Requirement~2 (as in Theorem~\ref{T2-2}).
We now show that the revised Requirement~1 suffices for the proof of \cite[Lemma~5.2]{MS2}.

By \cite[Lemma~8.3]{St}, we may take $W$ to be the linear iterate of $K$
by its order-zero measures above $\Omega_0$, as in Chapter~\ref{Ch2}.
Therefore, $W$ is definable from $\Omega_0$ and thus $W\in X$.
Put $\overline K=\pi_X^{-1}(K)$. If $\overline K$ moved in its comparison with $K$,
choose $\Omega_1,W_1\in X$ as above such that $\Omega_0<\Omega_1$ and
$\overline\Omega_1$ is above the index of the first extender used on
$\overline K$. Below the respective cutoffs the witnesses agree with $\overline K$
and $K$, so the comparison of $\overline W_1$ with $W_1$ would make
the same move, contrary to the revised Requirement~1.

Since $\overline K$ has height below $\Omega_0$ and $W$ agrees with $K$ up
to $\Omega_0$, $\overline K$ does not move against $W$ either. Let $\mathcal S$
be the resulting tree on $W$. Then $\overline K\lhd\mathcal M^{\mathcal S}_\infty$.
By elementarity, the linear iteration $\mathcal I$ of $\overline K$ by its
order-zero measures above $\overline\Omega_0=\pi_X^{-1}(\Omega_0)$
has last model $\overline W$. Thus $\mathcal S$ followed by $\mathcal I$ on
this initial segment gives a stack from $W$ to $\overline W$.

Now change $\overline W$ to $W_X$ and $\overline K$ to $K_X$, as $X$ may vary.
\cite[Lemma~5.2]{MS2} requires that $W_X$ is an iterate of $W$ for each $X\in\mathcal H$.
Since $K_X$ and $W_X$ agree below $\overline\Omega_0$, $\mathcal S$
and $\mathcal T_X$ use the same extenders below this cutoff. In
particular, the part of $\mathcal T_X$ producing $\mathcal M_X$
occurs on $\mathcal S$. Thus $W_X$ is such an iterate of $W$ and 
$$\mathrm{Ult}\bigl(W_X,E(\mathcal M_X,\mu_X,0)\bigr)$$
is a generalized iterate of $W$.

\subsection*{Second correction}
The following results correct \cite[Lemma~4.2]{MS2}, which compares the phalanxes
$$(\overrightarrow{\mathcal P}_Y\upharpoonright \mu_Y,\mathcal M_{Y},\mu_Y) \text{ and } (\overrightarrow{\mathcal P}_Y\upharpoonright \mu_Y,\mathcal S_{XY},\mu_Y).$$
Fix $X,Y\in\mathcal G$ with $X\subseteq Y$ and $\Omega_0\in X$.
Write $\mathcal U=\mathcal U_{XY}$ and $\mathcal V=\mathcal T_{XY}$ for the comparison trees on the $\mathcal M_Y$ and $\mathcal S_{XY}$ sides, respectively.

\begin{lemma}\label{LA-2}
Suppose $\mathcal S_{XY}$ is a set mouse. Then $\mathcal V$ is
trivial. Either $\mathcal S_{XY}\unlhd\mathcal M_Y$, or there is an
extender $E$ on the $\mathcal M_Y$-sequence such that, for
$\kappa=\mathrm{crit}(E)$ and $\lambda=(\kappa^+)^{W_Y}<\mu_Y$,
$$\mathcal S_{XY}=\mathrm{Ult}_n(\mathcal P^Y_\lambda,E).$$
\end{lemma}

\begin{proof}
Put $q=\Pi_{XY}(q_{\mu_X})$. The mouse $\mathcal S_{XY}$ is
$(n+1)$-sound above $\mu_Y$ relative to $q$.
Assume $\mathcal S_{XY}$ is not an initial segment of $\mathcal M_Y$.
The final branch of $\mathcal U$ may start at $\mathcal M_Y$ itself
or at an earlier $\mathcal P^Y_\lambda$ with $\lambda<\mu_Y$.
Write $\mathcal P^Y_\lambda$ for its root in either case, with
$\mathcal P^Y_{\mu_Y}=\mathcal M_Y$.

\begin{claim}\label{cA-2-1}
$\mathrm{root}^{\mathcal V}=\mathcal S_{XY}$.
\end{claim}

\begin{proof}
Suppose the root is $\mathcal P^Y_\xi$ with $\xi<\mu_Y$.
Since $\xi<\mu_Y$, the final branch of $\mathcal V$ moves its root.
Its first edge either starts from a mouse that is not $\omega$-sound
or has critical point at least the $\omega$-projectum of its root.
Thus $\mathcal M^{\mathcal V}_\infty$ is not $\omega$-sound.
The same argument applies if the final branch of $\mathcal U$ moves
its root. Otherwise $\mathcal M^{\mathcal U}_\infty=\mathcal M_Y$.
If $\mathcal M_Y\lhd\mathcal M^{\mathcal V}_\infty$, then
$\mathcal M_Y$ is an $\omega$-sound set mouse. But the branch of
$\mathcal T_Y$ from $W$ to this set mouse drops, and every model
after the drop on that branch, including $\mathcal M_Y$, fails to be
$\omega$-sound. Hence neither last model can be a proper initial
segment of the other, so the two last models are equal.

The diagram now gives two iterations of $W$ to the same model.
Here $\eta_Y(\alpha)$ denotes the stage $\eta(\alpha)$ of Chapter~\ref{Ch2}
for the structure $Y$.
\begin{center}
\begin{tikzcd}
 & \mathcal M^{\mathcal T_Y}_{\eta_Y(\lambda)}
    \arrow[draw=none]{r}[sloped,auto=false]{\unrhd}
 & \mathcal P^Y_\lambda \arrow[r,"\mathcal U"]
 & \mathcal M^{\mathcal U}_\infty \arrow[dd,equal] \\
 W \arrow[ru,"\mathcal T_Y"] \arrow[rd,"\mathcal T_Y"'] &&& \\
 & \mathcal M^{\mathcal T_Y}_{\eta_Y(\xi)}
    \arrow[draw=none]{r}[sloped,auto=false]{\unrhd}
 & \mathcal P^Y_\xi \arrow[r,"\mathcal V"]
 & \mathcal M^{\mathcal V}_\infty
\end{tikzcd}
\end{center}
The arrows denote branches and allow drops in model or degree. Apply the
comparison argument of \cite[Lemma~3.14, Claim~2]{MSS} to these two branches.
When both branches have a last drop, the two extenders used there are compatible. Comparing their lengths as in that claim gives the contradiction.
\end{proof}

The argument on the $\mathcal U$-side applies regardless of the root of
$\mathcal V$. If $\mathcal M^{\mathcal U}_\infty$ were a proper initial
segment of $\mathcal M^{\mathcal V}_\infty$, it would be $\omega$-sound.
We ruled this out both when the final branch of $\mathcal U$ moves its root
and when it leaves $\mathcal M_Y$ unmoved. Thus comparison gives
$\mathcal M^{\mathcal V}_\infty\unlhd\mathcal M^{\mathcal U}_\infty$.

\begin{claim}\label{cA-2-2}
$\mathcal V$ is trivial.
\end{claim}

\begin{proof}
Suppose $\mathcal V$ moves. Rearrange $\mathcal U$ as a 
continuation of $\mathcal T_Y\upharpoonright(\theta_Y+1)$.
Now $\mathcal U$ is the tree on $W$ in its comparison with the phalanx ending in $\mathcal S_{XY}$.
Since $\mathrm{root}^{\mathcal V}=\mathcal S_{XY}$, the final branch of
$\mathcal V$ does not drop by universality of $W$.
Its last model is not $\omega$-sound, hence is the common last model $\mathcal N$.
Let $j$ be the branch map on the $\mathcal V$-side. Then $\mathrm{crit}(j)\geq\mu_Y$.
By the fact that $\mathcal S_{XY}$ is $\mu_Y$-sound, the range of $j$ is
$$
 H=\mathrm{Hull}^{\mathcal N}_{n+1}(\mu_Y\cup p_{n+1}(\mathcal N)).
$$

On the $\mathcal U$-side, argue as in \cite[Lemma~3.10, Claim~1]{MSS}.
Suppose $\mathcal U$ has a last drop.
Let $\zeta<_{\mathcal U}\infty$ be the last drop on the main branch, so that
$$\mathcal M^{\mathcal U}_\eta\rhd \mathcal M^{\mathcal U,*}_\zeta\xrightarrow{i} \mathcal M^{\mathcal U}_\infty.$$
It implies that $\mathrm{crit}(i)\geq \mu_Y$, therefore $\mathrm{ran}(i) = H$. 
Otherwise, $\mathcal U$ has no drop.
If $\mathrm{root}^{\mathcal U}=\mathcal M_Y$, the $\mu_Y$-soundness of $\mathcal M_Y$ also gives the same range of last branch map.
If $\mathrm{root}^{\mathcal U}=\mathcal P^Y_\lambda$ with $\lambda<\mu_Y$, let $E$ be the
first extender on the final branch.
The $\mathcal V$-iteration preserves the subsets of $\mu_Y$, so the collapse of $H$
contains $\mathcal P(\mu_Y)\cap\mathcal N$.
If $\nu(E)>\mu_Y$, the factor map argument would then put a code for
$E\upharpoonright\mu_Y$ in the ultrapower by that restriction. Thus
$\nu(E)=\mu_Y$ and then the branch map after $E$ has range $H$. In each case both maps are
the inverse collapse of $H$, contrary to the proof of the Comparison Lemma \cite[Lemma~7.1]{MSt}.
\end{proof}

Now $\mathcal S_{XY}\unlhd\mathcal N=\mathcal M^{\mathcal U}_\infty$.
If the initial segment were proper, its height would be below
$(\mu_Y^+)^{\mathcal N}$, so coherence would give $\mathcal S_{XY}\lhd\mathcal M_Y$.
Thus $\mathcal S_{XY}=\mathcal N$ and $\mathcal N$ is sound above $\mu_Y$.
If $\mathrm{root}^{\mathcal U} = \mathcal M_Y$, soundness would give that $\mathcal U$ is trivial and thus $\mathcal S_{XY}=\mathcal M_Y$. 
Otherwise, $\mathrm{root}^{\mathcal U} = \mathcal P^Y_\lambda$ for some $\lambda<\mu_Y$ and that $\mathcal N = \mathrm{Ult}_n(\mathcal P^Y_\lambda,E)$ with $\nu(E) = \mu_Y$.
By coherence, $E$ is the first extender used on $\mathcal U$ and thus on the $\mathcal M_Y$-sequence. 
\end{proof}

\begin{lemma}\label{LA-3}
Suppose $\mathcal S_{XY}$ is a weasel. Put $\tau=\pi_Y^{-1}(\Omega_0)$.
Both trees have successor length and the same last model.
Neither final branch drops, and $\operatorname{root}^{\mathcal V}=\mathcal S_{XY}$.
If $\operatorname{root}^{\mathcal U}=\mathcal M_Y$, then $\mathcal M_Y$ is a weasel and
$$
  \mathcal M_Y\upharpoonright\tau=\mathcal S_{XY}\upharpoonright\tau.
$$
If $\operatorname{root}^{\mathcal U}=\mathcal P^Y_\lambda$ for some
$\lambda<\mu_Y$, then $\mathcal P^Y_\lambda$ is a weasel and there is an extender
$G$ such that
$$
  \mathrm{Ult}(\mathcal P^Y_\lambda,G)\upharpoonright\tau
    =\mathcal S_{XY}\upharpoonright\tau.
$$
\end{lemma}

\begin{proof}
Write $\mu=\mu_Y$ and put
$$
  t_X=c(\mathcal Q_X)\setminus\mu_X,\qquad
  t_{XY}=\Pi_{XY}(t_X)=c(\mathcal S_{XY})\setminus\mu .
$$
The root argument in Claim~\ref{cA-2-1} gives
$\operatorname{root}^{\mathcal V}=\mathcal S_{XY}$.
Comparing $W$ with the phalanx ending in $\mathcal S_{XY}$ gives the same tree pair $\mathcal U$ and $\mathcal V$.
By universality, the trees have a common last model $\mathcal N$ and neither
final branch drops. In particular, the root of $\mathcal U$ is a weasel.

Suppose first that $\operatorname{root}^{\mathcal U}=\mathcal M_Y$.
Then $\mathcal M_X$ is also a weasel, so $t_{XY}=\emptyset$.
Let $\Gamma$ be a thick class of common fixed points of
$i^{\mathcal U}\colon\mathcal M_Y\to\mathcal N$ and $i^{\mathcal V}\colon\mathcal S_{XY}\to\mathcal N$. Both maps fix every
ordinal below $\mu$, so
$$
i^{\mathcal U}\bigl[\mathrm{Hull}^{\mathcal M_Y}(\mu\cup\Gamma)\bigr]
=\mathrm{Hull}^{\mathcal N}(\mu\cup\Gamma)
=i^{\mathcal V}\bigl[\mathrm{Hull}^{\mathcal S_{XY}}(\mu\cup\Gamma)\bigr].
$$
By the definability property, the hulls in $\mathcal M_Y$ and $\mathcal S_{XY}$ contain
$\mathcal M_Y\upharpoonright\tau$ and $\mathcal S_{XY}\upharpoonright\tau$,
respectively. Both initial segments are therefore the restriction to $\tau$
of the transitive collapse of $\mathrm{Hull}^{\mathcal N}(\mu\cup\Gamma)$.

Now suppose the root is $\mathcal P=\mathcal P^Y_\lambda$ with $\lambda<\mu$.
Let $\xi+1$ be its immediate successor on the main branch and put $E=E^{\mathcal U}_\xi$.
Then $\mathcal M^{\mathcal U}_{\xi+1}=\mathrm{Ult}(\mathcal P,E)$.
Write $\kappa=\mathrm{crit}(E)$ and $h=i^{\mathcal U}_{\xi+1,\infty}$.
Here $\lambda=(\kappa^+)^{W_Y}$, so $\mathcal P$ and $\mathcal M^{\mathcal U}_\xi$
have the same subsets of $\kappa$. Since $\operatorname{lh}(E)>\mu$ and $\mu$
is a limit cardinal in $\mathcal M^{\mathcal U}_\xi$, we have $\mu\leq\nu(E)$.
By normality $h$ fixes every ordinal below $\nu(E)$, and $i^{\mathcal V}$ fixes every ordinal
below $\mu$.

To obtain a common hull, we need to find a parameter $r$ in $\mathcal M^{\mathcal U}_{\xi+1}$ such that
$h(r)=i^{\mathcal V}(t_{XY})$. When $t_{XY}=\emptyset$, take $r=\emptyset$. Now suppose $t_{XY}\ne\emptyset$. Apply the
least-parameter calculation in the
proof of \cite[Lemma~2.4]{MS2} to
$$\mathcal Q_X\xrightarrow{\Pi_{XY}}\mathcal S_{XY}\xrightarrow{i^{\mathcal V}}\mathcal N.$$
It implies that $i^{\mathcal V}(t_{XY})$ is the least
parameter for the hull property of $\mathcal N$ at $\mu$.

Let $r$ be the lexicographically least finite parameter for which
$\mathcal M^{\mathcal U}_{\xi+1}$ has the $r$-hull property at $\mu$.
Its existence follows from the argument of \cite[Lemma~3.11, Claim~1]{MSS}. Write
$r=i_E(f)(a)$ with $a\in[\nu(E)]^{<\omega}$. We claim that $a$ also
witnesses the hull property at $\mu$. Fix a thick class $\Gamma$ of fixed points
of $i_E$ and $A\in\mathcal P(\mu)\cap\mathcal M^{\mathcal U}_{\xi+1}$.
By the $r$-hull property,
$$A=\sigma^{\mathcal M^{\mathcal U}_{\xi+1}}(r,b,c)\cap\mu$$
for some Skolem term $\sigma$, $b\in[\mu]^{<\omega}$ and $c\in[\Gamma]^{<\omega}$.
In $\mathcal P$ define
$$g(x,y)=\sigma^{\mathcal P}(f(x),y,c)\cap\kappa,
\qquad (x\in[\kappa]^{|a|},\ y\in[\kappa]^{|b|}).$$
Since $\mu\leq i_E(\kappa)$, we have $A=i_E(g)(a,b)\cap\mu$.
The function $g$ is coded by a subset of $\kappa$, so definability at $\kappa$
gives $g\in\mathrm{Hull}^{\mathcal P}(\kappa\cup\Gamma)$.
Hence 
$$i_E(g)(a,b)\in\mathrm{Hull}^{\mathcal M^{\mathcal U}_{\xi+1}}(\mu\cup a\cup\Gamma).$$
Let $\pi$ be the collapse of this hull.
Then $A=\pi(i_E(g)(a,b))\cap\mu$ belongs to the collapsed hull,
so $a$ witnesses the hull property at $\mu$.
Hence $r\leq_{\mathrm{lex}}a$, so $h$ fixes $r$ and every smaller parameter.
For these parameters the hull property at $\mu$ is the same in $\mathcal M^{\mathcal U}_{\xi+1}$
and $\mathcal N$, since they have the same subsets of $\mu$ and the hull
terms over a thick class of fixed points pass both ways along $h$.
Thus $h(r)=r=i^{\mathcal V}(t_{XY})$.

In either case, let $G=E\upharpoonright(\mu\cup r)$,  $u\colon\mathcal P\to\mathcal M'=\mathrm{Ult}(\mathcal P,G)$ be the
ultrapower map, $k\colon\mathcal M'\to\mathcal M^{\mathcal U}_{\xi+1}$ the factor map,
$p=k^{-1}[r]$ and $j=h\circ k$. Choose a thick class $\Gamma$ of common
fixed points of $u,j,i^{\mathcal V}$. Both $j$ and $i^{\mathcal V}$ fix every ordinal below $\mu$,
and $j(p)=i^{\mathcal V}(t_{XY})$. Using also the definability property of
$\mathcal S_{XY}$, we obtain
$$
\begin{aligned}
j\bigl[\mathrm{Hull}^{\mathcal M'}(\mu\cup p\cup\Gamma)\bigr]
  &=\mathrm{Hull}^{\mathcal N}(\mu\cup j(p)\cup\Gamma)\\
  &=i^{\mathcal V}\bigl[\mathrm{Hull}^{\mathcal S_{XY}}(\mu\cup t_{XY}\cup\Gamma)\bigr],\end{aligned}$$
and
$$
\mathcal S_{XY}\upharpoonright\tau
  \subseteq\mathrm{Hull}^{\mathcal S_{XY}}(\mu\cup t_{XY}\cup\Gamma).
$$
For the corresponding inclusion in $\mathcal M'$, fix $\beta<\tau$. Write
$\beta=u(g)(a)$ with $a\in[\mu\cup p]^{<\omega}$ and
$g\colon[\kappa]^{|a|}\to\tau$ in $\mathcal P$.
By acceptability and definability at $\kappa$,
$$
g\in\mathcal P\upharpoonright\Omega_0
  \subseteq\mathrm{Hull}^{\mathcal P}(\kappa\cup\Gamma).
$$
Applying $u$ and evaluating at $a$ gives
$\beta\in\mathrm{Hull}^{\mathcal M'}(\mu\cup p\cup\Gamma)$. Hence
$$
\mathcal M'\upharpoonright\tau
  =\mathrm{cHull}^{\mathcal N}(\mu\cup j(p)\cup\Gamma)\upharpoonright\tau
  =\mathcal S_{XY}\upharpoonright\tau.
$$
\end{proof}

We illustrate the necessary correction to the proof of \cite[Lemma~4.3]{MS2}.
In the proof for Mitchell--Steel indexing, we replace the condition beginning
``If $\mathcal S_{XY}\lhd\mathcal M_Y$, then assume that'' by a condition
on the extender used to form $\mathcal S_{XY}$.
Suppose that we are in the second alternative of
Lemma~\ref{LA-2}. Fix $X,Y\in\mathcal H$ with $X\subseteq Y$ and write
$$
  \mathcal P=\mathcal P^Y_\lambda,\qquad
  \mathcal S=\mathcal S_{XY}=\mathrm{Ult}_n(\mathcal P,E),\qquad \mu=\mu_Y.
$$
Put $\kappa=\operatorname{crit}(E)$ and take $\eta\in I_Y$ large enough that
$$
  E\in H=\mathrm{Hull}^{\mathcal M_Y}_{m+1}
     (\mu^Y_\eta\cup p_{m+1}(\mathcal M_Y)),\qquad
  (\kappa^+)^{\mathcal M_Y}<\mu^Y_\eta.
$$
This replaces the condition $\mathcal S\in H$. Since $\nu(E)=\mu$ by
the proof of Lemma~\ref{LA-2}, we may also take $\eta$ large enough to write
$$
  p_{n+1}(\mathcal S)=[c,h]_E,\qquad c\in[\mu^Y_\eta]^{<\omega}.
$$
For such $\eta$ with $\pi_Y(\mu^Y_\eta)\in B_Y$, we have $H\cap\mu=\mu^Y_\eta$.
If $\pi_Y(\mu^Y_\eta)\notin C_X$, the calculation in the proof of \cite[Lemma~4.3]{MS2} gives
$$
  \mu^Y_\eta\ne\mu\cap\mathrm{Hull}^{\mathcal S}_{n+1}
     (\mu^Y_\eta\cup p_{n+1}(\mathcal S)).
$$
Choose $\beta\in[\mu^Y_\eta,\mu)$ in this hull. The range of the factor map
from $\mathrm{Ult}_n(\mathcal P,E\upharpoonright\mu^Y_\eta)$ to $\mathcal S$
contains $\mu^Y_\eta\cup\{p_{n+1}(\mathcal S)\}$ and is closed under the
$r\Sigma_{n+1}$ Skolem functions. Thus we may write
$$
  \beta=[b,f]_E,\qquad b\in[\mu^Y_\eta]^{<\omega}.
$$
Since $\beta<\mu\leq i_E(\kappa)$, $f$ takes values below $\kappa$ on
a set in $E_b$. Its graph restricted to these values belongs to $\mathcal P$
since $\kappa<\rho_n(\mathcal P)$. Filling in the remaining values with zero gives
$$
  \beta=i_E(g)(b),\qquad g\colon[\kappa]^{|b|}\longrightarrow\kappa.
$$
The graph of $g$ is coded in the common initial segment of $\mathcal P$
and $\mathcal M_Y$ below $(\kappa^+)^{\mathcal M_Y}<\mu^Y_\eta$, so $g\in H$.
Hence
$$
  \mu^Y_\eta\leq\beta=i_E(g)(b)\in H\cap\mu=\mu^Y_\eta,
$$
a contradiction.

When $X\in Y$, the subsets of $\mu_Y$ definable from the extender $G$ of Lemma~\ref{LA-3} can be recovered in
$\mathcal M_Y$ from a uniform code.
The following corollary explains it, which was used in Chapter~\ref{Ch5} to
show that $B\in\mathrm{ran}(i_k^X)$ in Case~3 of the proof of Lemma~\ref{L5-6}.

\begin{corollary}\label{CA-4}
Suppose $X\in Y$ and $\mathcal Q_X$ is a weasel. Put $\mu=\mu_Y$,
$t_X=c(\mathcal Q_X)\setminus\mu_X$ and
$t_{XY}=\Pi_{XY}(t_X)=c(\mathcal S_{XY})\setminus\mu$.
Let $h_X\colon W\to\mathcal Q_X$ be the map preceding $\Pi_X$ in
Lemmas~\ref{L2-5} and~\ref{L2-6}, with $\alpha=\mu_X$, and put
$j_{XY}=\Pi_{XY}\circ h_X$, where
$\Pi_{XY}\colon\mathcal Q_X\to\mathcal S_{XY}$.
There are $\rho<\mu$ and $R\in\mathcal P(\mu)\cap W_Y$
such that for each Skolem term $\sigma$ and each finite tuple $b$ of
ordinals in $W$, some $e<\rho$ satisfies
$$
\sigma^{\mathcal S_{XY}}(a,t_{XY},j_{XY}(b))\cap\mu
  =\{\beta<\mu:R(e,a,\beta)\}
$$
for every tuple $a$ from $\mu$ of the appropriate length. Here $R(e,a,\beta)$
means that $R$ contains the ordinal coding $(e,a,\beta)$.
\end{corollary}

\begin{proof}
Put $j_X=\Pi_X\circ h_X$, $\gamma=\mathrm{crit}(j_X)<\delta_X$ and $d=\Pi_X(t_X)$.
By Lemmas~\ref{L2-5} and~\ref{L2-6}, the extender derived from $j_X$ on
$\nu\cup d$ belongs to $W$. Enumerate the following sets in the canonical order of $W$.
$$
\{D_\eta:\eta<(\gamma^+)^W\}
  =\bigcup_{n<\omega}\bigl(\mathcal P(\gamma^n)\cap W\bigr).
$$
Using this extender, define
$$
T(\eta,a,\beta)\quad\Longleftrightarrow\quad
  (a,d,\beta)\in j_X(D_\eta)
\qquad(a\in\nu^{<\omega},\ \beta<\nu).
$$
This gives a code for $T$ in
$\mathcal P(\nu)\cap W$ definable from $X,W\in Y$. Its collapse $R$
therefore belongs to $W_Y$. The coiteration of $W$
with $W_Y$ gives $\mathcal P(\mu)\cap W_Y\subseteq\mathcal M_Y$.
Put $\rho=(\gamma^+)^{W_Y}<\mu$.

Fix $\sigma,b$. In $W$ let $D$ consist of the tuples $(x,y,z)$ from
$\gamma$ such that $z\in\sigma^W(x,y,b)$, with $y$ of length $|d|$.
If $n$ is the arity of $D$, then $D=h_X(D)\cap\gamma^n$.
Thus $D$ belongs to the common initial segment of $\mathcal Q_X$ and
$W_X$ below $\mu_X$. Since $N_X\subseteq Y$, the least $\eta$ with
$D=D_\eta$ belongs to $Y$. Take $e=\pi_Y^{-1}(\eta)<\rho$.
For $\beta<\mu$, applying the factor map
$\Phi_{XY}\colon\mathcal S_{XY}\to\mathcal S_X$ gives
$$
\beta\in\sigma^{\mathcal S_{XY}}(a,t_{XY},j_{XY}(b))
\quad\Longleftrightarrow\quad
(\pi_Y(a),d,\pi_Y(\beta))\in j_X(D),
$$
since $\Phi_{XY}$ agrees with $\pi_Y$ below $\mu$, sends $t_{XY}$ to $d$,
and satisfies $\Phi_{XY}\circ j_{XY}=j_X$. The right side is precisely $R(e,a,\beta)$.
\end{proof}

\end{document}